\documentclass{amsart}
\usepackage{dsfont, amssymb, mathtools, IEEEtrantools, subcaption}
\usepackage{calc,changepage}
\usepackage{csquotes}
\usepackage{hyperref}
\usepackage{enumitem}
\usepackage{booktabs}
\usepackage{diagbox}

\usepackage{xcolor}
\definecolor{wb}{RGB}{51,153,255}
\usepackage{tikz}
\usetikzlibrary{calc}
\usepackage{tikz-cd}

\usepackage[parfill]{parskip}
\numberwithin{equation}{subsection}
\usepackage{amsrefs}

\newcommand{\defeq}{\vcentcolon=}
\newcommand{\eqdef}{=\vcentcolon}

\makeatletter
\def\moverlay{\mathpalette\mov@rlay}
\def\mov@rlay#1#2{\leavevmode\vtop{%
   \baselineskip\z@skip \lineskiplimit-\maxdimen
   \ialign{\hfil$\m@th#1##$\hfil\cr#2\crcr}}}
\newcommand{\charfusion}[3][\mathord]{
    #1{\ifx#1\mathop\vphantom{#2}\fi
        \mathpalette\mov@rlay{#2\cr#3}
      }
    \ifx#1\mathop\expandafter\displaylimits\fi}
\makeatother

\newcommand{\cupdot}{\charfusion[\mathbin]{\cup}{\cdot}}
\newcommand{\bigcupdot}{\charfusion[\mathop]{\bigcup}{\cdot}}

\DeclareFontFamily{U}{mathb}{\hyphenchar\font45}
\DeclareFontShape{U}{mathb}{m}{n}{
      <5> <6> <7> <8> <9> <10> gen * mathb
      <10.95> mathb10 <12> <14.4> <17.28> <20.74> <24.88> mathb12
      }{}
\DeclareSymbolFont{mathb}{U}{mathb}{m}{n}
\DeclareMathSymbol{\precneq}{3}{mathb}{"AC}
\DeclareMathSymbol{\varprec}{3}{mathb}{"A0}

\newtheoremstyle{definitions}
 	{\topsep}
	{\topsep}
	{}
	{}
	{\bfseries}
	{:}
	{.5em}
	{}
  
\newtheoremstyle{lemmata}
	{\topsep}
	{\topsep}
	{\itshape} 
	{}
	{\bfseries}
	{:}
	{.5em}
	{}

\theoremstyle{lemmata}
\newtheorem{Theorem}[subsection]{Theorem}
\newtheorem{Lemma}[subsection]{Lemma}
\newtheorem{Corollary}[subsection]{Corollary}
\newtheorem{Proposition}[subsection]{Proposition}

\theoremstyle{definitions}
\newtheorem{Definition}[subsection]{Definition}

\newtheorem{Remarks}[subsection]{Remarks}
\newtheorem*{Remarks-nn}{Remarks}

\newtheorem{Example}[subsection]{Example}

\newtheorem{Observation}[subsection]{Observation}
\newtheorem{Algorithm}[subsection]{Algorithm}
\newtheorem{Examples}[subsection]{Examples}
\newtheorem{Concluding_Remarks}[subsection]{Concluding Remarks}

\DeclareMathOperator{\ord}{ord}

\title{Goss Polynomials, $q$-adic expansions, and Sheats compositions}
\author{Ernst-Ulrich Gekeler}
\date{\today}

\begin{document}

\begin{abstract}
	The zeroes of Goss polynomials $G_{k, \Lambda}(X)$ for $\Lambda = A \defeq \mathds{F}_{q}[T]$ and similar lattices $\Lambda$ are studied. Generically, the zero distribution follows a simple pattern
	governed by the $q$-adic expansion of $k-1$. However, if $q=p^{f}$ with $f \geq 2$ is a proper power of the prime $p$, irregularities of the $q$-adic sum-of-digits function may lead to deviations from this 
	pattern, to irregular zeroes, and an abundance of trivial zeroes of $G_{k, \Lambda}$ compared to the generic formula. These phenomena are related to properties of the Sheats compositions of natural numbers
	$n$ divisible by $q-1$. Among other things, we give a necessary and sufficient condition for the existence of irregular zeroes of $G_{k,A}$ and a formula for the vanishing order of $G_{k,A}$ at $x=0$.
\end{abstract}

\maketitle

\setcounter{section}{-1}

\section{Introduction} \label{Section.Introduction}

\subsection{} Let $q$ be a prime power, $\mathds{F} = \mathds{F}_{q}$ the field with $q$ elements, $A = \mathds{F}[T]$ its polynomial ring in the indeterminate $T$, and $C_{\infty}$ the completed algebraic closure of 
$K_{\infty} = \mathds{F}((T^{-1}))$ with respect to its absolute value, normalized through $\lvert T \rvert = q$. Associated with $A$ (or with other $\mathds{F}$-lattices $\Lambda$ in $C_{\infty}$), we consider the 
meromorphic function $C_{k} = C_{k,A}$ on $C_{\infty}$:
\begin{equation}
	z \longmapsto \sum_{\lambda \in A} \frac{1}{(z-\lambda)^{k}}.
\end{equation}
It is a basic fact (\cite{Goss80} Proposition 6.5) that there exists a sequence $(G_{k}(X))_{k \in \mathds{N}}$ of monic polynomials of degree $k$ with coefficients in $C_{\infty}$ such that 
\begin{equation}
	C_{k}(z) = G_{k}(t(z)),
\end{equation}
where
\begin{equation}
	t(z) = t_{A}(z) \defeq C_{1}(z) = \sum_{\lambda \in A} \frac{1}{z-\lambda}.
\end{equation}
Upon a suitable normalization (replace $A$ with the lattice $\Lambda = \bar{\pi}A$, where $\bar{\pi}$ is the characteristic-$p$ counterpart of $2\pi \mathrm{i}$, $\pi = 3,\,14\dots$, see \cite{Goss80} Corollary 6.7 or 
\cite{Gekeler88} Section 4), the \textbf{Goss polynomials} $G_{k, \Lambda}$ acquire coefficients in $K = \mathds{F}(T)$. Hence they are analogues of the classical Euler polynomials; see the discussion in \cite{Gekeler13} Section 4.

The $G_{k, \Lambda}$ and their generalizations (allow $\Lambda$ to be an arbitrary separable finite or infinite $\mathds{F}$-lattice in $C_{\infty}$) have numerous applications in a multitude of arithmetical 
questions, among which the following:
\begin{itemize}
	\item expansions at infinity of Drinfeld modular forms, in particular Eisenstein series \cite{Goss80-2}, \cite{Gekeler88}, \cite{Gekeler12}, $v$-adic congruences \cite{Goss14};
	\item Drinfeld quasi-modular forms \cite{BosserPellarin09}, \cite{Pellarin14};
	\item power sums and values at negative integers of the Carlitz-Goss zeta function defined on $s \in \mathds{N}$ by 
	\[
		\zeta_{\mathrm{Goss}}(s) = \sum_{a \in A \text{ monic}} a^{-s}
	\]
	and analytically continued \cite{Goss83}, \cite{Goss96}, \cite{Sheats98}, \cite{Thakur15};
	\item evaluation of Hecke operators \cite{Gekeler88}, \cite{PapanikolasZeng17}.
\end{itemize}

\subsection{} In the special case where $q=p$ is prime, a quite satisfactory description of the $G_{k,A}$ and their relevant properties has been obtained in \cite{Gekeler13-2}. Here all the non-trivial (different from 0) zeroes of $G_{k,A}$
are of integer type, that is, all the zeroes of $C_{k,A}$ satisfy $\lvert z \rvert = q^{i}$ for some $i \in \mathds{N}_{0}$, and there exists a simple formula for the \textbf{vanishing number}
\begin{equation}
	\gamma_{A}(k) \defeq \text{multiplicity of $0$ as a zero of $G_{k,A}$}
\end{equation}
(see \cite{Gekeler13-2} Theorem 6.12). These results remain valid if one replaces $A$ with an arbitrary finite or infinite $\mathds{F}$-lattice $\Lambda$, provided it is \textbf{separable}, which means it has an 
$\mathds{F}$-basis $\{ \lambda_{0}, \lambda_{1}, \dots \}$ with $\lvert \lambda_{0} \rvert < \lvert \lambda_{1} \rvert < \dots$ In this case, the abscissas of the Newton polygon $\mathrm{NP}(G_{k,\Lambda})$ of
$G_{k, \Lambda}$ depend only on the $q$-adic expansion of $k-1$, while its ordinates vary with the $\lvert \lambda_{i} \rvert$, see \cite{Gekeler13}.
\subsection{} The situation is much more complex for $q = p^{f}$ non-prime, and certainly more complex than the author expected in \cite{Gekeler13-2}. By non-archimedean techniques like contour integration
(introduced in \cite{GerritzenvdPut80}), the location of the zeroes of $C_{k, \Lambda}$ (hence those of $G_{k, \Lambda}$) may be restricted to certain critical spheres $\mathbf{S}(q^{r}) \subset C_{\infty}$. Their radii $q^{r}$
are related to the power sums of elements of $\Lambda$
\begin{equation}
	S_{i, \Lambda}(n) \defeq \sum_{\substack{\lambda \in \Lambda \\ \lvert \lambda \rvert < \lvert \lambda_{i} \rvert}} \lambda^{n},
\end{equation}
whose evaluation involves control of the binomial coefficients $\binom{m_{i}}{m_{i-1}} \pmod{p}$ along sequences 
\begin{equation} \label{Eq.Prototype-of-sequences-discussed}
	0 = m_{0} \lneq m_{1} \lneq \dots \lneq m_{h} = n.
\end{equation}
Here the $m_{i}$ (thus $n$, too) are supposed to be divisible by $q-1$, and the \textbf{height} $\mathrm{ht}(n)$ of $n$ ($q$ being fixed) is the maximal length of such a sequence where all the $\binom{m_{i}}{m_{i-1}}$ don't 
vanish $\pmod{p}$. Note that by Lucas, the non-vanishing $\pmod{p}$ of $\binom{b}{a}$ is equivalent to the fact that there is no carryover of digits in the $p$-adic expansion of $b = a+(b-a)$. In this case we write
\begin{equation}
	a <_{p} b \quad \text{(or $a \lneq_{p} b$ if $a <_{p} b$ and $a \neq b$)}.
\end{equation}
\subsection{} Jeffrey Sheats in \cite{Sheats98} showed that there is a distinguished sequence of shape \eqref{Eq.Prototype-of-sequences-discussed} (we call it the \textbf{Sheats sequence} of $n$) characterized as follows.
For $1 \leq i \leq h = \mathrm{ht}(n)$ put $X_{i} \defeq m_{i} - m_{i-1}$. Then
\begin{multline}
	0 \neq X_{i} \pmod {q-1} \text{ for all $i$}, \qquad \sum X_{i} = n, \\ \text{and there is no carryover in the $p$-adic expansions},
\end{multline}
which by definition means that $\mathbf{X} \defeq (X_{1}, \dots, X_{h})$ is a \textbf{composition of length $h$ of $n$}. Define the \textbf{weight} $\mathrm{wt}(\mathbf{X})$ as 
\begin{equation}
	\mathrm{wt}(\mathbf{X}) \defeq \sum i X_{i}.
\end{equation}
The \textbf{Sheats composition} $\mathbf{X} = \mathbf{Sh}(n)$ is the unique composition of $n$ of length $h$ with maximal weight; its weight is also called the \textbf{weight} $\mathrm{wt}(n)$ \textbf{of $n$}. The \textbf{truncated
Sheats composition} $\mathbf{Sh}^{(i)}(n)$ of $n$ of length $i \leq h = \mathrm{ht}(n)$ is 
\begin{equation}
	\mathbf{Sh}^{(i)}(n) = (X_{1}, \dots, X_{i-1}, X_{i} + X_{i+1} + \dots + X_{h}).
\end{equation} 
The uniqueness of this $\mathbf{X} = \mathbf{Sh}(n)$ is a highly non-trivial fact, as is the assertion that the transpose $\mathbf{X}^{t} = (X_{h}, X_{h-1}, \dots, X_{1})$ of $\mathbf{X} = \mathbf{Sh}(n)$ is maximal with respect to
the lexicographic order on the set of all compositions of length $h$ (\cite{Sheats98} Theorem 1.2). His theorem enabled Sheats to show, among other things, that
\begin{equation} \label{Eq.Relation-Sheats-factor-and-weight}
	\deg S_{i,A}(n) = \mathrm{wt}(\mathbf{Sh}^{(i)}(n)) - n
\end{equation}
whenever $S_{i,A}(n) = \sum_{a \in A, \deg a <i } a^{n}$ doesn't vanish, which is the case if and only if $n \equiv 0 \pmod{q-1}$ and $\mathrm{ht}(n) \geq i$. This result has been obtained (should we say conjectured?) by Carlitz in \cite{Carlitz48}, but
with an invalid proof, see the discussion in \cite{Sheats98} Section 1. It will also be crucial for our study of $G_{k,A}$.

\subsection{} As meaningful as it is, Sheats' theorem leaves important questions open. First, it is a sheer existence statement, based on a proof by contradiction, and thus not constructive. For our purposes however, we need 
explicit descriptions and estimates on the terms of the Sheats composition. Second, it doesn't provide qualitative properties of $\mathbf{Sh}(n) = (X_{1}, \dots, X_{h})$; for example, is $X_{1}$ minimal with
$X_{1} \equiv 0 \pmod{q-1}$, $\mathrm{ht}(X_{1}) = 1$ and $\binom{n}{X_{1}} \equiv 0 \pmod{q-1}$, etc.? Third, while the maximal weight property of $\mathbf{Sh}(n)$ explicitly refers to the weights $\mathrm{wt}(X_{i}) = i$ given
by $\log_{q} \lvert T^{i} \rvert = i$ of the basis $\{ T^{i} \mid i \geq 0\}$ of $A$, the lexicographic characterization doesn't. Are there similar results for other weight systems like $r_{i} \defeq \log_{q} \lvert \lambda_{i} \rvert$ that
apply to arbitrary separable lattices with separating basis $\{ \lambda_{i} \}$?

\subsection{} We will offer satisfactory answers to these questions. Among other things, we will develop an algorithm (see \ref{Algorithm.To-determine-Sheats-factor} and Theorem \ref{Theorem.On-correctness-of-Alg}) that produces the largest Sheats factor $\mathrm{Sh}(h) =X_{h}$ of $n$, and thus
the complete Sheats composition. It allows to derive qualitative and extremal properties of $\mathbf{Sh}(n)$, among which the important inequality
\begin{equation}
	X_{i+1} \geq q X_{i} \qquad (1 \leq i < h(n))
\end{equation}
for the Sheats factors (Theorem \ref{Theorem.Inequalitites-for-Sheats-factors}). Together with the rigid-analytic arguments mentioned earlier, these suffice to specify under which circumstances the zero distribution of $G_{k, \Lambda}$ is as in the case of a prime $q$ 
(Theorem \ref{Theorem.Regularity-and-zeroes-of-Ck-Lambda}), and to describe the possible deviations. A crucial role is played by the sequence $(\mu_{i}(k))_{i \in \mathds{N}_{0}}$, which determines the overall shape of the Newton polygon $\mathrm{NP}(G_{k,A})$. Here 
the \textbf{approximation number} $\mu_{i}(k)$ is defined by
\begin{equation}
	\mu_{i}(k) = \min \left\{ n \in (q-1)\mathds{N}_{0} : \binom{n+k-1}{n} \not\equiv 0 \pmod{p} \text{ and } \mathrm{ht}(n) \geq i \right\}.
\end{equation}
If $\mu_{i}(k) <_{p} \mu_{i+1}(k)$ for all $i$ (this inequality anyway holds if $i$ is large enough compared to $k$), then all the zeroes of $C_{k, \Lambda}$ (or of $G_{k, \Lambda}$) for arbitrary separable $\mathds{F}$-lattices
$\Lambda$ are of integer type (Theorem \ref{Theorem.Regularity-and-zeroes-of-Ck-Lambda}); in this case we call $k$ \textbf{regular} (with respect to the given $q$). The regularity condition is fulfilled for all $k$ if $q$ is prime. 
In general, regularity of $k$ is some kind of balancedness property of the $p$-adic expansion coefficients of $k-1$. Empirically (by some numerical experiments) and by heuristic reasoning, \enquote{most} $k$ should be 
regular for a given $q$; it would be an interesting challenge for analytical number theorists to give a precise mathematical meaning to the \enquote{most}.

If $k$ fails to be regular, Theorem \ref{Theorem.Irregular-zeroes-Ck} describes the possible irregular zeroes of $C_{k,A}$ and $G_{k,A}$; but note that this holds for the lattice $A$ only. The case of general lattices $\Lambda$ 
in conjunction with the irregularity of $k$ appears too complex to integrate in one description scheme. Finally, from the equation
\begin{equation}
	\gamma_{A}(k) = \lim_{i \to \infty} \frac{k + \mu_{i}(k)}{q^{i}},
\end{equation}
which comes from contour integration, we derive an explicit formula for $\gamma_{A}(k)$ in Theorem \ref{Theorem.Vanishing-orders-of-Goss-polynomials-in-Cases}. (Recall that the vanishing number $\gamma_{A}(k)$ is the 
multiplicity of 0 as a zero of $G_{k,A}$; it plays a role for example in the expansion of Drinfeld-Eisenstein series at cusps \cite{Gekeler12}.)

Let $R$ be the representative $\pmod{q-1}$ of $k-1$ in $\{0, 1, \dots, q-2\}$, $\ell(k-1)$ the sum of $q$-adic digits of $k-1$ and $[\cdot]$ the Gauß bracket. A coarse version of Theorem 
\ref{Theorem.Vanishing-orders-of-Goss-polynomials-in-Cases} reads as follows: In most cases (specified in the theorem)
\begin{equation} \label{Eq.Equality-for-gamma-A-k-in-most-cases}
	\gamma_{A}(k) = (R+1)q^{[\ell(k-1)/(q-1)]};
\end{equation}
in the other cases (here another irregularity phenomenon takes place), a slightly more complicated but nevertheless explicit formula holds, but the right hand side of \eqref{Eq.Equality-for-gamma-A-k-in-most-cases} is always a lower
bound for $\gamma_{A}(k)$.

\subsection{} Let us briefly outline the plan of the paper. Sections 1 and \ref{Section.Power-sums-and-Sheats-compositions} present known introductory material and Sheats' results tailored to our purposes. In Section \ref{Section.p-adic-and-q-adic-expansions} we introduce and study the action of the cyclic group
$\mathbf{C} = \mathds{Z}/(f)$ on $p$-adic and $q$-adic expansions. It allows to formulate an expedient criterion (Theorem 3.9) for the height $\mathrm{ht}(n)$, and is our basic tool to handle the complications caused by
non-prime $q$. In Section \ref{Properties-of-Sheats-compositions} the algorithm $\mathbf{Alg}$ for determining the Sheats composition is given. It is sufficiently explicit to derive the properties needed for our study of the Goss polynomials $G_{k}$. In particular,
we find in Section \ref{Section.Location-of-zeroes-regular-case} that the $G_{k}$ behave as in the ($q=p$) case whenever $k$ is regular. Before treating irregular zeroes in Section \ref{Section.Irregular-zeroes}, Section \ref{Section.Auxiliary-results} investigates the 
approximation numbers $\mu_{i}(k)$, and gives some estimates.
We are able to determine $\gamma_{A}(k)$ in Section \ref{Section.Determination-of-vanishing-numbers}. While the normal behavior of $\gamma_{A}(k)$ results without difficulty from the preceding, we must carefully 
analyze the effect of $\mathbf{Alg}$ on a sequence of special numbers in the exceptional cases. We collect the accumulated insight about the functions $C_{k, \Lambda}$ and $G_{k,\Lambda}$ in our main theorems 
Theorem \ref{Theorem.Regularity-and-zeroes-of-Ck-Lambda}, Theorem \ref{Theorem.Irregular-zeroes-Ck}, and Theorem \ref{Theorem.Vanishing-orders-of-Goss-polynomials-in-Cases}.

In the concluding Section we describe how the results about $C_{k, \Lambda}$ and $G_{k, \Lambda}$ must be adapted to apply to finite separable lattices $\Lambda$. Last not least we present examples for irregular behavior of
the quantities in question. Focussing on the least prime power $q=4$, the least example of exceptional behavior of $\gamma_{A}(k)$ is $k=21$, where $\gamma_{A}(k) = 6$ instead of the value 3 expected from 
\eqref{Eq.Equality-for-gamma-A-k-in-most-cases}. The least example of irregular $k$ is $k=75$; here $C_{k, A}$ has zeroes $z$ with $\log_{q} \lvert z \rvert = 13/12$ non-integral.

All our results are valid for $q = p^{f}$ with $f \geq 1$, but are either trivial or established in \cite{Gekeler13} and \cite{Gekeler13-2} if $f=1$. Hence the focus is on the case $f \geq 2$ where $q$ is non-prime.

\subsection*{Notation} Throughout, $p$ is a fixed prime number with a power $q = p^{f}$, $\mathds{F} = \mathds{F}_{q}$ the field with $q$ elements, $A = \mathds{F}[T]$ the polynomial ring over $\mathds{F}$, 
$K = \mathds{F}(T)$ its field of fractions, and $K_{\infty} = \mathds{F}((T^{-1}))$ is the completion of $K$ at its infinite place.

We let $\lvert \cdot \rvert$ be the absolute value on $K_{\infty}$, normalized by $\lvert T \rvert = q$, and $C_{\infty} = \widehat{\bar{K}}_{\infty}$ the completed algebraic closure of $K_{\infty}$.

The ring of integers $\{z \in C_{\infty} \mid \lvert z \rvert \leq 1 \}$ and its maximal ideal $\{ z \in C_{\infty} \mid \lvert z \rvert < 1 \}$ are denoted by $O_{C_{\infty}}$ and $\mathfrak{m}_{C_{\infty}}$, respectively, and 
$\log \colon C_{\infty}^{*} \to \mathds{Q}$ is the map $z \mapsto \log_{q} \lvert z \rvert$, extended by $\log 0 = -\infty$. 

By an $\mathds{F}$-\textbf{lattice} in $C_{\infty}$ we mean a (finite or infinite) $\mathds{F}$-subspace $\Lambda$
of $C_{\infty}$ which has finite intersection with each ball $\mathbf{B}(q^{r}) \defeq \{ z \in C_{\infty} \mid \log z \leq r \}$ of radius $q^{r}$, where $r \in \mathds{Q}$. $\mathds{N} = \{1,2,\dots\}$ and $\mathds{N}_{0} = \{0,1,2,\dots \}$
denote the usual sets of natural numbers. We further use self-explaining constructs like $\mathds{Q}_{\geq 0} \defeq \{x \in \mathds{Q} \mid x \geq 0 \}$, etc. Further, $F \defeq \{0,1,\dots, f-1\}$, $P \defeq \{0,1,\dots,p-1\}$,
$Q \defeq \{0,1, \dots, q-1\}$, $Q' \defeq \{0,1,\dots,q-2\}$ are sets of representatives for $\mathds{Z}/(n)$, where $n=f,p,q,q-1,$ respectively.

\section{Non-archimedean Prerequisites} \label{Section.Non-archimedean-prerequisites}

\subsection{} In what follows, we fix an $\mathds{F}$-lattice $\Lambda$ in $C_{\infty}$. For $k \in \mathds{N}$, define the meromorphic function
\begin{equation} \label{Eq.C-k-lambda-meromorphic-function}
	C_{k,\Lambda}(z) \defeq \sum_{\lambda \in \Lambda} \frac{1}{(z-\lambda)^{k}}.
\end{equation}\stepcounter{subsubsection}%
The sum converges locally uniformly on $C_{\infty}$ (in the sense of meromorphic functions) and yields a $\Lambda$-periodic meromorphic function on $C_{\infty}$ with the obvious poles. The \textbf{exponential function of $\Lambda$}
is
\begin{equation} \label{Eq.Exponential-function-of-lattice-Lambda}
	e_{\Lambda}(z) \defeq z \prod_{0 \neq \lambda \in \Lambda} (1- z/\lambda).
\end{equation}\stepcounter{subsubsection}%
It is an entire, surjective, $\Lambda$-periodic function on $C_{\infty}$ with a series expansion
\begin{equation} \label{Eq.Series-expansion-of-exponential-function}
	e_{\Lambda}(z) = \sum_{i \geq 0} \alpha_{i} z^{q^{i}},
\end{equation}\stepcounter{subsubsection}%
where $\alpha_{i} = \alpha_{i}(\Lambda) \in C_{\infty}$ and $\alpha_{0} = 1$. The function
\begin{equation} \label{Eq.Inverse-Exponential}
	t_{\Lambda}(z) \defeq C_{1, \Lambda}(z) = \sum_{ \lambda \in \Lambda} \frac{1}{z- \lambda}
\end{equation}\stepcounter{subsubsection}%
in fact equals $e_{\Lambda}(z)^{-1}$. (For all these properties, consult \cite{Gekeler88}, \cite{Goss96}, or \cite{Thakur04}.) As is well-known and easy to show (see the cited references), there exists a sequence 
$(G_{k,\Lambda}(X))_{k \in \mathds{N}}$ of polynomials with coefficients in the closure of the field $\mathds{F}(\Lambda)$ generated over $\mathds{F}$ by $\Lambda$ such that 
\begin{equation} \label{Eq.Characterization-of-Goss-polynomials}
	C_{k, \Lambda}(z) = G_{k, \Lambda}(t_{\Lambda}(z)).
\end{equation}\stepcounter{subsubsection}%
These polynomials are the \textbf{Goss polynomials} of $\Lambda$ (see \cite{Goss80} Proposition 6.5); besides others, they enjoy the following properties (\cite{Gekeler88}; we omit the subscripts $\Lambda$ as long as
$\Lambda$ is fixed):
\subsubsection{}\label{Subsubsection.Properties-Goss-polynomials} \begin{enumerate}[label=(\roman*)]
	\item $G_{k}$ is monic of degree $k$, with $G_{k}(0) = 0$;
	\item $G_{pk} = (G_{k})^{p}$;
	\item $G_{k}(X) = X^{k}$ for $k \leq q$;
	\item $G_{k}(X) = X(G_{k-1} + \alpha_{1} G_{k-q} + \alpha_{2}G_{k-q^{2}} + \dots)$, where $G_{k}(X) = 0$ for $k \leq 0$ and the $\alpha_{i} = \alpha_{i}(\Lambda)$ are the coefficients of \eqref{Eq.Series-expansion-of-exponential-function}
\end{enumerate}

\subsection{} Throughout, we will assume that $\Lambda$ is \textbf{separable}, that is, it possesses an $\mathds{F}$-basis $\{ \lambda_{0}, \lambda_{1}, \dots, \}$ (a \textbf{separating basis}) such that $\lvert \lambda_{0} \rvert < \lvert \lambda_{1} \rvert < \dots$ The \textbf{critical radii} $q^{r_{i}} \defeq \lvert \lambda_{i} \rvert$ of $\Lambda$ are invariants of $\Lambda$, that is, independent of choices made. We let
\begin{equation}
	\Lambda_{i} \defeq \sum_{0 \leq j < i} \mathds{F} \lambda_{j}
\end{equation}
and
\begin{equation}
	\mathbf{S}(\lvert \lambda_{i} \rvert) \defeq \{ z \in C_{\infty} \mid \lvert z \rvert = \lvert \lambda_{i} \rvert \}
\end{equation}
be the \textbf{$i$-th critical sphere} of $\Lambda$. The principal example for $\Lambda$ is of course $\Lambda = A = \mathds{F}[T]$, with separating basis $\{ T^{i} \mid i \in \mathds{N}_{0} \}$. Here the nature of the zeroes of
$G_{k}$ (their sizes, the Newton polygon $\mathrm{NP}(G_{k})$, the multiplicity $\gamma(k)$ by which $x=0$ occurs as a zero, \dots) has important arithmetic implications. Some properties of $G_{k}$ may be translated via 
\eqref{Eq.Characterization-of-Goss-polynomials} to properties of the function $C_{k}$; in fact, for technical reasons we will mainly work with the $C_{k}$ and then translate back to the Goss polynomials $G_{k}$.

\subsection{} \label{Subsection.Functions-under-Scaling} All the functions introduced so far show a simple behavior under scaling $\Lambda \leadsto c \Lambda$ the lattice $\Lambda$ by $c \in C_{\infty}^{*}$, namely
\begin{align}
	e_{c\Lambda}(cz) 			&= c e_{\Lambda}(z) \\
	\alpha_{k}(c\Lambda)		&= c^{1-q^{k}} \alpha_{k}(\Lambda) \\
	C_{k, c\Lambda}(cz)		&= c^{-k} C_{k, \Lambda}(z) \\
	G_{k, c\Lambda}(c^{-1}X)	&= c^{-k} G_{k,\Lambda}(X).
\end{align}
This entitles us to make without restriction of generality the 
\subsection{Normalizing assumption} \label{Subsection.Normalizing-assumption} $\lambda_{0} = 1$, so $r_{0} = \log \lambda_{0} = 0$.

\textbf{Remark:} For arithmetical purposes it might be useful to apply a different normalization. For example, replacing $A$ with $\Lambda \defeq \bar{\pi}A$, where $\bar{\pi}$ is the characteristic-$p$ substitute of $2 \pi \mathrm{i}$ 
(see \cite{Gekeler88} Section 4), the Drinfeld module associated with $\Lambda$ will be the Carlitz module. As a consequence, the coefficients $\alpha_{i}(\Lambda)$ and therefore those of $G_{k, \Lambda}$ (see \ref{Subsubsection.Properties-Goss-polynomials} (iv)) are in $K$, with arithmetic relevance. Using the formulas of \ref{Subsection.Functions-under-Scaling}, there will be no difficulty to translate back and forth results between
the different normalizations.

\subsection{} We will make heavy use of the non-archimedean residue theorem and the formalism of contour integration as has been introduced in \cite{GerritzenvdPut80}. Let 
\[
	\mathbf{B} = \mathbf{B}(q^{r}) = \{ z \in C_{\infty} \mid \log z \leq r \}
\]
be a \enquote{closed} ball, where $r \in \mathds{Q}$, with corresponding \enquote{open ball} $\mathbf{B}^{-} = \{ z \mid \log z < r\}$ and boundary $\partial \mathbf{B} = \mathbf{S}(q^{r}) = \mathbf{B} \smallsetminus \mathbf{B}^{-}$.
Then $\partial \mathbf{B}$ is an analytic subspace of $\mathbf{B}$. Let $w$ be a coordinate function on $\partial \mathbf{B}$ with absolute value 1, for example $w = z/w_{0}$, where $w_{0} \in C_{\infty}$ is fixed with
$\log w_{0} = r$. Each invertible function $f$ on $\partial \mathbf{B}$ has a Laurent expansion
\begin{equation} \label{Eq.Laurent-expansion-on-Border-of-sphere}
	f(z) = w^{m} \sum_{n \in \mathds{Z}} a_{n}w^{n}
\end{equation}
with $a_{n} \in C_{\infty}$, $\lvert a_{0} \rvert > \max_{n \neq 0} \lvert a_{n} \rvert$ and some 
\begin{equation}
	m \eqdef \ord_{\partial \mathbf{B}}(f) \in \mathds{Z},
\end{equation}
the \textbf{order of $f$ along $\partial \mathbf{B}$}. Let $f$ be meromorphic on $\mathbf{B}$ without zeroes or poles on $\partial \mathbf{B}$. Then the residue formula
\begin{equation} \label{Eq.Residue-formula}
	\sum_{x \in \mathbf{B}} \ord_{x}(f) = \ord_{\partial \mathbf{B}}(f)
\end{equation}
holds, where $\ord_{x}(f)$ is the vanishing order (negative if $f$ has a pole at $x$) of $f$ at $x$ (\cite{GerritzenvdPut80} pp. 93-95).

\subsection{}\label{Subsection.Laurent-expansion-of-C-k-Lambda} A standard calculation, performed in \cite{Gekeler13-2} Section 4 in the case $\Lambda = A$ (which carries over to the general case) yields the Laurent expansion of $C_{k} = C_{k, \Lambda}$ on non-critical spheres.

Let $\mathbf{S} = \mathbf{S}(q^{r})$, where $r_{i} < r < r_{i+1}$, with coordinate $w = z/w_{0}$ as described above. Then the Laurent expansion of $C_{k}$ on $\mathbf{S}$ is
\begin{equation} \label{Eq.Laurent-expansion-of-Ck-on-Sphere}
	C_{k}(z) = \sum_{n \in \mathds{Z}} a_{n}w^{n},
\end{equation}
where
\begin{align*}
	a_{n}		&= (-1)^{k} \binom{k-1+n}{n} w_{0}^{n} \sum_{ \lambda \in \Lambda \smallsetminus \Lambda_{i+1}} \lambda^{-k-n}	&&(n \geq 0) \\
	a_{n}		&= 0																																																	&&(-k \leq n < 0) \\
	a_{-k-n} 	&= \binom{k-1+n}{n} w_{0}^{-k-n} \sum_{\lambda \in \Lambda_{i+1}} \lambda^{n}															&&(n > 0).
\end{align*}
Here the binomial coefficients must of course be evaluated in $\mathds{F} \hookrightarrow C_{\infty}$. In many cases we will be able to isolate dominant terms as in \eqref{Eq.Laurent-expansion-on-Border-of-sphere} for
\eqref{Eq.Laurent-expansion-of-Ck-on-Sphere} and thereby show that $C_{k}$ has no zeroes between the critical spheres $\mathbf{S}(q^{r_{i}})$ and $\mathbf{S}(q^{r_{i+1}})$. Later (Corollary 
\ref{Corollary.Coefficient-condition-Laurent-expansion-of-Ck}) we will see that the $a_{n}$ with $n \geq 0$ are irrelevant for our purposes. As for the others, we define 
\begin{equation}
	S_{i, \Lambda}(n) \defeq \sum_{ \lambda \in \Lambda_{i}} \lambda^{n} = \sum_{\lambda \in \Lambda, \log \lambda < r_{i}} \lambda^{n}.
\end{equation} 
Then the last formula in \eqref{Eq.Laurent-expansion-of-Ck-on-Sphere} becomes
\begin{equation}
	a_{-k-n} = \binom{k-1+n}{n} w_{0}^{-k-n} S_{i+1, \Lambda}(n).
\end{equation}
A first observation about zeroes of $C_{k}$ is

\begin{Proposition}
	$C_{k} = C_{k, \Lambda}$ has no zeroes $z$ with $0 < \lvert z \rvert < 1$.
\end{Proposition}

\begin{proof}
	Suppose $0 < \lvert z \rvert < 1$. Then $C_{k}(z) = z^{-k} +$ smaller terms, hence it cannot be a zero.
\end{proof}

\subsection{} The relationship between zeroes of $C_{k}$ and $G_{k}$ comes out as described below. We follow \cite{Gekeler13} Sections 5 and 6, where also proofs for the next assertions are given. The general assertion $q=p$
in \cite{Gekeler13} is not needed. For $z \in C_{\infty}$ with $\lvert z \rvert \geq 1$, put 
\begin{equation}
	\lvert z \rvert_{\min} \defeq \inf_{\lambda \in \Lambda} \lvert z - \lambda \rvert = \min_{\lambda \in \Lambda} \lvert z - \lambda \rvert.
\end{equation}
Note that $\lvert z \rvert = \lvert z \rvert_{\min}$ if $\lvert z \rvert$ doesn't belong to the value set $\{ \lvert \lambda_{i} \rvert \mid i \in \mathds{N}_{0} \}$ of $\Lambda$. Define the analytic subspaces of $\mathds{A}^{1}/C_{\infty}$:
\begin{equation}
	\Omega_{\Lambda} \defeq \{ z \in C_{\infty} \mid \lvert z \rvert_{\min} \geq 1 \}, \quad \text{and} \quad \mathbf{F}_{\Lambda} \defeq \{ z \in C_{\infty} \mid \lvert z \rvert = \lvert z \rvert_{\min} \geq 1 \}.
\end{equation}
The group $\Lambda$ acts on $\Omega_{\Lambda}$ through shifts $z \mapsto z + \lambda$, and $\mathbf{F}_{\Lambda}$ is a fundamental domain in the sense that each $z \in \Omega_{\Lambda}$ is $\Lambda$-equivalent
with at least one and at most finitely many elements of $\mathbf{F}_{\Lambda}$. For $z \in \Omega_{\Lambda}$ which is $\Lambda$-equivalent with $z' \in \mathbf{F}_{\Lambda}$, define further
\begin{equation}
	\mathrm{type}(z) \defeq \begin{cases} i,	&\text{if $\lvert z' \rvert = \lvert \lambda_{i} \rvert = q^{r_{i}}$ with some $i \in \mathds{N}_{0}$}, \\ i+a,	&\text{if $\lvert z' \rvert = q^{r_{i}(1-a)+r_{i+1}a}$ with $a \in \mathds{Q} \cap (0,1)$} \end{cases}
\end{equation}
(which is independent of the choice of $z'$).

\begin{Proposition}[\cite{Gekeler13} Proposition 5.6]
	The function $t_{\Lambda}$ restricted to $\mathbf{F}_{\Lambda}$ defines a biholomorphic isomorphism of the quotient $\mathbf{F}_{\Lambda}/\Lambda$ (that is, the image of $\mathbf{F}_{\Lambda}$ in $C_{\infty}/\Lambda$,
	which equals $\Omega_{\Lambda}/\Lambda$) with the pointed closed ball $\mathbf{B}(1)\smallsetminus \{0\}$.
\end{Proposition}

\begin{Proposition}[\cite{Gekeler13} Proposition 5.7] \label{Proposition.Map-L-IQ-geq-to-IQ-leq}
	For $z \in \Omega_{\Lambda}$, the absolute value $\lvert t_{\Lambda}(z) \rvert$ depends only on $\mathrm{type}(z)$. Hence there is a well-defined map $L \colon \mathds{Q}_{\geq 0} \to \mathds{Q}_{\leq 0}$ which makes 
	the diagram
	\begin{equation}
		\begin{tikzcd}
			\Omega_{\Lambda} \ar[r, "t_{\Lambda}"] \ar[d, "\mathrm{type}"']	& \mathbf{B}(1) \smallsetminus \{0\} \ar[d, "\log"] \\
			\mathds{Q}_{\geq 0} \ar[r, "L"']																& \mathds{Q}_{\leq 0}
		\end{tikzcd}
	\end{equation}
	commutative. $L$ is bijective and strictly monotonically decreasing. Its values on integers are given by
	\begin{equation}
		L(i) = (q-1) \sum_{0 \leq j < i} r_{j}q^{j} - r_{i}q^{i}	\qquad (i \in \mathds{N}_{0}).
	\end{equation}
\end{Proposition}

Combining the last two propositions, we find

\begin{Corollary} \label{Corollary.Zeroes-of-the-Ck-Lambda-and-G-k-Lambda}
	Given the separable $\mathds{F}$-lattice $\Lambda$ normalized as in \ref{Subsection.Normalizing-assumption}, the following statements for $k \in \mathds{N}$ are equivalent:
	\begin{enumerate}[label=$\mathrm{(\roman*)}$]
		\item All the zeroes of $C_{k, \Lambda}$ lie on critical spheres $\mathbf{S}(q^{r_{i}})$ of $\Lambda$;
		\item all the zeroes of $C_{k, \Lambda}$ are of integer type;
		\item all the zeroes $x \neq 0$ of $G_{k, \Lambda}$ satisfy $\log x = L(i)$ for some $i \in \mathds{N}_{0}$;
		\item all the slopes of the Newton polygon $\mathrm{NP}(G_{k, \Lambda})$ of $G_{k, \Lambda}$ are of shape $L(i)$ for some $i \in \mathds{N}_{0}$.
	\end{enumerate}
\end{Corollary}

Here we use the definitions and conventions of \cite{Neukirch99} II Section 6 for the Newton polygon. It has been shown in \cite{Gekeler13} and \cite{Gekeler13-2}
that statements $\mathrm{(i)}$-$\mathrm{(iv)}$ hold true if $q=p$, and it was conjectured in \cite{Gekeler13-2} that they hold for arbitrary $q = p^{f}$ if $\Lambda = A$. However, this is far too optimistic, see Sections \ref{Section.Irregular-zeroes} and 9,
where counterexamples are given. We will see that $\mathrm{(i)}$-$\mathrm{(iv)}$ at least hold for some $k$ (depending on $q$); these will be called regular.

\subsection{}\label{Subsection.Numbers-of-zeroes} To enable precise statements about the zeroes, we make the following definitions for $i \in \mathds{N}_{0}$.
\begin{align}
	\tilde{\gamma}_{i, \Lambda}(k) 		&\defeq \# \{ \text{zeroes $z$ of $C_{k, \Lambda} \mid \log z = r_{i}$} \}, \\ 
	\tilde{\gamma}_{\Lambda}^{(i)}(k) 	&\defeq \# \{ \text{zeroes $z$ of $C_{k, \Lambda} \mid r_{i} \leq \log z < r_{i+1}$}\} \nonumber.
\end{align}
Here and in the sequel, \enquote{number of zeroes} or \enquote{$\# \{ \text{zeroes}\}$} is always understood as \enquote{numbers of zeroes counted with multiplicity}. The group 
$\Lambda_{i+1} = \sum_{0 \leq j \leq i} \mathds{F}\lambda_{j}$ identifies zeroes $z$ of $C_{k, \Lambda}$ with $r_{i} \leq \log z < r_{i}$ under $t_{\Lambda}$, which explains the factors $q^{i+1}$ in
\begin{align}
	\gamma_{i, \Lambda}(k)			&\defeq \# \{ \text{zeroes of $G_{k, \Lambda}$ with $\log x = L(i)$}\} = \tilde{\gamma}_{i, \Lambda}(k)/q^{i+1}, \\
	\gamma_{\Lambda}^{(i)}(k) 	&\defeq \# \{ \text{zeroes of $G_{k, \Lambda}$ with $L(i+1) < \log x \leq L(i)$}\} = \tilde{\gamma}_{\Lambda}^{(i)}(k)/q^{i+1}. \nonumber
\end{align}
Then $\gamma_{i, \Lambda}(k) = \gamma_{\Lambda}^{(i)}(k)$ or $\tilde{\gamma}_{i,\Lambda}(k) = \tilde{\gamma}_{\Lambda}^{(i)}(k)$ for all $i$ is equivalent with the statements in \ref{Corollary.Zeroes-of-the-Ck-Lambda-and-G-k-Lambda}.
We finally define $\gamma_{\Lambda}(k)$ (without a sub- or superscript $i$) as
\begin{equation}
	\gamma_{\Lambda}(k) \defeq \text{multiplicity of $x=0$ as a zero of $G_{k, \Lambda}$}.
\end{equation}
\subsection{} It is easy to find an expression for $\gamma_{0, \Lambda}(k) = \tilde{\gamma}_{0, \Lambda}(k)/q$. Consider the expansion
\[
	e_{\Lambda}(z) = \sum_{i \geq 0} \alpha_{i}(\Lambda) z^{q^{i}}
\]
of $e_{\Lambda}$. In view of our normalization \ref{Subsection.Normalizing-assumption}, all the $\alpha_{i}$ lie in $O_{C_{\infty}}$,  $\lvert \alpha_{i} \rvert < 1$ for $i>1$, and the coefficientwise congruence
\begin{equation}
	e_{\Lambda}(z) \equiv e_{\mathds{F}}(z) \pmod{\mathfrak{m}_{C_{\infty}}}
\end{equation}
holds, which implies
\begin{equation}
	G_{k, \Lambda}(X) \equiv G_{k, \mathds{F}}(X) \pmod{\mathfrak{m}_{C_{\infty}}}.
\end{equation}
Therefore, $\gamma_{0, \Lambda}(k) = \text{number of zeroes $x \neq 0$ of $G_{k, \mathds{F}}$}$ (as always, counted with multiplicity). Now we have the following closed expression for $G_{k, \mathds{F}}$ (\cite{Gekeler88} 3.7):
\begin{equation}
	G_{k, \mathds{F}}(X) = \sum_{j \geq 0} (-1)^{j} \binom{k-1-j(q-1)}{j}X^{k-j(q-1)}.
\end{equation}
If we put
\begin{multline*}
	\bar{j} = \bar{j}(k) \defeq \text{largest integer $j$ such that} \\ \text{the binomial coefficient $\binom{k-1-j(q-1)}{j}$ doesn't vanish $\pmod{p}$},
\end{multline*}
then:

\begin{Proposition}
	$\gamma_{0, \Lambda}(k) = \bar{j}(q-1)$, $\tilde{\gamma}_{0, \Lambda}(k) = \bar{j}q(q-1)$, independently of $\Lambda$. 
\end{Proposition}

Before approaching the other vanishing numbers, we make an important observation.

\begin{Proposition}[\cite{Gekeler13} Lemma 6.7]
	Let $r \in \mathds{Q}_{\geq 0}$ and $\mathbf{B} = \mathbf{B}(q^{r})$ be the ball with radius $q^{r}$. The number of poles of $C_{k, \Lambda}$ on $\mathbf{B}$ always exceeds the number of zeroes (both counted with multiplicity)
	of $C_{k, \Lambda}$ on $\mathbf{B}$. 
\end{Proposition}

As a consequence, the following corollary holds:

\begin{Corollary}[\cite{Gekeler13} Corollary 6.8] \label{Corollary.Coefficient-condition-Laurent-expansion-of-Ck}
	In the situation of \ref{Subsection.Laurent-expansion-of-C-k-Lambda} assume that in 
	\[
		C_{k, \Lambda}(z) = \sum_{n \in \mathds{Z}} a_{n}w^{n},
	\]
	$\lvert a_{n} \rvert = \max \{ \lvert a_{m} \rvert \mid m \in \mathds{Z} \}$ holds for some $n$. Then $n<0$. 
\end{Corollary}

This is why in the quest for dominating terms we may focus on the coefficients $a_{-k-n}$ in \eqref{Eq.Laurent-expansion-of-Ck-on-Sphere}.

\section{Power sums and Sheats compositions} \label{Section.Power-sums-and-Sheats-compositions}

\subsection{} Still, $\Lambda$ is an $\mathds{F}$-lattice in $C_{\infty}$ with separating basis $\{ \lambda_{0}=1, \lambda_{1}, \dots \}$, $r_{i} = \log \lambda_{i}$, and $\Lambda_{i}$ is the sublattice generated by 
$\lambda_{0}, \dots, \lambda_{i-1}$ (provided that $i \leq \dim_{\mathds{F}} \Lambda$ if $\Lambda$ is finite). 

\subsection{} We first recall the Lucas congruence and some of its consequences. Let $m,n$ be numbers from $\mathds{N}_{0}$, given in their respective $p$-adic expansions
\begin{equation}
	m = \sum_{i \geq 0} m_{i}p^{i}, \qquad n = \sum_{i \geq 0} n_{i}p^{i}, \qquad m_{i},n_{i} \in P \defeq \{0,1,\dots,p-1\}.
\end{equation}
Then
\begin{equation}
	\binom{n}{m} \equiv \prod \binom{n_{i}}{m_{i}} \pmod{p}.
\end{equation}
In particular,
\begin{multline*}
	\binom{n}{m} \not\equiv 0 \pmod{p} \Longleftrightarrow \forall i : m_{i} \leq n_{i} \\ \Longleftrightarrow \text{There is no carryover of digits in $n = m+(n-m)$}.
\end{multline*}
More generally, let $\mathbf{m} = (m_{1}, \dots, m_{h})$ be an $h$-tuple of non-negative integers with $m_{1} + \dots + m_{h}= n$. Then the multinomial coefficient
\begin{equation}
	\binom{n}{\mathbf{m}} = \binom{n}{m_{1}, \dots, m_{h}} = \frac{n!}{m_{1}! \cdots m_{h}!}
\end{equation}
is incongruent to $0 \pmod{p}$ if and only if there is no carryover of $p$-adic digits in the sum $n = m_{1} + \dots + m_{h}$. In this case we write
\begin{equation}
	n = m_{1} \oplus \dots \oplus m_{h}
\end{equation}
and call it the \textbf{smooth sum} of $m_{1}, \dots, m_{h}$.

\subsection{} Recall that $S_{i, \Lambda}(n)$ denotes the power sum $\sum_{\lambda \in \Lambda_{i}} \lambda^{n}$. Writing $S_{i, \Lambda} = \sum_{\mathbf{c} \in \mathds{F}^{i}} (c_{0} \lambda_{0} + \dots + c_{i-1} \lambda_{i-1})^{n}$
for $i>0$ and appyling the multinomial theorem, we find
\[
	S_{i, \Lambda}(n) = \sum_{\mathbf{c} \in \mathds{F}^{i}} \sum_{\mathbf{m}} \binom{n}{\mathbf{m}} \mathbf{c}^{\mathbf{m}} \boldsymbol{\lambda}^{\mathbf{m}} = \sum_{\mathbf{m}} \binom{n}{\mathbf{m}} \boldsymbol{\lambda}^{\mathbf{m}} \sum_{\mathbf{c}} \mathbf{c}^{\mathbf{m}},
\]
where $\mathbf{c} = (c_{0}, \dots, c_{i-1})$, $\mathbf{m} = (m_{0}, \dots, m_{i-1})$ runs through the $i$-tuples of non-negative integers that sum up to $n$, $\mathbf{c}^{\mathbf{m}} = c_{0}^{m_{0}} \cdots c_{i-1}^{m_{i-1}}$, 
$\mathbf{\lambda}^{\mathbf{m}} = \lambda_{0}^{m_{0}} \cdots \lambda_{i-1}^{m_{i-1}}$. Now
\begin{equation}
	\sum_{c \in \mathds{F}} c^{n} = \begin{cases} -1,		&\text{if $0<n \equiv 0 \pmod{q-1}$}, \\ 0,	&\text{if $n=0$ or $n \not\equiv 0 \pmod{q-1}$}. \end{cases}
\end{equation}
Therefore the inner sum $\sum_{\mathbf{c}} \mathbf{c}^{\mathbf{m}}$ vanishes unless all the $m_{j}$ are positive and divisible by $q-1$, in which case it takes the value $(-1)^{i}$. Thus we find
\begin{equation} \label{Eq.Powersum-S-i-Lambda}
	S_{i, \Lambda}(n) = (-1)^{i} \sum_{\mathbf{m}} \binom{n}{\mathbf{m}} \boldsymbol{\lambda}^{\mathbf{m}},
\end{equation}
where $\mathbf{m}$ runs through the $i$-tuples of elements of $(q-1)\mathds{N}$ with $n = m_{0} \oplus \dots \oplus m_{i-1}$. We call such an $\mathbf{m}$ an \textbf{$i$-composition of $n$}. In particular, if $S_{i, \Lambda}(n) \neq 0$
then $n \equiv 0 \pmod{q-1}$ and $n$ admits an $i$-composition as above.

\subsection{} We introduce some more notation. For $p$-adic numbers with (not necessarily finite) expansions $a = \sum a_{i}p^{i}$, $b = \sum_{i} b_{i}p^{i}$, $a_{i}, b_{i} \in P$, we write
\begin{equation}
	a <_{p} b \vcentcolon \Longleftrightarrow \forall i : a_{i} \leq b_{i},
\end{equation}
and for $a,b \in \mathds{N}_{0}$,
\begin{equation}
	a \varprec b \vcentcolon \Longleftrightarrow a <_{p} b \text{ and } a \equiv b \pmod{q-1}.
\end{equation}
Both \enquote{$<_{p}$} and \enquote{$\varprec$} are order relations. As usual, $a \lneq_{p} b$ means $a <_{p} b$ and $a \neq  b$, ditto for \enquote{$\varprec$}. The symbols \enquote{$\inf_{p}$} and \enquote{$\sup_{p}$} denote the
infimum (resp. supremum) with respect to \enquote{$<_{p}$}. We further let $\deg_{p} a$ (resp. $\deg_{q} a$) be the degree of $a$ as a \enquote{polynomial} in $p$ (resp. $q$). Finally, we define the
\textbf{height} $\mathrm{ht}(n)$ of $n \in \mathds{N}_{0}$ as the maximal length of a chain
\begin{equation} \label{Eq.Maximal-chain}
	0 \precneq n_{1} \precneq n_{2} \precneq \dots \precneq n_{h} <_{p} n.
\end{equation}
Hence
\begin{equation}
	S_{i, \Lambda}(n) = 0 \text{ unless } n \equiv 0 \pmod{q-1} \text{ and } \mathrm{ht}(n) \geq i.
\end{equation}
\subsection{}\label{Subsection.Non-vanishing-term-in-Powersum-S-i-Lambda} For a non-vanishing term $\binom{n}{\mathbf{m}} \lambda^{\mathbf{m}}$ in \eqref{Eq.Powersum-S-i-Lambda}, 
$0 \precneq n_{1} \precneq \dots \precneq n_{i} = n$, where $n_{j} \defeq m_{0} \oplus \dots \oplus m_{j-1}$. Note that the size of such a term is given by 
\begin{equation}
	\log \boldsymbol{\lambda}^{\mathbf{m}} = m_{0} r_{0} + \dots + m_{i-1} r_{i-1}, \qquad r_{j} = \log \lambda_{j}.
\end{equation}
We therefore define the \textbf{weight of $\mathbf{m}$} with respect to the \textbf{weight system $\mathbf{r} = (r_{0}, r_{1}, \dots)$} determined by $\Lambda$ as
\begin{equation}
	\mathrm{wt}_{\mathbf{r}}(\mathbf{m}) =\log \boldsymbol{\lambda}^{\mathbf{m}} = \sum_{0 \leq j < i} m_{j} r_{j}.
\end{equation}
If there exists an $i$-composition $\mathbf{m}$ with \textbf{dominant weight}, that is, some $\mathbf{m} = (m_{0}, \dots, m_{i-1})$ such that $\mathrm{wt}_{\mathbf{r}}(\mathbf{m}) > \mathrm{wr}_{\mathbf{r}}(\mathbf{m}')$ for
all $\mathbf{m}' \neq \mathbf{m}$ in \eqref{Eq.Powersum-S-i-Lambda}, then
\begin{equation}
	\log S_{i, \Lambda}(n) = \mathrm{wt}_{\mathbf{r}}(\mathbf{m}).
\end{equation}
Especially, $S_{i, \Lambda}(n) \neq 0$ in this case.

\subsection{} Jeffrey Sheats showed in \cite{Sheats98} that for the lattice $\Lambda = A = \mathds{F}[T]$, that is, for the weight system $\mathds{N}_{0} = (0,1,2,\dots)$ associated to $A$, there always exists a dominant weight
composition $\mathbf{m}$. More precisely, Sheats' Theorem 1.2 says in our language:

Let $n$ be a natural number divisible by $q-1$ with $\mathrm{ht}(n) \geq i$. There exists a unique $i$-composition $\mathbf{X} = (X_{1}, \dots, X_{i})$ of $n$ such that $\mathrm{wt}_{\mathds{N}_{0}}(\mathbf{X})$ is dominant
among all $i$-compositions of $n$. Further, $\mathbf{X}$ is characterized as the $i$-composition of $n$ whose transpose $\mathbf{X}^{t} = (X_{i}, \dots, X_{1})$ is lexicographically maximal. We call $\mathbf{X}$ the
\textbf{Sheats $i$-composition} $\mathbf{Sh}^{(i)}(n)$ of $n$. The \textbf{Sheats composition $\mathbf{Sh}(n)$} of $n$ is $\mathbf{Sh}^{(h)}(n)$ with $h = \mathrm{ht}(n)$. It is obvious that for $i \leq h$ and
$\mathbf{Sh}(n) = (X_{1}, \dots, X_{h})$, the formula
\begin{equation}
	\mathbf{Sh}^{(i)}(n) = (X_{1}, \dots, X_{i-1}, X_{i} + X_{i+1} + \dots + X_{h})
\end{equation}
holds. Similarly,
\begin{equation} \label{Eq.Sheats-composition-is-multiplicative}
	\mathbf{Sh}(pn) = p \mathbf{Sh}(n)
\end{equation}
is straight from the definition. We finally define the \textbf{Sheats sequence} of $n$ as the sequence $(n_{0}, n_{1}, \dots, n_{h})$ with $n_{i} = \sum_{1 \leq j \leq i} X_{j}$. From the lexicographic property we see that
\begin{equation} \label{Eq.deg-of-factors-in-monotonically-increasing}
	\deg_{q}(n_{h-1}) < \deg_{q}(n)
\end{equation}
and, combined with \eqref{Eq.Sheats-composition-is-multiplicative}, even
\begin{equation}
	\deg_{p}(n_{h-1}) \leq \deg_{p}(n) - f
\end{equation}
holds. Also, since \enquote{$\log$} on $A$ agrees with the degree function \enquote{$\deg$} on polynomials, 
\begin{equation} \label{Eq.Relation-deg-of-powersum-and-weight}
	\deg S_{i, A}(n) = \mathrm{wt}_{\mathds{N}_{0}}(\mathbf{Sh}^{(i)}(n))
\end{equation}
if $i \leq \mathrm{ht}(n)$.

\begin{Remarks} \label{Remarks.Differences-to-Sheats-work}
	\begin{enumerate}[wide, label=(\roman*)]
		\item Actually Sheats worked with the weight system $\mathds{N} = (1,2,\dots)$ instead of $\mathds{N}_{0}$. But apart from the obvious formula
		\begin{equation}
			\mathrm{wt}_{\mathds{N}_{0}}(\mathbf{X}) = \mathrm{wt}_{\mathds{N}}(\mathbf{X}) - n
		\end{equation}
		for $i$-compositions $\mathbf{X}$ of $n$ (which explains the difference of \eqref{Eq.Relation-Sheats-factor-and-weight} and \eqref{Eq.Relation-deg-of-powersum-and-weight}), this doesn't affect the 
		nature and formulation of his result.
		\item The reader will have noticed that $i$-compositions of $n$ sometimes are indexed by $\{0,1,\dots,i-1\}$, sometimes by $\{1,2,\dots,i\}$. This is to some extent unavoidable since, depending on the context, the one or the
		other index set is more natural. The author is confident that this doesn't cause serious confusion. 
	\end{enumerate}
\end{Remarks}

As is discussed in the introduction, we need more precise information about Sheats compositions. This will be furnished in Section \ref{Properties-of-Sheats-compositions}. For the moment we restrict to state and prove the following property.

\begin{Proposition} \label{Proposition.Sheats-composition-of-n-J}
	Let $\mathbf{X} = (X_{1}, \dots, X_{i})$ be the Sheats $i$-composition of its value $n = \bigoplus_{1 \leq j \leq i} X_{j}$, and let $J$ be a non-empty subset of $I = \{1,2,\dots, i\}$. Then $\mathbf{X}^{(J)} \defeq (X_{j})_{j \in J}$
	is the Sheats $\lvert J \rvert$-composition of its value $n^{(J)} \defeq \bigoplus_{j \in J} X_{j}$.
\end{Proposition}

\begin{proof}
	\begin{enumerate}[wide, label=(\roman*)]
		\item It is clear that $\mathbf{X}^{(J)}$ inherits the smooth sum property. Therefore it suffices to show that $(\mathbf{X}^{(J)})^{t}$ is lexicographically maximal among all $\lvert J \rvert$-compositions of
		$n^{(J)}$; call this property LM.
		\item By induction, we may suppose that $I = J \cupdot \{k\}$. If $k=i$ then LM for $\mathbf{X}^{(J)}$ is straight.
		\item Similarly, assume $k=1$, and let $(X_{2}', \dots, X_{i}')$ be the Sheats $(i-1)$-composition of $n-X_{1}$. Then in particular $X_{i}' \geq X_{i}$, but, as $(X_{1}, X_{2}', \dots, X_{i}')$ is an $i$-composition 
		of $n$, also $X_{i}' \leq X_{i}$. Hence $X_{i} = X_{i}'$ and $(X_{2}', \dots, X_{i-1}') = (X_{2}, \dots, X_{i-1})$ by the case $k=i$.
		\item Finally, suppose that $1 < k < i$. Let $(X_{1}'', \dots, X_{i-1}'')$ be the Sheats $(i-1)$-composition of $n-X_{k}$, so in particular $X_{i-1}'' \geq X_{i}$. Integrate $X_{k}$ into $(X_{1}'', \dots, X_{i-1}'')$, place it at the right
		position and re-number such that the resulting $(X_{1}', \dots, X_{i}')$ is a monotonically increasing composition of $n$, where $X_{k}$ is one of the $X_{j}'$. Then $X_{i}' = X_{i+1}'' \geq X_{i}$, so we have equality
		$X_{i}' = X_{i}$ since $\mathbf{X}$ is the Sheats $i$-composition of $n$. As before, this implies $X_{j}' = X_{j}$ for all $j$, and in particular $(X_{1}, \dots, \hat{X}_{k}, \dots, X_{i})$ is the Sheats $(i-1)$-composition of $n-X_{k}$. (As usual,
		$\hat{X}$ means omitting the symbol $X$.) \qedhere
	\end{enumerate}
\end{proof}

\section{$p$-adic and $q$-adic expansions} \label{Section.p-adic-and-q-adic-expansions}

\subsection{} Recall that $q = p^{f}$ with the prime number $p$. We let $F \defeq \{0,1, \dots, f-1\}$, $P\defeq \{0,1,\dots, p-1\}$ and $Q \defeq \{0,1,\dots,q-1\}$ be the natural sets of representatives modulo $f$, $p$ and $q$, respectively.
Given $n \in \mathds{N}_{0}$, we let
\begin{equation} \label{Eq.p-adic-and-q-adic-expansions-of-n}
	n = \sum_{j \geq 0} a_{j} p^{j} = \sum_{j \geq 0} b_{j}q^{j}
\end{equation}
with $a_{j} \in P$, $b_{j} \in Q$ be the expansions to base $p$ and $q$. Occasionally, we use the representation of $n$ as a $p$-string
\begin{equation}
	n = (a_{0}, \dots, a_{\alpha}) = (a_{0}, \dots, a_{\alpha})_{p}, \qquad \alpha = \deg_{p}(n)
\end{equation}
or as a $q$-string
\begin{equation*}
	n = (b_{0}, \dots, b_{\beta}) = (b_{0}, \dots, b_{\beta})_{q}, \qquad \beta = \deg_{q}(n).
\end{equation*}
We regard $n$ as the sum of its \textbf{$p$-parts} $p^{j}$ as in \ref{Eq.p-adic-and-q-adic-expansions-of-n}, where the $p$-part $p^{j}$ is of type $i \in F$ if $j \equiv i \pmod{f}$. For $i \in F$ put
\begin{equation}
	\ell_{i}(n) = \sum_{j \equiv i \pmod{f}} a_{j}
\end{equation}
for the number of $p$-parts of type $i$, and 
\begin{equation}
	\ell(n) = \sum_{i \in F} \ell_{i}(n) p^{i} = \sum_{j \geq 0} b_{j}
\end{equation}
for the $q$-adic digit sum of $n$. Let further $\mathbf{C}$ be the group of cyclic permutations of $F$ generated by the $f$-cycle $\sigma_{1} = (0,1,2,\dots,f-1)$. Then 
\begin{align}
	F \overset{\cong}{\longrightarrow}	\mathds{Z}/(f) 	&\overset{\cong}{\longrightarrow}	\mathbf{C}. \\
																						i		&\longmapsto \sigma_{i} = (\sigma_{1})^{i} \nonumber
\end{align}
We decompose $\mathds{N}_{0} = \bigcupdot_{t \in \mathds{N}_{0}} (tf + F)$ into shifts of $F$ and extend the $\mathbf{C}$-action to $\mathds{N}_{0}$ by 
\begin{equation}
	\sigma(tf + i) = tf + \sigma(i), \qquad i \in F,
\end{equation}
the \textbf{tautological action}. This in turn induces an action on $Q$ by 
\begin{equation}
	\Big( \sum_{i \in F} a_{i}p^{i} \Big)^{\sigma} \defeq \sum a_{\sigma(i)} p^{i} = \sum a_{i} p^{\sigma^{-1}(i)}.
\end{equation}
We write this action (the \textbf{true action}) as $(\cdot)^{\sigma}$ from the right in order to not confuse it with the tautological action. Of course we still have the rule $x^{\sigma \tau} = (x^{\tau})^{\sigma} = (x^{\sigma})^{\tau}$, as
$\mathbf{C}$ is abelian. The true action of $\mathbf{C}$ is extended to $\mathds{Z}_{p}$ and its subset $\mathds{N}_{0}$ by
\begin{equation}
	\Big( \sum_{j \geq 0} b_{j}q^{j} \Big)^{\sigma} = \sum_{j \geq 0} (b_{j})^{\sigma} q^{j}
\end{equation}
on $q$-adic expansions. On $p$-adic expansions it reads
\begin{equation} \label{Eq.True-action-on-p-adic-expansions}
	\Big( \sum_{j \geq 0} a_{j}p^{j} \Big)^{\sigma} = \sum_{j \geq 0} a_{\sigma(j)} p^{j} = \sum_{j \geq 0} a_{j} p^{\sigma^{-1}(j)}.
\end{equation}
Especially on $p$-parts: If $p^{j}$ is of type $i \in F$ and $\sigma = \sigma_{s}$ with $s \in F$, then
\begin{equation}\label{Eq.True-action-on-p-parts}
	(p	^{j})^{\sigma} = p^{\sigma^{-1}(j)} = \begin{cases} p^{j-s},	&\text{if } s \leq i, \\ p^{j+f-s},		&\text{if } s > i. \end{cases}
\end{equation}

Finally, we define the twisted sum-of-digits function $\ell^{\sigma}$ on $\mathds{N}_{0}$ by
\begin{equation}
	\ell^{\sigma}(n) \defeq \ell(n^{\sigma}), \qquad \sigma \in \mathbf{C}.
\end{equation}
Caution: This does not agree with $(\ell(n))^{\sigma}$! But see Corollary \ref{Corollary.Congruences-of-sigma-action}.

\begin{Lemma} \label{Lemma.Compatibility-with-smooth-sums}
	\begin{enumerate}[label=$\mathrm{(\roman*)}$]
		\item The functions $\ell_{i}$, $\ell$, $\ell^{\sigma}$ are compatible with smooth sums. That is, $\ell_{i}(m \oplus n) = \ell_{i}(m) + \ell_{i}(n)$, and ditto for $\ell$ and $\ell^{\sigma}$;
		\item $\sigma \in \mathbf{C}$ respects smooth sums, i.e., $(m \oplus n)^{\sigma} = m^{\sigma} \oplus n^{\sigma}$;
		\item $\ell(n) \equiv n \pmod{q-1}$;
		\item for $m,n \in \mathds{N}_{0}$, the equivalence holds:
		\[
			m <_{p} n \Longleftrightarrow \forall \sigma \in \mathbf{C}: m^{\sigma} \leq n^{\sigma}.
		\]
	\end{enumerate}
\end{Lemma}

\begin{proof}
	(i) and (ii) are straight from definitions, and (iii) is easy and well-known. It suffices to show (iv) for $m,n \in Q$. Then \enquote{$\Rightarrow$} is obvious. For the other direction, write $m = \sum_{i \in F} m_{i}p^{i}$ and
	$n = \sum_{i \in F} n_{i}p^{i}$ with $m_{i}$, $n_{i} \in P$. If $m_{i} > n_{i}$ for some $i$ then for $\sigma = \sigma_{f-1-i}$, $m^{\sigma} = m_{i}p^{f-1} + \text{smaller terms} > n^{\sigma} = n_{i}p^{f-1} + \text{smaller terms}$.
\end{proof}

For not necessarily smooth sums, there is at least a congruence.

\begin{Proposition} \label{Proposition.Congruences-for-m-sigma-and-n-sigma}
	For $m, n \in \mathds{N}_{0}$ and $\sigma \in \mathbf{C}$, the congruence $(m+n)^{\sigma} \equiv m^{\sigma} + n^{\sigma} \pmod{q-1}$ holds.
\end{Proposition}

\begin{proof}
	\begin{enumerate}[label=(\roman*), wide]
		\item By an obvious induction, it suffices to treat the case where $n=1$.
		\item First assume that both $m$ and $m+n = m+1$ belong to $Q$. Write $m = \sum_{i \in F} a_{i}p^{i}$ with $a_{i} \in P$, and let $\bar{i}$ be the least $i$ such that $a_{i} < p-1$. Then
		\[
			m+n = m+1 = (a_{\bar{i}} + 1)p^{\bar{i}} + \sum_{i > \bar{i}} a_{i}p^{i}
		\]
		and, with $\sigma = \sigma_{s}$ ($0 \neq s \in F$):
		\begin{align}
			(m+n)^{\sigma} 	&= (a_{\bar{i}} +1)p^{\sigma^{-1}(\bar{i})} + R, \qquad R \defeq \sum_{i > \bar{i}} a_{i}p^{\sigma^{-1}(i)}, \nonumber \\
			m^{\sigma} 			&= (p-1)p^{\sigma^{-1}(0)} + \dots + (p-1)p^{\sigma^{-1}(\bar{i}-1)} + a_{\bar{i}} p^{\sigma^{-1}(\bar{i})} + R. \label{Eq.Equality-for-m-sigma}
		\end{align}
		We distinguish the three cases $s > \bar{i}$, $s = \bar{i}$ and $s < \bar{i}$.
		
		If \fbox{$s > \bar{i}$} then by \eqref{Eq.True-action-on-p-parts} the right hand side of \eqref{Eq.Equality-for-m-sigma} equals
		\[
			(p-1) p^{f-s} + \dots + (p-1)p^{f-s+\bar{i}-1} + a_{\bar{i}}p^{f-s+\bar{i}} + R,
		\]
		which added to $n^{\sigma} = 1^{\sigma} = p^{f-s}$ yields $(m+n)^{\sigma}$.
		
		For \fbox{$s = \bar{i}$} we get 
		\begin{align*}
			m^{\sigma}							&= (p-1)p^{f-s} + \dots + (p-1)p^{f-1} + a_{\bar{i}} p^{0}  +R, \\
			m^{\sigma} + n^{\sigma}	&= m^{\sigma} + p^{f-s} = q-1 + (a_{\bar{i}} +1)p^{0} + R = q-1 + (m+n)^{\sigma}.
		\end{align*}
		
		Now suppose that \fbox{$s < \bar{i}$}. Then
		\[
			m^{\sigma} = (p-1)p^{f-s} + \dots + (p-1)p^{f-1} + (p-1)p^{0} + \dots + (p-1)p^{\bar{i}-s-1} + a_{\bar{i}} p^{\bar{i}-s} + R.
		\]
		In this case, 
		\[
			m^{\sigma} + n^{\sigma} = m^{\sigma} +p^{f-s} = q-1 + (a_{\bar{i}} +1)p^{\bar{i}-s} + R = q-1 + (m+n)^{\sigma}
		\]
		as before. Hence the assertion is true in all cases.
		\item At last, we handle the case of general $m = \sum b_{j}q^{j}$, $b_{j} \in Q$. Let now $\bar{j}$ be minimal with $b_{j} < q-1$. Then
		\[
			(m+n)^{\sigma} = (m+1)^{\sigma} = (b_{\bar{j}} +1)^{\sigma} q^{\bar{j}} + R, \qquad R \defeq \sum_{j > \bar{j}} b_{j}^{\sigma} q^{j}.
		\]
		Now
		\begin{align*}
			m^{\sigma}							&= (q-1)q^{0} + \dots + (q-1)q^{\bar{j}-1} + (b_{\bar{j}})^{\sigma} q^{\bar{j}} + R &&(\text{since $(q-1)^{\sigma} = q-1$}) \\
			n^{\sigma}								&= 1^{\sigma} = p^{f-s}, \text{ and} \\
			m^{\sigma} + n^{\sigma}	&= q^{\bar{j}}-1 + 1^{\sigma} + b_{\bar{j}}^{\sigma}q^{\bar{j}} + R.
		\end{align*}
		Hence
		\[
			m^{\sigma} + n^{\sigma} - (m+n)^{\sigma} = q^{\bar{j}} -1+1^{\sigma} + (b_{\bar{j}}^{\sigma} - (b_{\bar{j}} + 1)^{\sigma})q^{\bar{j}} \equiv 1^{\sigma}(1-q^{\bar{j}}) \equiv 0 \pmod{q-1},
		\]
		where the first congruence comes from (ii), applied to $(b_{\bar{j}} + 1)^{\sigma}$. \qedhere
	\end{enumerate}
\end{proof}

\begin{Corollary} \label{Corollary.Congruences-of-sigma-action}
	\begin{enumerate}[label=$\mathrm{(\roman*)}$]
		\item The following congruences modulo $q-1$ hold for $\sigma \in \mathbf{C}$:
		\begin{align}
			\text{If } m &\equiv n \text{ then } m^{\sigma} \equiv n^{\sigma}, \label{Eq.Congruence-of-m-and-n-implies-congruence-of-m-sigma-and-n-sigma}\\ 
			\intertext{in particular,}
			\text{if } m	&\equiv 0 \text{ then } m^{\sigma} \equiv 0; \label{Eq.Congruence-to-zero-implies-congruence-of-m-sigma-to-0} \\
			\ell^{\sigma}(n)	&\equiv (\ell(n))^{\sigma}.
		\end{align}
		\item If $m<_{p} n$ (resp. $m \varprec n$) then $m^{\sigma} <_{p} n^{\sigma}$ (resp. $m^{\sigma} \varprec n^{\sigma}$),
		\item For all $\sigma \in \mathbf{C}$: $\mathrm{ht}(n^{\sigma}) = \mathrm{ht}(n)$.
	\end{enumerate}
\end{Corollary}

\begin{proof}
	From \ref{Proposition.Congruences-for-m-sigma-and-n-sigma}, $(m+q-1)^{\sigma} \equiv m^{\sigma} + (q-1)^{\sigma} = m^{\sigma} +q-1 \equiv m^{\sigma}$, which gives 
	\eqref{Eq.Congruence-of-m-and-n-implies-congruence-of-m-sigma-and-n-sigma} and \eqref{Eq.Congruence-to-zero-implies-congruence-of-m-sigma-to-0}. Further, 
	$\ell^{\sigma}(n) = \ell(n^{\sigma}) \equiv n^{\sigma} \equiv (\ell(n))^{\sigma}$, where the last congruence comes from \eqref{Eq.Congruence-of-m-and-n-implies-congruence-of-m-sigma-and-n-sigma}, as $n \equiv \ell(n)$. 
	This shows (i). The assertion about \enquote{$<_{p}$} in (ii) has been given in \ref{Lemma.Compatibility-with-smooth-sums}(iv), the one concerning \enquote{$\varprec$} comes from 
	\ref{Proposition.Congruences-for-m-sigma-and-n-sigma}.
	
	Ad (iii): By (ii), the sequence \eqref{Eq.Maximal-chain} for $n$ is mapped under $\sigma$ to a similar sequence for $n^{\sigma}$.
\end{proof}

\begin{Lemma} \label{Lemma.Formula-for-sigma-of-partial-sum}
	Consider the sum $S = \sum_{j \in F} a_{j}p^{j}$ with arbitrary coefficients $a_{j} \in \mathds{N}_{0}$. For $\sigma = \sigma_{s}$ ($s\in F$) and $i$ with $1 \leq i \leq f$ let 
	\begin{equation}
		S^{(i)} \defeq \sum_{0 \leq j < i} a_{j}p^{j}
	\end{equation}
	be the partial sum and 
	\begin{equation}
		S^{\sigma} \defeq \sum_{j \in F} a_{\sigma(j)} p^{j}.
	\end{equation}
	Then the relation
	\begin{equation} \label{Eq.True-action-on-sum}
		S^{\sigma} = p^{-s}(S + (q-1)S^{(s)})
	\end{equation}
	holds.
\end{Lemma}

\begin{proof}
	Direct computation, omitted.
\end{proof}

Caution: $S^{\sigma}$ is not determined by $\sigma$ and the value of $S$, but depends on the sum presentation of $S$. 

We will apply the Lemma to $S = \ell(n) = \sum_{j \in F} \ell_{j}(n) p^{j}$ and the corresponding partial sums
\begin{equation}
	\ell^{(i)}(n) \defeq \sum_{0 \leq j < i} \ell_{j}(n)p^{j}, \qquad 1 \leq i \leq f.
\end{equation}

\begin{Corollary} \label{Corollary.Inequalities-for-ell-under-true-action}
	For $n \in \mathds{N}_{0}$ we have: $\forall \sigma \in \mathbf{C} : \ell(n) \leq \ell^{\sigma}(n)$ if and only if for each $i$ with $1 \leq i \leq f$, the inequality labeled by i
	\begin{equation} \tag{$i$}
		\ell^{(i)}(n)/(p^{i}-1) \geq \ell(n)/(q-1)
	\end{equation}
	holds.
\end{Corollary}

\begin{proof}
	This follows from \ref{Lemma.Formula-for-sigma-of-partial-sum} by a small computation.
\end{proof}

\begin{Definition} \label{Definition.Indicator-function}
	The \textbf{indicator function $I$} is defined through
	\[
		I(n) \defeq \min_{\sigma \in \mathbf{C}} \ell^{\sigma}(n).
	\]
	It is strictly monotonically increasing from $(\mathds{N}_{0}, <_{p})$ to $(\mathds{N}_{0}, \leq)$. Some $\sigma \in \mathbf{C}$ or the corresponding $s \in F$ with $\sigma = \sigma_{s}$ is \textbf{critical} for $n$ if
	$I(n) = \ell^{\sigma}(n)$. 
\end{Definition}

We may now state the main technical result of this section

\begin{Theorem} \label{Theorem.Characterization-indicator-function}
	Given $n$ and $H$ in $\mathds{N}_{0}$, the following assertions are equivalent:
	\begin{enumerate}[label=$\mathrm{(\roman*)}$]
		\item $\mathrm{ht}(n) \geq H$;
		\item $I(n) \geq (q-1)H$;
		\item there exists $\sigma \in \mathbf{C}$ such that for $i = 1,2,\dots, f$ the inequalities $\ell^{(i)}(n^{\sigma}) \geq (p^{i}-1)H$ hold.
	\end{enumerate}
\end{Theorem}

\begin{proof}
	\enquote{(i) $\Rightarrow$ (ii)}: Let $0 = n_{0} \precneq n_{1} \precneq \dots \precneq n_{h} <_{p} n$ be a chain of length $h = \mathrm{ht}(n) \geq H$. Each difference $n_{i} - n_{i-1}$ is divisible by $q-1$, so $\ell^{\sigma}(n_{i}-n_{i-1})$
	too ($\sigma \in \mathbf{C}$, $i = 1,2,\dots, f$) by \eqref{Eq.Congruence-to-zero-implies-congruence-of-m-sigma-to-0}. As $\ell^{\sigma}$ is additive in smooth sums, $\ell^{\sigma}(n) \geq \ell^{\sigma}(n_{h}) \geq (q-1)h \geq (q-1)H$.
	
	\enquote{(ii) $\Rightarrow$ (iii)}: Replacing $n$ with $n^{\sigma}$, where $\sigma$ is critical for $n$, we may assume that $I(n) = \min_{\sigma} \ell^{\sigma}(n) = \ell(n)$. By \ref{Corollary.Inequalities-for-ell-under-true-action}, 
	for each $i$
	\[
		\ell^{(i)}(n) \geq (p^{i}-1) \ell(n)/(q-1) \geq (p^{i}-1)H.
	\]
	
	\enquote{(iii) $\Rightarrow$ (i)}: \begin{enumerate}[label=(\alph*), wide] \item As before we may assume $\sigma = \sigma_{0} = \mathrm{id}$ for the $\sigma$ in assertion (iii). We shall give an algorithm that picks some $p$-parts
	of $n$ which, summed up, will give $n_{1}$ with $\ell(n_{1}) = q-1$. Then $0 \precneq n_{1} <_{p} n$. The algorithm is such that for $n' \defeq n-n_{1}$ the system of inequalities
	\begin{equation} \label{Eq.proof-system-of-inequalities}
		\ell^{(i)}(n') \geq (p^{i}-1)(H-1) \qquad (i=1,2,\dots,f)
	\end{equation}
	holds and thus induction applies.
	\item \textbf{Algorithm:} Put $n_{1} \defeq 0$.
	
	\textbf{Step 0:} Put the $p-1$ smallest $p$-parts of $n$ of type 0 to $n_{1}$. This is possible since $\ell_{0}(n) = \ell^{(1)}(n) \geq (p-1)H \geq p-1$. Now $\ell_{0}(n_{1})=p-1$.
	
	\textbf{Step 1:} Put the $p-1$ smallest $p$-parts of type 1 to $n_{1}$. This is possible if $\ell_{1}(n) \geq p-1$. If there are not enough of them, that is, if $a \defeq \max(p-1-\ell_{1}(n),0)>0$ then compensate by putting
	the smallest $ap$ $p$-parts of type 0 not used in Step 0 to $n_{1}$. This is possible since
	\[
		\ell_{0}(n) + p\ell_{1}(n) = \ell^{(2)}(n) \geq (p^{2}-1)H \geq p^{2}-1.
	\]
	Now $\ell_{0}(n_{1}) = p-1+ap$, $\ell_{1}(n_{1}) = p-1-a$, $\ell_{0}(n_{1}) + p\ell_{1}(n) = p^{2}-1$.
	
	\textbf{Step 2:} Put the $p-1$ smallest $p$-parts of type 2 to $n_{1}$. Possible if $\ell_{2}(n) \geq p-1$. If $b \defeq \max(p-1-\ell_{2}(n), 0) > 0$, compensate by adding the largest possible number $b_{1}$ with
	$b_{1}p \leq bp^{2}$ of smallest $p$-parts of type 1 not yet used to $n_{1}$, and if necessary, of $b_{0}$ smallest not yet used $p$-parts of type 0 such that $b_{0} + b_{1}p = bp^{2}$. This is possible by 
	\[
		\ell_{0}(n) + p\ell_{1}(n) + p^{2} \ell_{2}(n) = \ell^{(3)}(n) \geq (p^{3}-1)H \geq p^{3}-1.
	\]
	The current values of the $\ell_{i}$ are 
	\[
		\ell_{0}(n_{1}) = p-1 +ap +b_{0}, \qquad \ell_{1}(n_{1}) = p-1-a+b_{1}, \qquad \ell_{2}(n_{1}) = p-1-b,
	\]
	such that $\ell_{0}(n_{1}) + p \ell_{1}(n_{1}) + p^{2} \ell_{2}(n_{1}) = p^{3}-1$.
	
	\textbf{Step 3:} \\
	$\vdots$  \\
	\textbf{Step f-1:} Continue until step $f-1$, which determines $\ell_{0}(n_{1}), \dots, \ell_{f-1}(n_{1})$ with
	\begin{equation} \label{Eq.proof-expansion-of-ln1-in-terms-of-lin}
		\sum_{0 \leq i < f} p^{i} \ell_{i}(n_{1}) = \ell^{(f)}(n_{1}) = \ell(n_{1}) = q-1.
	\end{equation}
	This finishes the algorithm. In the $i$-th step, $\ell_{i}(n_{1})$ is first determined as $\min( \ell_{i}(n), p-1)$; if the wanted number $p-1$ cannot be reached since $\ell_{i}(n) < p-1$ then compensate by the maximal available number
	of minimal unused $p$-parts of type $i-1$, then possibly of $p$-parts of type $i-2$, etc.
	\item Let $n' \defeq n-n_{1}$. It remains to show the inequalities \eqref{Eq.proof-system-of-inequalities}. This is done by descending induction on $i$, where the assertion for the starting value $i=f$ is 
	\eqref{Eq.proof-expansion-of-ln1-in-terms-of-lin}. Put $d_{f} \defeq 0$ and for $i<f$, $d_{i} \defeq \max(p-1+d_{i+1}-\ell_{i}(n),0)$, so
	\[
		d_{i} > 0 \Longleftrightarrow \ell_{i}(n) < p-1 + d_{i+1}.
	\]
	Then $d_{i}$ is such that for each $i<f$
	\begin{equation} \label{Eq.proof-Relationship-between-geometric-p-sum-and-li-expansion}
		(\ell_{i}(n_{1}) + d_{i})p^{i} + \sum_{i<j<f} \ell_{j}(n_{1})p^{j} = (p-1)(p^{i}+p^{i+1}+ \dots + p^{f-1})
	\end{equation}
	holds.
	\item Let $i<f$ and suppose that \fbox{$d_{i}=0$}. Then by \eqref{Eq.proof-Relationship-between-geometric-p-sum-and-li-expansion}, $\ell^{(i)}(n_{1}) = (p-1)(1+p+\dots+p^{i-1}) = p^{i}-1$, so
	\[
		\ell^{(i)}(n') = \ell^{(i)}(n) - \ell^{(i)}(n_{1}) \geq (p^{i}-1)H - (p^{i}-1) = (p^{i}-1)(H-1).
	\]	
	If \fbox{$d_{i}>0$} then $\ell_{i}(n_{1}) = \ell_{i}(n)$, so 
	\[
		\ell^{(i+1)}(n') = \ell^{(i)}(n') + \ell_{i}(n')p^{i} = \ell^{(i)}(n'),
	\]
	where by induction assumption, $\ell^{(i+1)}(n') \geq (p^{i+1}-1)(H-1)$. Hence $\ell^{(i)}(n') \geq (p^{i+1}-1)(H-1) \geq (p^{i}-1)(H-1)$, too. In both cases, we have the wanted inequality, which finishes the induction for
	\eqref{Eq.proof-system-of-inequalities} and thereby the proof of \enquote{(iii)$\Rightarrow$(i)}. \qedhere
	\end{enumerate}
\end{proof}

\section{Properties of the Sheats composition} \label{Properties-of-Sheats-compositions}

\subsection{} In this section, $n$ will be a natural number divisible by $q-1$ with height $h = \mathrm{ht}(n)$. Let
\begin{equation}
	\mathbf{Sh}(n) = (X_{1}, \dots, X_{h})
\end{equation}
bei its Sheats composition and $(0=n_{0}, n_{1}, \dots, n_{h}=n)$ its Sheats sequence, $n_{i} = \sum_{1 \leq j \leq i} X_{j}$. We also write $\mathrm{Sh}_{i}(n) \defeq X_{i}$ and call it the $i$-th \textbf{Sheats factor} of $n$.
Below we shall present an Algorithm $\mathbf{Alg}$ that calculates the largest factor $\mathrm{Sh}(n) \defeq \mathrm{Sh}_{\mathrm{ht}(n)}(n)$, and thus by the induction formulas
\begin{equation}
	n_{i-1} = n_{i} - \mathrm{Sh}_{i}(n) = n_{i} - \mathrm{Sh}(n_{i}), \qquad \mathrm{Sh}_{i-1}(n) = \mathrm{Sh}(n_{i-1}) \qquad (i=1,2,\dots,h)
\end{equation}
the complete Sheats composition and sequence. We start with the trivial

\begin{Observation} \label{Observation.On-sigmas-critical-for-n}
	Let $\sigma \in \mathbf{C}$ be critical for $n$. Then $\ell^{\sigma}(n_{i}) = (q-1)i$, and for $m$ with $n_{i-1} \lneq_{p} m \lneq_{p} n_{i}$, the inequalities $(q-1)(i-1) < \ell^{\sigma}(m) < (q-1)i$ hold. In particular, 
	$\sigma$ is critical for all the $n_{i}$.
\end{Observation}

\subsection{}\label{Subsection.Algorithm-and-characterization-of-height} The algorithm uses in a crucial way the characterization of the height through the indicator function $I$ given in Theorem 
\ref{Theorem.Characterization-indicator-function}. As usual, we regard $n$ as the sum of its $p$-parts $d_{\nu} = p^{j_{\nu}}$. These are arranged in \textbf{decreasing order} and numbered accordingly. 
Thus $d_{1} = p^{j_{1}} \geq d_{2} = p^{j_{2}} \geq \dots$, $j_{1} = \deg_{p} n$, and each $p^{j}$ appears $a_{j}$ times, where $n = \sum a_{j}p^{j}$ is the $p$-adic expansion. The number of $p$-parts is 
\begin{equation}
	\ell_{p}(n) \defeq \ell_{0}(n) + \ell_{1}(n) + \dots + \ell_{f-1}(n) = \sum a_{j}.
\end{equation}
$\mathbf{Alg}$ dictates which of the $p$-parts have to be picked for $X_{h} = \mathrm{Sh}(n)$. It is written as pseudo-code, and we use the convention that \enquote{$m \defeq m+t$} adds $t$ to the stored variable $m$, etc.

\begin{Algorithm}[$\mathbf{Alg}$] ~\label{Algorithm.To-determine-Sheats-factor}
	\begin{enumerate}\setcounter{enumi}{-1}
		\item begin
		\item $m \defeq 0$, $m' \defeq n$, $\nu = 1$;
		\item if $I(m' -d_{\nu}) \geq (q-1)(h-1)$ then $(m \defeq m+ d_{\nu}, m' \defeq m' - d_{\nu}$);
		\item if $I(m) < q-1$ then $(\nu \defeq \nu+1$, goto (2)) \\
					else $(X_{h} \defeq m$, $n_{h-1} \defeq m'$);
		\item end
	\end{enumerate}
\end{Algorithm}

\subsection{} We say that the $\nu$-th $p$-part $d_{\nu}$ with value $p^{j_{\nu}}$ \textbf{belongs to $X_{h}$} (or is \textbf{contained in $X_{h}$}) if it has been picked for $X_{h}$ in step (2). (This depends on the number $\nu$ and not
merely on the value $p^{j_{\nu}}$ of $d_{\nu}$!)

We need to generalize this notion---somewhat long-winded---to arbitrary numbers $M <_{p} n$ with $p$-expansion $M = \sum b_{j}p^{j}$, where $b_{j} \leq a_{j}$. Given $d_{\nu}$ with value $p^{j_{\nu}}$, consider the numbers
\begin{equation}
	\alpha \defeq \min\{ \mu \mid j_{\mu} = j_{\nu}\}, \qquad \beta \defeq \max\{ \mu \mid j_{\mu} = j_{\nu}\}.
\end{equation}
Then $\alpha \leq \nu \leq \beta$, and we say that $d_{\nu}$ \textbf{belongs to $M$} if $v-\alpha < b_{j_{\nu}}$. That is, among all the $a_{j_{\nu}} = \beta - \alpha + 1$ many terms $d_{\nu}$ ($\alpha \leq \nu \leq \beta$) with value
$p^{j_{\nu}}$, only the first $b_{j_{\nu}}$ many belong to $M$. Due to the nature of $\mathbf{Alg}$, this is coherent with the former definition of $d_{\nu}$ belonging to $X_{h}$.

\begin{Theorem} \label{Theorem.On-correctness-of-Alg}
	\begin{enumerate}[label=$\mathrm{(\roman*)}$]
		\item $\mathbf{Alg}$ terminates at a step $\nu \leq \ell_{p}(n)$;
		\item the final value of $I(m)$ is $q-1$;
		\item the final values of $m$ and $m'$ are $m = X_{h} = \mathrm{Sh}(n)$ and $m' = n_{h-1} = n-\mathrm{Sh}(n)$ (as claimed in the else clause of step $\mathrm{(3)}$).
	\end{enumerate}
\end{Theorem}

\begin{proof}
	\begin{enumerate}[label=(\roman*), wide]
		\item Suppose this is not the case, i.e., we still have $I(m) < q-1$ when $\nu$ attains its maximal value $\ell_{p}(n)$. Then all the $p$-parts $p^{j}$ of $m'$ are forebidden, that is,
		\begin{equation} \label{Eq.Inequality-for-Im-prime-minus-p-j}
			I(m' - p^{j}) < (q-1)(h-1).
		\end{equation}
		Let $\sigma \in \mathbf{C}$ be critical for $n$. As $(q-1)h = \ell^{\sigma}(n) = \ell^{\sigma}(m) + \ell^{\sigma}(m')$, then $\ell^{\sigma}(m') > (q-1)(h-1)$, so by Theorem \ref{Theorem.Characterization-indicator-function} applied
		to $(m')^{\sigma}$, there exists a sequence 
		\[
			0 \precneq m_{1}' \precneq \dots \precneq m_{h-1}' <_{p} m'.
		\]
		
		As $\ell^{\sigma}(m_{i}') = (q-1)i$ for $1 \leq i \leq h-1$, the inequality $m_{h-1}' <_{p} m'$ is strict, and there exists a $p$-part belonging to $m'$ with some value $p^{j}$ such that $\mathrm{ht}(m'-p^{j}) \geq h-1$, that is,
		$I(m'-p^{j}) \geq (q-1)(h-1)$, which contradicts \eqref{Eq.Inequality-for-Im-prime-minus-p-j}.
		\item For the last value of $\nu$, $I(m') \geq (q-1)(h-1)$ by definition and $I(m) \geq q-1$ by the exit condition. As $I(m) \leq \ell^{\sigma}(m)$, $I(m') \leq \ell^{\sigma}(m')$ and 
		$\ell^{\sigma}(m) + \ell^{\sigma}(m') = \ell^{\sigma}(n) = (q-1)h$ for the $n$-critical $\sigma \in \mathbf{C}$, we have in fact equality everywhere.
		\item Let $m$ and $m'$ be the final values of $\mathbf{Alg}$. Then $0 \precneq m \varprec n$ and $\mathrm{ht}(n-m) = \mathrm{ht}(m') \geq h-1$. Suppose that the true Sheats factor $M \defeq \mathrm{Sh}(n)$ is strictly larger 
		than $m$. Let $\mu \in \mathds{N}$ be the minimal number such that $d_{\mu} = p^{j_{\mu}}$ belongs to precisely one of $m$ or $M$. As $m<M$, $d_{\mu}$ belongs to $M$ but not to $m$. Put
		\[
			S \defeq \{ \lambda \in \mathds{N} \mid \lambda < \mu \text{ and } d_{\lambda} \text{ belongs to $m$, i.e., to $m$ and $M$}\}.
		\]
		As $\sum_{\lambda \in S} p^{j_{\lambda}} + p^{j_{\mu}} <_{p} M$, the complement $n - \sum_{\lambda \in S} d_{\lambda} - d_{\mu}$ has height $\geq h-1$, so $I(n - \sum_{\lambda \in S} d_{\lambda} - d_{\mu}) \geq (q-1)(h-1)$,
		which conflicts with $d_{\mu}$ not belonging to $m$. Therefore $m = \mathrm{Sh}(n)$ and $m' = n - \mathrm{Sh}(n) = n_{h-1}$. \qedhere
	\end{enumerate}
\end{proof}

\begin{Corollary} \label{Corollary.Maximal-Sheats-compositions}
	Let $\mathbf{Sh}(n) = (X_{1}, \dots, X_{h})$ be the Sheats composition of $n$, with Sheats sequence $(n_{i})$, $n_{i} = X_{1} + \dots + X_{i}$, $i \leq h$.
	\begin{enumerate}[label=$\mathrm{(\roman*)}$]
		\item For each $i$ with $1 \leq i < h$, $n-n_{i} = X_{i+1} + \dots + X_{h}$ is maximal among all $m \varprec n$ that satisfy $\mathrm{ht}(m) = h-i$, $\mathrm{ht}(n-m)=i$.
		\item Dually, $n_{i}$ is minimal among all $m \varprec n$ with $\mathrm{ht}(m) = i$, $\mathrm{ht}(n-m) = h-i$.
	\end{enumerate}
\end{Corollary}

\begin{proof}
	Clearly, (ii) follows from (i). Now (i) holds for $i = h-1$ as the characteristic property of the Sheats composition, and the general case follows by descending induction on $i$. As in the proof of of \ref{Theorem.On-correctness-of-Alg}(iii),
	let $M \varprec n$ be such that $\mathrm{ht}(M) = h-i$ and $\mathrm{ht}(n-M) = i$, and assume that $M > m \defeq n-n_{i}$. Let $\mu$ be the first number such that $m$ and $M$ differ in the $p$-part $d_{\mu}$; then
	$d_{\mu}$ belongs to $M$ but not to $m$. Again letting
	\[
		S \defeq \{ \lambda \in \mathds{N} \mid \lambda < \mu \text{ and $d_{\lambda}$ belongs to $m$ and $M$}\},
	\]
	then $I(n - \sum_{\lambda \in S} d_{\lambda} - d_{\mu}) \geq (q-1)i$ since $\sum d_{\lambda} + d_{\mu} <_{p} M$, but the same quantitiy must be less than $(q-1)i$ since $d_{\mu}$ has not been picked for 
	$X_{h}, X_{h-1}, \dots, X_{i+1}$. Hence $m = n-n_{i}$ has the stated maximality property.
\end{proof}

It is now the right place to show by counterexample that some likely-seeming assertions on Sheats compositions and sequences are in fact wrong. In the examples, $q=4$, numbers are given by their $p$-strings (in this case: as $0$-$1$
strings), and $\sigma$ is the non-trivial element in $\mathbf{C}$.

\begin{Examples}
	\begin{enumerate}[wide, label=(\roman*)]
		\item Unrefinable chains $0 \precneq m_{1} \precneq \dots \precneq m_{\ell} = n$ don't necessarily have length $\ell = \mathrm{ht}(n)$. Therefore the complicated conditions in the preceding corollary. A counterexample is given by 
		(compare \ref{Example.Irregular-k-Case-2})
		\begin{equation}
			n = (1,0,1,0,1,1,0,1,0,1).
		\end{equation}
		Then $\mathrm{ht}(n) = 3$ with $\mathrm{Sh}(n) = (X_{1}, X_{2}, X_{3})$ and unrefinable chain $0 \precneq m \precneq n$, where 
		\begin{align*}
			   X_{1} 	&= (1,0,0,0,0,1), \qquad X_{2} = (0,0,1,0,0,0,0,1), \qquad X_{3} = (0,0,0,0,1,0,0,0,0,1), \\
				    m &= (1,0,1,0,1).
		\end{align*}
		\item For $\sigma \in \mathbf{C}$ and $n \equiv 0 \pmod{q-1}$, $\mathrm{ht}(n) = \mathrm{ht}(n^{\sigma})$ by \ref{Corollary.Congruences-of-sigma-action}(iii). If $\mathbf{Sh}(n) = (X_{1}, \dots, X_{h})$
		then $(\mathbf{Sh}(n))^{\sigma} = (X_{1}^{\sigma}, \dots, X_{h}^{\sigma})$ is an $h$-composition of $n^{\sigma}$ (even monotonically increasing, as follows from \ref{Theorem.Inequalitites-for-Sheats-factors}), but in 
		general $(\mathbf{Sh}(n))^{\sigma} \neq \mathbf{Sh}(n^{\sigma})$. Here is a counterexample. Let
		\begin{equation}
			n = (0,1,0,1,1,1,0,1).
		\end{equation}
		Then $\mathbf{Sh}(n) = (X_{1}, X_{2})$, $\mathbf{Sh}(n^{\sigma}) = (Y_{1}, Y_{2}) \neq (\mathbf{Sh}(n))^{\sigma}$, with
		\[
			\begin{aligned}
				X_{1} 	&= (0,1,0,0,1,0)		& X_{1}^{\sigma}		&= (1,0,0,0,0,1) \\
				X_{2}	&= (0,0,0,1,0,1,0,1)	& X_{2}^{\sigma}	&= (0,0,1,0,1,0,1,0) \\
				Y_{1}	&= (1,0,1,0,1,0)		& Y_{2}						&= (0,0,0,0,0,1,1,0).
			\end{aligned}
		\]
		Fortunately, some assertions not only appear likely but are even true. The next result sharpens Sheats' characterization of the Sheats composition.
	\end{enumerate}
\end{Examples}

\begin{Corollary} \label{Corollary.Maximal-weight-of-sheats-i-composition}
	Let $\mathbf{r} = (r_{0}, r_{1}, \dots)$ be any weight system in the sense of \ref{Subsection.Non-vanishing-term-in-Powersum-S-i-Lambda}, that is, the $r_{j}$ are rational numbers with $0 \leq r_{0} < r_{1} < \dots$ Let $i \in \mathds{N}$, $i \leq h = \mathrm{ht}(n)$. Then the Sheats $i$-composition
	$\mathbf{Sh}^{(i)}(n) = (X_{1}, \dots, X_{i})$ of $n$ has dominant weight with respect to $\mathbf{r}$. That is,
	\begin{equation} \label{Eq.Corollary-maximal-weight-of-sheats-i-composition}
		\mathrm{wt}_{\mathbf{r}}(\mathbf{Sh}^{(i)}(n)) = \sum_{1 \leq j \leq i} r_{j-1}X_{j} > \mathrm{wt}_{\mathbf{r}}(\mathbf{Y})
	\end{equation}
	for each $i$-composition $\mathbf{Y} = (Y_{1}, \dots, Y_{i})$ of $n$ different from $\mathbf{Sh}^{(i)}(n)$.
\end{Corollary}

\begin{proof}
	For $j = 1,\dots,i$ let $n_{j} \defeq X_{1} + \dots + X_{j}$ (resp. $m_{j} = Y_{1} + \dots + Y_{j}$) be the $j$-th terms of the respective Sheats sequences. Putting $n_{0} = m_{0} = r_{-1} = 0$, with a small calculation involving telescope sums,
	\begin{equation}
		\sum_{1 \leq j \leq i} r_{j-1}X_{j} = r_{i-1}n - \sum_{1 \leq j \leq i} (r_{j-1} - r_{j-2})n_{j-1}
	\end{equation}
	and correspondingly
	\[
		\sum_{1 \leq j \leq i} r_{j-1} Y_{j} = r_{i-1}n - \sum_{1 \leq j \leq i} (r_{j-1} - r_{j-2}) m_{j-1}.
	\]
	Now the coefficients $(r_{j-1} - r_{j-2})$ are non-negative (positive for $j>1$) and $n_{j-1} \leq m_{j-1}$ by Corollary \ref{Corollary.Maximal-Sheats-compositions}, with at least two strict inequalities. Therefore 
	\eqref{Eq.Corollary-maximal-weight-of-sheats-i-composition} holds.
\end{proof}

From the above and \ref{Subsection.Non-vanishing-term-in-Powersum-S-i-Lambda} we get the following result, which generalizes \eqref{Eq.Relation-deg-of-powersum-and-weight}. (We remind the reader of Remark \ref{Remarks.Differences-to-Sheats-work}(ii)!)

\begin{Corollary} \label{Corollary.Non-vanishing-lattice-sum-formula}
	Let $\Lambda$ be a separable $\mathds{F}$-lattice of dimension larger or equal to $i$ and critical radii $q^{r_{0}}, q^{r_{1}}, \dots$ Then for $n \in \mathds{N}_{0}$ the lattice sum $S_{i, \Lambda}(n)$ vanishes if 
	$n \not\equiv 0 \pmod{q-1}$ or $\mathrm{ht}(n) < i$. Otherwise, it satisfies 
	\[
		\log S_{i, \Lambda}(n) = \mathrm{wt}_{\mathbf{r}}(\mathbf{Sh}^{(i)}(n)) = \sum_{1 \leq j \leq i} r_{j-1}X_{j}
	\]
	with $\mathbf{Sh}^{(i)}(n) = (X_{1}, \dots, X_{i})$.
\end{Corollary}

We conclude this section with another important property of the Sheats composition.

\begin{Theorem} \label{Theorem.Inequalitites-for-Sheats-factors}
	Let $n \in \mathds{N}$ be divisible by $	q-1$ with $h = \mathrm{ht}(n) \geq 2$ and $\mathbf{Sh}(n) = (X_{1}, \dots, X_{h})$. For $1 \leq i < h$ the inequality $X_{i+1} \geq qX_{i}$ holds. 
\end{Theorem}

\begin{proof}
	\begin{enumerate}[wide, label=(\roman*)]
		\item In view of \ref{Proposition.Sheats-composition-of-n-J} we may replace $n$ by $X_{i} \oplus X_{i+1}$. Thus, without restriction, \fbox{$h=2$} and $\mathbf{Sh}(n) = (a,b)$.
		\item Let $\nu = \deg_{q} n$ and $\alpha \defeq \deg_{q} a$ be the degrees with respect to $q$. Then $\nu \geq \alpha+1$. For $\nu \geq \alpha + 2$ the assertion $n > (q+1)a$ is obvious, from which $b > q\alpha$ follows.
		Thus assume \fbox{$\nu = \alpha +1$}. We will have a closer look to the $q$-adic coefficients $a_{j}$, $b_{j}$, $n_{j} \in Q$ of $a$, $b$, $n$, respectively. Write
		\begin{equation} \label{Eq.q-adic-coefficients-aj-bj-n-j}
			\begin{split}
				n	&= n^{*} + n_{\alpha}q^{\alpha} + n_{\alpha+1}q^{\alpha +1} \\
				b	&= b^{*} + b_{\alpha}q^{\alpha} + b_{\alpha+1}q^{\alpha+1} \\
				a	&= a^{*} + a_{\alpha}q^{\alpha},
			\end{split}
		\end{equation}
		where $b_{\alpha+1} = n_{\alpha+1}$, $a_{\alpha} \oplus b_{\alpha} = n_{\alpha}$, and the $q$-degrees of $n^{*}$, $b^{*}$, $a^{*}$ are less than $\alpha$. Also, write
		\begin{equation}
			n_{\alpha} = \sum_{i \in F} n_{\alpha, i} p^{i} \quad \text{with } n_{\alpha,i} \in P,
		\end{equation}
		and similarly for the other $q$-adic coefficients $n_{\alpha+1}$, $b_{\alpha}$, $a_{\alpha}$.
		\item Multiplying by a power of $p$ if necessary and using \eqref{Eq.Sheats-composition-is-multiplicative}, we may further assume that \fbox{$n_{\alpha+1, f-1} \neq 0$}.
		\item If $s \in F$ is critical for $n$, that is, $\sigma = \sigma_{s} \in \mathbf{C}$ satisfies $\ell^{\sigma}(n) = 2(q-1)$, then $B \defeq b_{\alpha, f-1}$ is subject to 
		\begin{equation}
			(B \cdot p^{f-1})^{\sigma} + n_{\alpha+1}^{\sigma} \leq q-1,
		\end{equation}
		as results from the algorithm $\mathbf{Alg}$ that calculates $b = \mathrm{Sh}(n)$. Together with \eqref{Eq.True-action-on-p-adic-expansions} and \eqref{Eq.True-action-on-sum} we find
		\[
			B \cdot p^{f-1-s} + p^{-s}(n_{\alpha+1} + (q-1)n_{\alpha+1}^{(s)}) \leq q-1,
		\]
		which gives the estimate for $B$:
		\begin{equation}
			B \leq ((q-1)(p^{s} - n_{\alpha+1}^{(s)}) - n_{\alpha+1})/p^{f-1} \eqdef R(s).
		\end{equation}
		A closer look to $\mathbf{Alg}$ shows that non-critical $s$ imply no restrictions on $B$. Hence
		\begin{align} 
			B 	&= \min( \min_{s \text{ critical}} [R(s)], n_{\alpha, f-1})		&&([\cdot] = \text{Gauß bracket}) \label{Eq.proof-characterization-of-B} \\
				&= \min( [R(s)], n_{\alpha, f-1}), 											&& s \in F \text{ minimal critical}, \nonumber
		\end{align}
		as $R(s)$ is monotonically increasing in $s$.
		\item If $a_{\alpha, f-1} < b_{\alpha+1, f-1} = n_{\alpha+1, f-1}$ (which is $\neq 0$ by (iii), then $qa < b$ and we are ready; otherwise
		\begin{align*}
			B 	&= n_{\alpha, f-1} - a_{\alpha, f-1} < n_{\alpha,f-1}, 
			\intertext{so by \eqref{Eq.proof-characterization-of-B}}
			B	&= [R(s)] \geq [R(0)] = p-1-n_{\alpha+1, f-1},
		\end{align*}
		so
		\begin{align}
			a_{\alpha,f-1} 	&= n_{\alpha, f-1} - B \leq n_{\alpha, f-1} - (p-1-n_{\alpha+1, f-1}) \\
										&= n_{\alpha+1, f-1} - (p-1 - n_{\alpha, f-1}). \nonumber
		\end{align}
		Thus, if the wanted $qa \leq b$ should fail, then $n_{\alpha, f-1} = p-1$ and $[R(s)] = [R(0)]$, which implies $s=0$.
		\item We can therefore assume: $s=0$ is critical, i.e., $\ell(n) = 2(q-1)$, $\ell(a) = \ell(b) = q-1$, and $a_{\alpha, f-1} = n_{\alpha+1, f-1}$, $n_{\alpha, f-1} = p-1$. Moreover, by the Observation \ref{Observation.On-sigmas-critical-for-n}, we
		may replace the target function $I(\cdot)$ for calculating $b = \mathrm{Sh}(n)$ with $\mathbf{Alg}$ by the sum-of-digits function $\ell(\cdot)$.
		\item Let $k \defeq \min\{ i\in F \mid a_{\alpha,i} \neq 0 \}$. Write each element $x = \sum_{i \in F} x_{i}p^{i}$ ($x_{i} \in P$) as 
		\begin{equation}
			x = x^{(k)} + \bar{x} \quad \text{with} \quad x^{(k)} = \sum_{i < k} x_{i}p^{i}, \quad \bar{x} = \sum_{i \geq k} x_{i}p^{i}.
		\end{equation}
		Now the condition $q-1 = \ell(b) \geq n_{\alpha+1} + b_{\alpha}$ implies $n_{\alpha+1, f-1} + b_{\alpha, f-1} \leq p-1$, and we have in fact equality, as $b_{\alpha, f-1}$ is maximal with that condition. This argument works since
		$\mathbf{Alg}$ first picks a $p$-part $d_{\nu}$ with value $p^{f-1} q^{\alpha}$ for $b$ before it spurns it and assigns it to $a$. Therefore, 
		\begin{IEEEeqnarray*}{rClCl}
			n_{\alpha+1, f-1} 	&+& b_{\alpha, f-1}	&=& p-1, \quad \text{and even} \\
			n_{\alpha+1, f-2}	&+& b_{\alpha, f-2}	&\leq & p-1, \\
											& \vdots & \\
			n_{\alpha+1, i}		&+& b_{\alpha, i}		&\leq& p-1,
		\end{IEEEeqnarray*}
		with equality as long as there exists some $a_{\alpha, j} \neq 0$ with $j \leq i$, that is, as long as $i \geq k$. We have thus shown:
		\begin{equation}
			n_{\alpha+1,i} + b_{\alpha,i} = p-1 \quad \text{for } i\geq k,
		\end{equation}
		and as a consequence
		\begin{equation} \label{Eq.proof.equality-for-bar-ns}
			\bar{n}_{\alpha+1} + \bar{n}_{\alpha} - \bar{a}_{\alpha} = (p-1)(p^{k} + p^{k+1} + \dots + p^{f-1}) = q - p^{k}
		\end{equation}
		and
		\begin{equation} \label{Eq.proof-inequality-a-alpha-n-alpha+1}
			a_{\alpha} <_{p} n_{\alpha+1}.
		\end{equation}
		\item If the inequality \eqref{Eq.proof-inequality-a-alpha-n-alpha+1} is proper, a view to \eqref{Eq.q-adic-coefficients-aj-bj-n-j} shows that $qa < b$. Therefore we may restrict to the case where \fbox{$a_{\alpha} = n_{\alpha+1}$} $\eqdef c \in Q$. Note that by 
		definition $c = \bar{c}$. Then by \eqref{Eq.proof.equality-for-bar-ns}:
		\begin{equation}
			\bar{n}_{\alpha} = q-p^{k} \quad \text{and} \quad b_{\alpha} = n_{\alpha}^{(k)} \oplus (\bar{n}_{\alpha} - c).
		\end{equation}
		Now $n_{\alpha}^{(k)} \leq p^{k}-1$; let $d = d^{(k)}$ be the difference $p^{k}-1 - n_{\alpha}^{(k)}$. Then
		\begin{equation} \label{Eq.proof-smooth-sum-decomposition-of-ell-b-star}
			b_{\alpha} = q-1-c-d \quad \text{and} \quad \ell(b^{*}) = q-1-c-b_{\alpha} = d.
		\end{equation}
		It follows that $b^{*}$ is composed of $p$-parts $p^{i}q^{j}$ with $0 \leq i <k$.
		
		The following table presents the relevant values of the $q$-coefficients of $n$, $b$, $a$ with indices $\alpha+1$, $\alpha$, $\alpha -1$ (as far as applicable). Each entry $x$ is given in the format $x = x^{(k)}|\bar{x}$, which 
		subdivides the columns into two parts. If the actual values differ from those in the table, then $qa < b$ holds.
		\begin{table}[h!]
			\begin{tabular}{c|ll|ll|ll}
				\toprule
							&	 \multicolumn{2}{c}{$\alpha-1$}	& \multicolumn{2}{c}{$\alpha$}	& \multicolumn{2}{c}{$\alpha+1$} \\
				\midrule
				$n$		& $p^{k}-1$		& $q-p^{k}-c$			& $p^{k}-1-d$	& $q - p^{k}$			& 0				& $c$ \\
				$b$		& $d$					& $0$							& $p^{k} -1-d$	& $q-p^{k}-c$		& 0				& $c$ \\
				$a$		& $p^{k}-1-d$	& $q-p^{k}-c$			& 0						& $c$						& 0				& 0 \\
				\bottomrule
			\end{tabular}
			\caption{}
		\end{table}
		Here the columns $\alpha + 1$ and $\alpha$ reflect the state of knowledge of (viii). 
		\item If $\alpha$ was 0 then $a = q-1$, $b \geq q(q-1)$, so we can discard this case and study the $(\alpha-1)$-column. From \eqref{Eq.proof-smooth-sum-decomposition-of-ell-b-star}, 
		$\bar{b}_{\alpha-1} = \bar{b}_{\alpha-2} = \dots = \bar{b}_{0} = 0$, so $\bar{n}_{\alpha-1} = \bar{a}_{\alpha-1}$. Further, $q-1 = \ell(a) \geq c + n_{\alpha +1}$ implies $\bar{n}_{\alpha-1} \leq q-1-c$ and in fact 
		$\bar{n}_{\alpha-1} \leq q-p^{k} - c = \bar{b}_{\alpha}$ (as $\bar{n}_{\alpha-1}$ and $c$ are divisible by $p^{k}$). If $\bar{n}_{\alpha-1} = \bar{a}_{\alpha-1} < \bar{b}_{\alpha}$ then $qa < b$; hence we can assume 
		$\bar{n}_{\alpha-1} = q-p^{k}-c$, which gives the right half of the $(\alpha-1)$-column.
		\item From $\ell(b^{*}) = d$ we find that $b_{\alpha-1} = b_{\alpha-1}^{(k)} \leq d$, so $b_{\alpha-1}^{(k)} \leq \inf_{p} (d, n_{\alpha-1}^{(k)})$, that is, for $0 \leq i < k$, $b_{\alpha-1, i} \leq \min(d_{i}, n_{\alpha-1,i})$,
		and the bound is actually attained by the nature of $\mathbf{Alg}$. Hence
		\begin{equation}
			a_{\alpha-1}^{(k)} = n_{\alpha-1}^{(k)} - b_{\alpha-1}^{(k)} \leq p^{k}-1 - d = b_{\alpha}^{(k)}.
		\end{equation}
		Again, proper inequality implies $qa < b$, so we may assume equality, which is possible only if \fbox{$n_{\alpha-1}^{(k)} = p^{k}-1$} and \fbox{$b_{\alpha-1}^{(k)} = d$}. Now the table is complete.
		\item The wanted inequality $qa \leq b$ holds certainly if $d = 0$ (here $n_{j} = b_{j} = a_{j} = 0$ for $j < \alpha -1$, and we have in fact $qa = b$); otherwise, $\alpha \geq 2$, 
		$a_{0} \oplus \dots \oplus a_{\alpha-2} = a_{0}^{(k)} \oplus \dots \oplus a_{\alpha-2}^{(k)} = d$, and so $q \sum_{0 \leq j \leq \alpha-2} a_{j}q^{j} \leq d q^{\alpha -1}$, which finally gives the assertion. \qedhere
	\end{enumerate}
\end{proof}

\begin{Remarks}
	\begin{enumerate}[label=(\roman*), wide]
		\item The inequalitiy $qX_{i} \leq X_{i+1}$ is optimal, as the trivial example $n = q^{h}-1$ shows. Here $\mathbf{Sh}(n) = (q-1, (q-1)q, \dots, (q-1)q^{h-1})$. 
		\item It would be desirable to dispose of a more conceptual proof of the theorem which avoids the displeasing case considerations and reductions in the proof actually given. The core of 
		\ref{Theorem.Inequalitites-for-Sheats-factors} is the following (perhaps known?) fact from elementary number theory: Let $n \in \mathds{N}$ be divisible by $q-1$ and such that there exists $a \in \mathds{N}$ with 
		\begin{enumerate}[label=(\arabic*)]
			\item $0 < a < n$;
			\item $a \equiv 0 \pmod{q-1}$;
			\item $\binom{n}{a} \not\equiv 0 \pmod{p}$.
		\end{enumerate}
		Then such an $a$ may be found with $(q+1)a \leq n$.
	\end{enumerate}
\end{Remarks}

\section{Location of the zeroes of $C_{k, \Lambda}$: The regular case} \label{Section.Location-of-zeroes-regular-case}

\subsection{} Having now the necessary input about the ingredients in \ref{Subsection.Laurent-expansion-of-C-k-Lambda}, we come back to that situation: 

$\Lambda$ is an infinite separable $\mathds{F}$-lattice in $C_{\infty}$ (the modifications necessary to treat the case of a
finite lattice will be given later) with critical radii $1 = q^{r_{0}} < q^{r_{1}} < \dots$ We put $\mathbf{r}$ for the sequence $(0=r_{0}<r_{1}< \dots)$. The principal and most important example is of course $\Lambda = A = \mathds{F}[T]$,
with $\mathbf{r} = \mathds{N}_{0} = (0,1,2,\dots)$. We let $C_{k} = C_{k, \Lambda}$, $G_{k} = G_{k, \Lambda}$ and use the notation introduced in Section \ref{Section.Non-archimedean-prerequisites}, where we mostly omit reference 
to the lattice $\Lambda$. 

In particular, $\mathbf{S}_{i} = \mathbf{S}(\lvert \lambda_{i} \rvert) = \{ z \in C_{\infty} \mid \log z = r_{i}\}$ is the $i$-th critical sphere of $\Lambda$. Our aim is to calculate and/or estimate the various vanishing numbers in
\ref{Subsection.Laurent-expansion-of-C-k-Lambda}.

\subsection{} Let $r \in \mathds{Q}$ satisfy $r_{i} < r < r_{i+1}$, where $i \in \mathds{N}_{0}$. The data $(\Lambda, k,i)$ will be fixed in this section. As in \ref{Subsection.Laurent-expansion-of-C-k-Lambda}, $\mathbf{S} = \mathbf{S}(q^{r})$ 
is the sphere with radius $q^{r}$, with coordinate $w = z/w_{0}$, $w_{0} \in C_{\infty}$ fixed with $\log w_{0} = r$. By \ref{Corollary.Coefficient-condition-Laurent-expansion-of-Ck}, we must study the coefficients $a_{-k-n}$ with 
negative subscripts $-k-n$ ($n>0$) of the Laurent expansion \eqref{Eq.Laurent-expansion-of-Ck-on-Sphere} of $C_{k}$ on $\mathbf{S}$,
\[
	a_{-k-n} = \binom{k-1+n}{n} w_{0}^{-k-n} S_{i+1}(n),
\]
where $S_{i+1}(n) = S_{i+1, \Lambda}(n) = \sum_{\lambda \in \Lambda_{i+1}} \lambda^{n}$ and $\Lambda_{i+1}$ is spanned by the first $i+1$ elements of a separating $\mathds{F}$-basis of $\Lambda$.

\subsection{} The binomial coefficient $\binom{k-1+n}{n}$, evaluated in $\mathds{F} \hookrightarrow C_{\infty}$, and up to sign, equals $\binom{-k}{n}$, where $-k$ as a $p$-adic number is given by the power series
\begin{equation} \label{Eq.p-adic-expansion-of-kappa}
	{-}k \eqdef \kappa = \sum_{j \geq 0} \kappa_{j} q^{j} \qquad (\kappa_{j} \in Q).
\end{equation}
Here
\begin{equation}
	(-1)^{n} \binom{k-1+n}{n} = \binom{-k}{n} = \prod_{j \geq 0} \binom{\kappa_{j}}{n_{j}},
\end{equation}
where the product on the right hand side is finite, as almost all the $q$-adic coefficients $n_{j} \in Q$ of $n$ vanish. Let $k-1$ have the $q$-adic expansion
\begin{equation}
	k-1 = \sum_{0 \leq j \leq \alpha} k_{j} q^{j},
\end{equation}
where $k_{j} \in Q$ and $\alpha \defeq \deg_{q}(k-1)$, i.e., $k_{\alpha} \neq 0$. Then
\begin{equation}
	\kappa_{j} = q-1-k_{j},
\end{equation}
since $\kappa+k-1 = -1 = \sum_{j \geq 0} (q-1)q^{j}$. Therefore, by Corollary \ref{Corollary.Non-vanishing-lattice-sum-formula} and the Lucas congruences, the coefficient $a_{-k-n}$ of \eqref{Eq.Laurent-expansion-of-Ck-on-Sphere} satisfies:
\begin{equation} \label{Eq.Condition-on-Vanishing-of-the-coefficients-a-k-n}
	a_{-k-n} \neq 0 \text{ if and only if the conditions on $n$ hold:}
\end{equation}
\begin{enumerate}[label=(\alph*)]
	\item $n \equiv 0 \pmod{q-1}$;
	\item $n <_{p} \kappa$ and
	\item $\mathrm{ht}(n) \geq i+1$.
\end{enumerate}
Here $n <_{p} \kappa$ is the coefficientwise subordination $n_{j} <_{p} \kappa_{j}$ for all $j$.

\textbf{Remark:} Instead of $q$-adic expansions we could have used $p$-adic expansions in the reasoning above, which would make no difference in view of the obvious relations between $p$-adic and $q$-adic expansions and
the formula $\binom{a}{b} \equiv \prod_{i \in F} \binom{a_{i}}{b_{i}} \pmod{p}$ for the $p$-adic coefficients $a_{i}$, $b_{i}$ of $a = \sum_{i \in F} a_{i}p^{i}$, $b = \sum_{i \in F} b_{i}p^{i}$.

\subsection{}\label{Subsection.Admissibility} As we are only interested in the absolute values of the $a_{-k-n}$, we will suppress the binomial coefficients and work with $n$'s only that satisfy the conditions (a), (b), (c) of 
\eqref{Eq.Condition-on-Vanishing-of-the-coefficients-a-k-n}. We call such $n$ \textbf{$(i+1)$-admissible}. (This depends of course on $k$, which however is fixed through the whole section.)

Secondly, for comparing the $\lvert a_{-k-n} \rvert$, the factor $\lvert w_{0} \rvert^{-k}$ is unimportant. Therefore, we define for $(i+1)$-admissible $n$:
\begin{equation}
	A_{n}^{(r)} \defeq \log( w_{0}^{-n} S_{i+1}(n)).
\end{equation}\stepcounter{subsubsection}%
Let $n$ have Sheats $(i+1)$-composition $\mathbf{Sh}^{(i+1)}(n) = (X_{1}, \dots, X_{i+1})$ with $h \defeq \mathrm{ht}(n) \geq i+1$. By Corollary \ref{Corollary.Non-vanishing-lattice-sum-formula},
\begin{equation} \label{Eq.Formula-for-A-n-r}
	A_{n}^{(r)} = \sum_{0 \leq j \leq i} r_{j}X_{j+1} - rn.
\end{equation}\stepcounter{subsubsection}%
The questions of maximality or dominance of coefficients $a_{-k-n}$ of $C_{k}$ on the sphere $\mathbf{S}$ are translated to the same questions about the $A_{n}^{(r)}$.

For $i \in \mathds{N}_{0}$, let $\mu_{i} = \mu_{i}(k)$ be the least $i$-admissible non-negative integer. We call it the $i$-\textbf{th approximation number} of $k$. In down-to-earth terms:
\begin{equation} \label{Eq.Definition-approximation-number}
	\mu_{i}(k) \defeq \min\{ n \in \mathds{N}_{0} \mid n \equiv 0 \pmod{q-1}, n<_{p} \kappa, \mathrm{ht}(n) \geq i \}.
\end{equation}\stepcounter{subsubsection}%
Consider the following conditions on the sequence $(\mu_{1}, \mu_{2}, \dots)$:
\subsubsection{} \begin{enumerate}[label=(\alph*), wide]
	\item $(\mu_{1}, \mu_{2}, \dots, \mu_{i})$ is the Sheats sequence of $\mu_{i}$;
	\item $(\mu_{1}, \mu_{2} - \mu_{1}, \dots, \mu_{i} - \mu_{i-1})$ is the Sheats composition $\mathbf{Sh}(\mu_{i})$ of $\mu_{i}$;
	\item $\mu_{1} \varprec \mu_{2} \varprec \dots \varprec \mu_{i}$.
\end{enumerate}

\begin{Lemma} \label{Lemma.Characterization-Regularity}
	The three conditions $\mathrm{(a)}$, $\mathrm{(b)}$, $\mathrm{(c)}$ are equivalent. We define $k$ to be $i$-\textbf{regular} if they are fulfilled for all $\mu_{j} = \mu_{j}(k)$, where $1 \leq j \leq i$, and $k$ is \textbf{regular} if 
	it is $i$-regular for all $i \in \mathds{N}$. (So each $k$ is $1$-regular by definition.)
\end{Lemma}

\begin{proof}
	(a) and (b) are equivalent by definition, and \enquote{(a)$\Rightarrow$(c)} is trivial. Suppose that (c) holds, and let $m_{1} \varprec m_{2} \varprec \dots \varprec m_{i} = \mu_{i}$ be the Sheats sequence of $\mu_{i}$. As the minimality condition
	on the $\mu_{j}$ is on a larger domain than the one imposed on the $m_{j}$ by Corollary \ref{Corollary.Maximal-Sheats-compositions}, we have $\mu_{j} \leq m_{j}$ for all $j$. But, since $\mu_{j} \varprec \mu_{i}$, also by 
	\ref{Corollary.Maximal-Sheats-compositions} equality must hold.
\end{proof}

\begin{Proposition} \label{Proposition.Regularity-and-Laurent}
	Suppose that $k$ is $(i+1)$-regular. Then the coefficient $a_{-k-\mu_{i+1}}$ in the Laurent expansion \eqref{Eq.Laurent-expansion-of-Ck-on-Sphere} dominates.
\end{Proposition}	

\begin{proof}
	Let $m \defeq \mu_{i+1}$ be as above and $m' > m$ any $(i+1)$-admissible number, $\mathbf{Sh}(m) = \mathbf{Sh}^{(i+1)}(m) = (X_{1}, \dots, X_{i+1})$ with $X_{j+1} = \mu_{j+1} - \mu_{j}$, $\mathbf{Sh}^{(i+1})(m') = (X_{1}', \dots, X_{i+1}')$,
	$X_{j+1}' = \mu_{j+1}' - \mu_{j}'$ and $\mu_{1}' \varprec \dots \varprec \mu_{i+1}' = m'$ the Sheats sequence of $m'$.
	We must show that 
	\[
		A_{m}^{(r)} = \sum_{0 \leq j \leq i} r_{j} X_{j+1} - rm
	\]
	is strictly larger than its counterpart
	\[
		A_{m'}^{(r)} = \sum_{0 \leq j \leq i} r_{j}X_{j+1}' - rm' \qquad (see \eqref{Eq.Formula-for-A-n-r}).
	\]
	First note that 
	\begin{equation}
		A_{m}^{(r)} > A_{m'}^{(r)} \Longleftrightarrow \sum_{0 \leq j \leq i} r_{j}(X_{j+1}' - X_{j+1}) < r(m'-m).
	\end{equation}
	Now
	\begin{align*}
		\sum_{0 \leq j \leq i} r_{j}(X_{j+1}' - X_{j+1}) 	&= \sum_{0 \leq j \leq i} (r_{j}(\mu_{i+1}' - \mu_{i+1}) - r_{j}(\mu_{j}' - \mu_{j}))
	\intertext{($\mu_{0} = \mu_{0}' = 0$; all the differences $\mu_{j}' - \mu_{j}$ are non-negative by the minimality of $\mu_{j}$)}
																					&\leq \sum_{0 \leq j \leq i} (r_{j}(\mu_{j+1}' - \mu_{j+1}) - r_{j-1}(\mu_{j}' - \mu_{j}))\\
																					&= r_{i}(\mu_{i+1}' - \mu_{i+1}) < r(m'-m),
	\end{align*}
	which was to be shown.
\end{proof}

\begin{Theorem} \label{Theorem.Regularity-and-zeroes-of-Ck-Lambda}
	Suppose that $k$ is $(i+1)$-regular for some $i \in \mathds{N}_{0}$. 
	\begin{enumerate}[label=$\mathrm{(\roman*)}$]
		\item All the zeroes $z$ of $C_{k} = C_{k, \Lambda}$ with $\log z < r_{i+1}$ lie on critical spheres $\mathbf{S}_{j} = \mathbf{S}(\lvert \lambda_{j} \rvert)$ of $\Lambda$ with $j \leq i$. That is, 
		for $j \leq i$, the numbers $\tilde{\gamma}_{\Lambda}^{(j)}(k)$ and $\tilde{\gamma}_{j, \Lambda}(k)$ of \ref{Subsection.Numbers-of-zeroes} agree. Correspondingly, $\gamma_{\Lambda}^{(j)}(k) = \gamma_{j, \Lambda}(k)$, and all the slopes $s$ of the Newton
		polygon $\mathrm{NP}(G_{k, \Lambda})$ of $G_{k, \Lambda}$ with $s > L(i+1)$ are of shape $L(j)$ for some $j \in \mathds{N}_{0}$ with $j \leq i$.
		\item The number $\gamma_{j, \Lambda}(k)$ is given for $j \leq i$ by 
		\begin{equation} \label{Eq.Formula-for-the-gamma-j-k}
			\gamma_{j}(k) = \gamma_{j, \Lambda}(k) = \frac{(q-1)k + q\mu_{j}(k) - \mu_{j+1}(k)}{q^{j+1}},
		\end{equation}
		and is independent of $\Lambda$.
		\item Suppose that $k$ is even regular ($i$-regular for all $i$). Then $\gamma_{\Lambda}(k) = \gamma(k)$ is independent of $\Lambda$.
	\end{enumerate}
\end{Theorem}

\begin{proof}
	\begin{enumerate}[wide, label=(\roman*)]
		\item Let $\mathbf{S} = \partial \mathbf{B}$ be a sphere $\mathbf{S}(q^{r})$ with $\mathbf{B} = \mathbf{B}(q^{r})$, where $r_{i} < r < r_{i+1}$. By \ref{Proposition.Regularity-and-Laurent}, $C_{k, \Lambda}$ has no zeroes on 
		$\mathbf{S}$. As this holds for each $r$ with $r_{i} < r < r_{i+1}$ and $i$-regularity implies $j$-regularity for $j < i$, (i) results.
		\item We shall apply the residue formula \eqref{Eq.Residue-formula} to $C_{k, \Lambda}$. The ingredients are:
		\[
			\sum_{x \in \mathbf{B}} \ord_{x} (C_{k, \Lambda}) = \text{number of zeroes minus number of poles of $C_{k, \Lambda}$ in $B$}
		\]
		and $\ord_{\mathbf{S}}(C_{k, \Lambda}) = -k-\mu_{i+1}(k)$ by \ref{Proposition.Regularity-and-Laurent}.
		
		Now the number of zeroes is $q^{i+1} \sum_{0 \leq j \leq i} \gamma_{j, \Lambda}(k)$ (as $\gamma_{j, \Lambda}(k) = \gamma_{\Lambda}^{(j)}(k)$ and the map $t_{\Lambda}$ from zeroes of $C_{k, \Lambda}$ to zeroes of $G_{k, \Lambda}$
		is $q^{i+1}$-to-1 on $\mathbf{B}$), and the number of poles is $k \#(\Lambda_{i}) = kq^{i+1}$. Therefore,
		\begin{equation} \label{Eq.Determining-equation-for-gamma-j-k}
			\Big( k -  \sum_{0 \leq j \leq i} \gamma_{j, \Lambda}(k) \Big) q^{i+1} = k + \mu_{i+1}(k),
		\end{equation}
		and \eqref{Eq.Formula-for-the-gamma-j-k} results from solving for $\gamma_{0, \Lambda}(k)$, $\gamma_{1, \Lambda}(k)$, \dots, $\gamma_{i, \Lambda}(k)$.
		\item If $k$ is regular, then by (i)
		\begin{equation}
			\gamma_{\Lambda}(k) = k - \gamma_{0}(k) - \gamma_{1}(k) - \dots - \gamma_{i}(k)
		\end{equation}
		for $i$ large enough. As the $\gamma_{j}(k)$ are independent of $\Lambda$, the same holds for $\gamma(k)$. \qedhere
	\end{enumerate}
\end{proof}

\begin{Corollary}
	Regardless of higher regularity of $k$, the following hold:
	\begin{enumerate}[label=$\mathrm{(\roman*)}$]
		\item The function $C_{k, \Lambda}$ has no zeroes $z$ with $0 < \log z < r_{1}$.
		\item The formula
		\begin{equation} \label{Eq.Formula-for-gamma-0-k}
			\gamma_{0}(k) = \frac{(q-1)k - \mu_{1}(k)}{q}
		\end{equation}
		is valid.
	\end{enumerate}
\end{Corollary}

\begin{proof}
	As $k$ is $1$-regular, this is the special case $i=0$ of the theorem.
\end{proof}

\begin{Remarks}
	\begin{enumerate}[label=(\roman*), wide]
		\item Formula \eqref{Eq.Formula-for-gamma-0-k} is an alternative for the one given in Proposition 1.14. The underlying relation between the numbers $\bar{j}(k)$ and $\mu_{1}(k)$ seems however less accessible to direct proof than 
		\eqref{Eq.Formula-for-gamma-0-k} and 1.14.
		\item Later (see Corollary \ref{Corollary.Approximation-numbers-in-cases}) we will see that $k$ is regular once it is $i$-regular for $i = \deg_{q}(k-1)$ (or even smaller $i$).
		\item Suppose that $k$ is regular. As the approximation numbers $\mu_{i}(k)$ determine the $\gamma_{i}(k)$, they also determine the segments of $\mathrm{NP}(G_{k, \Lambda})$, that is, the abscissas of its break points.
		It is only the ordinates of the break points which depend on $\Lambda$, that is, on the critical radii of $\Lambda$.
		\item In the case where $q=p$, the approximation numbers $\mu_{i}(k)$ collapse to the numbers $\lambda_{i}(k)$ of \cite{Gekeler13-2} 5.11, or the $\sigma_{i}(k)$ of \cite{Gekeler13} 2.12, which have an easy description.
		This is considerably more complicated in the non-prime case $q = p^{f}$ ($f \geq 2$), and is the topic of the next section.
	\end{enumerate}
\end{Remarks}	

\section{Auxiliary results} \label{Section.Auxiliary-results}

As we have seen, the approximation numbers $\mu_{i}(k)$ have decisive influence on the location of the zeroes of $G_{k, \Lambda}$. Therefore we now develop an algorithm to determine $\mu_{i}(k)$.

\subsection{} Let $k \in \mathds{N}$ be given and $-k = \kappa = \sum \kappa_{j} q^{j}$ as in \eqref{Eq.p-adic-expansion-of-kappa}. From \eqref{Eq.Definition-approximation-number} and Theorem \ref{Theorem.Characterization-indicator-function},
\begin{equation} \label{Eq.New-characterization-approximation-numbers}
	\mu_{i}(k) = \min\{ n \in \mathds{N}_{0} \mid n \equiv 0 \pmod{q-1}, n <_{p} \kappa, I(n) \geq (q-1)i \}.
\end{equation}
Here $I(n) = \min_{\sigma \in \mathbf{C}} \ell^{\sigma}(n)$ is the indicator function of \ref{Definition.Indicator-function}. Along the lines of \ref{Subsection.Algorithm-and-characterization-of-height} and \ref{Algorithm.To-determine-Sheats-factor}, we consider 
\begin{equation} \label{Eq.Kappa-as-sum-of-pnu-j}
	\kappa = \sum_{j \geq 0} a_{j}p^{j} = \sum_{\nu \geq 1} p^{j_{\nu}} \qquad (a_{j} \in P, a_{j} = p-1 \text{ for } j \gg 0)
\end{equation}
as the sum of its $p$-parts $d_{\nu} = p^{j_{\nu}}$, but now in increasing order, $p^{j_{1}} \leq p^{j_{2}} \leq \dots$, where precisely $a_{j}$ many $j_{\nu}'s$ take the value $j \in \mathds{N}_{0}$. We note that $I$ is strictly monotonically
increasing with respect to partial sums of \eqref{Eq.Kappa-as-sum-of-pnu-j}, in the sense that 
\begin{equation}
	I\Big( \sum_{1} d_{\nu} \Big) < I\Big( \sum_{2} d_{\nu} \Big),
\end{equation}
if $\sum_{1}$ is a proper partial sum of the finite sum $\sum_{2}$.

\subsection{} Given $i \in \mathds{N}$, let $\bar{\nu}$ be the least $\nu$ with 
\begin{equation}
	I\Big( \sum_{1 \leq \nu \leq \bar{\nu}} d_{\nu} \Big) \geq (q-1)i.
\end{equation}
Put
\begin{equation}
	g_{i} = g_{i}(k) = \sum_{1 \leq \nu \leq \bar{\nu}} d_{\nu}.
\end{equation}
Then $\mathrm{ht}(g_{i}) \geq i$, and in fact $\mathrm{ht}(g_{i}) = i$ since $\ell^{\sigma}(g_{i}) < (q-1)(i+1)$ if $\sigma$ is critical for $g_{i}$. If $I(g_{i}) = (q-1)i$ then $g_{i} \equiv 0 \pmod{q-1}$ and thus
$g_{i} = \mu_{i}(k)$. Otherwise, by Theorem \ref{Theorem.Characterization-indicator-function} there exists some $p$-part $d_{\nu} = p^{j_{\nu}}$ of $g_{i}$ such that $g_{i}' \defeq g_{i} - d_{\nu}$ still satisfies $\mathrm{ht}(g_{i}') \geq i$.
This suggests the following procedure to determine $\mu_{i}(k)$:

\begin{Algorithm}[to determine $\mu_{i}(k)$] ~ \label{Algorithm.Determination-of-mu-i-k}
	\begin{enumerate}[label=$\mathrm{(\arabic*)}$] \setcounter{enumi}{-1}
		\item begin
		\item $m \defeq g_{i}(k)$, $\nu \defeq \bar{\nu}$;
		\item if $I(m) = (q-1)i$ then ($\mu_{i}(k) \defeq m$, end);
		\item $\nu \defeq \nu-1$, if $I(m-d_{\nu}) \geq (q-1)i$ then $m \defeq m-d_{\nu}$;
		\item goto (2).
	\end{enumerate}
\end{Algorithm}

\begin{Proposition}
	The algorithm terminates and outputs $\mu_{i}(k)$.
\end{Proposition}

\begin{proof}
	As long as the algorithm runs, $I(m) \geq (q-1)i$. If at some value of $\nu$ we have $I(m) > (q-1)i$, then there exists $\nu' < \nu$ with $I(m-d_{\nu'}) \geq (q-1)i$. As replacing $m$ with $m - d_{\nu'}$ strictly lowers $I$, we have
	equality at some point, so the algorithm terminates, and delivers some $i$-admissible $m$. By definition, $\mu_{i}(k) \leq m$. Assume that $\mu_{i}(k) < m$. Then there exists some $p$-part $d_{\nu}$ that belongs to $m$, but not
	to $\mu_{i}(k)$, considered as partial sums of $g_{i}(k)$. Let $\nu$ be maximal with that property. Then, writing $m[\nu]$ for the current value of $m$ at the point $\nu$, $\mu_{i}(k) <_{p} m[\nu] - d_{\nu}$ and thus
	$(q-1)i = I(\mu_{i}(k)) \leq I(m[\nu] - d_{\nu})$, which conflicts with $d_{\nu}$ belonging to $m$. Hence the final value of $m$ is $\mu_{i}(k)$ as asserted.
\end{proof}

\begin{Example}
	\begin{enumerate}[label=(\roman*), wide]
		\item Let $\alpha \defeq \deg_{q}(k-1)$, such that $\kappa = \sum \kappa_{j} q^{j}$ satisfies $\kappa_{j} = q-1$ for $j > \alpha$. Assume that $\ell \defeq \sum_{0 \leq j \leq \alpha} \kappa_{j} \leq q-1$.
		Then $\mu_{1}(k) = \sum_{0 \leq j \leq \alpha} \kappa_{j}q^{j} + (q-1-\ell)q^{\alpha+1}$.
		\item Let $\alpha=1$, so $\ell = \kappa_{0} + \kappa_{1}$. If $\ell \leq q-1$ then $\mu_{1}(k) = \kappa_{0} + \kappa_{1}q + (q-1-\kappa_{0}-\kappa_{1})q^{2}$. If $\ell > q-1$ then $\mu_{1}(k) = \kappa_{0} + bq + (q-1-\kappa_{0}-b)q^{2}$,
		where $b$ is the largest non-negative integer that satisfies $b <_{p} \kappa_{1}$ and $b \leq q-1-\kappa_{0}$.
	\end{enumerate}
\end{Example}

\begin{proof}
	\begin{enumerate}[label=(\roman*), wide]
		\item Follows by formally applying and evaluating the algorithm \ref{Algorithm.Determination-of-mu-i-k}, but may be seen directly: The given number $m$ apparently has $\ell(m) = q-1$, and is minimal among all
		$n <_{p} \kappa$ with $\ell(n) = q-1$. Therefore it must agree with $\mu_{1}(k)$.
		\item The first case is a specialization of (i). For the second case, consider the expansion $\mu_{1}(k) = \sum_{j \geq 0} m_{j}q^{j}$, $m_{j} \in Q$. Then a priori, $m_{0} = \kappa_{0}$ and $m_{1} <_{p} \kappa_{1}$. Further,
		as $\ell^{\sigma}(\kappa_{0} + \kappa_{1}q + (q-1)q^{2}) \geq q-1$ for each $\sigma \in \mathbf{C}$, the $g_{1}(k)$ in the algorithm satisfies $g_{1}(k) \leq \kappa_{0} + \kappa_{1}q + (q-1)q^{2}$. In particular, $\ell(\mu_{1}(k)) = q-1$
		or $2(q-1)$, but not larger. The least $m' \in \mathds{N}$ with $\ell(m') = 2(q-1)$ and $m' <_{p} \kappa$ is obviously
		\[
			m' = \kappa_{0} + \kappa_{1}q + (2q-2 - \kappa_{0} - \kappa_{1})q^{2},
		\]
		while the least $m$ with $\ell(m) = q-1$ and $m <_{p} \kappa$ is 
		\[
			m = \kappa_{0} + bq + (q-1-\kappa_{0}-b)q^{2},
		\]
		where $b \leq q-1-\kappa_{0}$, $b <_{p} \kappa_{1}$, and is maximal with these constraints. Now, as $m$ is always less than $m'$, we must have $m = \mu_{1}(k)$. \qedhere
	\end{enumerate}
\end{proof}

The example shows that even $\mu_{1}(k)$, whose critical condition is $I(m) = q-1$, might be difficult to evaluate, depending on the complexity of the $p$-adic number $-k = \kappa = \sum \kappa_{j}q^{j}$,

\subsection{} We now give estimates on some numbers related to the Sheats composition of $n \in (q-1)\mathds{N}$. These are needed to describe the zeroes of $C_{k, A}$ and $G_{k,A}$ in the case of
irregular $k$. Until the end of the section, we make the \textbf{Assumption}
\begin{equation}
	\text{The weight system $\mathbf{r}$ equals $\mathds{N}_{0} = (0,1,2,\dots)$.}
\end{equation}
Let $n \in \mathds{N}$ be divisible by $q-1$ with $\mathbf{Sh}(n) = (X_{1}, \dots, X_{h})$, where $h = \mathrm{ht}(n)$. We define the following quantities, where we omit the subscript $\mathds{N}_{0}$ in 
$\mathrm{wt} = \mathrm{wt}_{\mathds{N}_{0}}$: The \textbf{weight} of $n$,
\begin{equation}
	\mathrm{wt}(n) \defeq \mathrm{wt}(\mathbf{Sh}(n)) = \sum_{1 \leq j \leq h} (j-1)X_{j};
\end{equation}
for $i \in \mathds{N}$ with $i \leq h$ the $i$-\textbf{th coweight}
\begin{equation}
	\mathrm{ct}^{(i)}(n) \defeq in - \mathrm{wt}(\mathbf{Sh}^{(i)}(n)) = iX_{1} + (i-1)X_{2} + \dots + 2X_{i-1} + (X_{i} + X_{i+1} + \dots + X_{h})
\end{equation}
and the (absolute) \textbf{coweight} $\mathrm{ct}(n) \defeq \mathrm{ct}^{(h)}(n)$; the \textbf{reduced $i$-th coweight}
\begin{equation}
	\mathrm{ct}_{i}(n) \defeq (i-1)n - \mathrm{wt}(\mathbf{Sh}^{(i)}(n)) = \mathrm{ct}^{(i)}(n) - n.
\end{equation}
Then for example
\begin{equation}
	n = \mathrm{ct}^{(1)}(n) < \mathrm{ct}^{(2)}(n) < \dots < \mathrm{ct}^{(h)}(n) = \mathrm{ct}(n).
\end{equation}
Another important property is: Let $m = X_{1} + \dots + X_{h-1}$ be the next-to-last term of the Sheats sequence of $n$, then
\begin{equation} \label{Eq.Relation-between-different-weights}
	\mathrm{ct}_{h}(n) = \mathrm{ct}^{(h-1)}(m) = \mathrm{ct}(m).
\end{equation}
While these are not related to the possible relation $n <_{p} \kappa$, the next result (which is the key to describing irregular zeroes of $G_{k, A}$ in the next section) requires this condition. Recall that in the present framework,
$n$ is $h$-admissible if and only if $n \equiv 0 \pmod{q-1}$, $n <_{p} \kappa$, and $\mathrm{ht}(n) \geq h$.

\begin{Proposition} \label{Proposition.Admissibility-and-weights}
	Let $n$ be $h$-admissible with $n > \mu_{h}(k)$. Then $\mathrm{ct}^{(h)}(n) > \mathrm{ct}(\mu_{h}(k))$.
\end{Proposition}

\begin{proof}
	Let $\mathbf{Sh}(\mu_{h}(k)) = \mathbf{Sh}^{(h)}(\mu_{h}(k)) = (X_{1}, \dots, X_{h})$ and $\mathbf{Sh}^{(h)}(n) = (Y_{1}, \dots, Y_{h})$ be the Sheats $h$-compositions of $\mu_{h}(k)$ and $n$, respectively. Put 
	$m \defeq \sup_{p} (\mu_{h}(k), n)$, where $\sup_{p}$ is the supremum with respect to \enquote{$<_{p}$}, $X_{h+1} \defeq m - \mu_{h}(k)$, and $Y_{h+1} \defeq m-n$. Then
	\begin{equation} \label{Eq.Proof-equality-of-smooth-sums-of-Xj-and-Yj}
		\bigoplus_{1 \leq j \leq h+1} X_{j} = m = \bigoplus_{1 \leq j \leq h+1} Y_{j},
	\end{equation}
	and both of them are \textbf{valid compositions} of length $h+1$ of $m$ in the sense of Sheats \cite{Sheats98} 1.4. (This is a more general notion of composition than ours given in \eqref{Eq.Powersum-S-i-Lambda}, in that it allows
	the last factor $X_{h+1}$ resp. $Y_{h+1}$ to be $\not\equiv 0 \pmod{q-1}$. Therefore it may be applied to the present $m$, which in general is not divisible by $q-1$.) Now the valid composition
	$(X_{1}, \dots, X_{h}, X_{h+1})$ for $m$ agrees with Sheats' greedy or optimal composition (loc. cit.), due to the characterizing minimality property of $\mu_{h}(k) = \sum_{1 \leq j \leq h} X_{j}$. Hence, by Sheats' Theorem 1.2,
	\begin{equation} \label{Eq.Proof-strict-inequality-of-Xjsums-over-Yj}
		\sum_{1 \leq j \leq h+1} (j-1)X_{j} > \sum_{1 \leq j \leq h+1} (j-1) Y_{j}.
	\end{equation}
	(Sheats' theorem gives directly $\sum_{j} X_{j} > \sum_{j} Y_{j}$, which by \eqref{Eq.Proof-equality-of-smooth-sums-of-Xj-and-Yj} is equivalent with \eqref{Eq.Proof-strict-inequality-of-Xjsums-over-Yj}.) Now substracting $h$
	times the equation \eqref{Eq.Proof-equality-of-smooth-sums-of-Xj-and-Yj} from \eqref{Eq.Proof-strict-inequality-of-Xjsums-over-Yj} and multiplying by $-1$ gives
	\[
		\sum_{1 \leq j \leq h} (h+1-j) X_{j} < \sum_{1 \leq j \leq h} (h+1-j)Y_{j},
	\]
	which is nothing else than the wanted inequality $\mathrm{ct}(\mu_{h}(k)) < \mathrm{ct}^{(h)}(n)$.
\end{proof}

\section{The irregular zeroes of $G_{k} = G_{k,A}$} \label{Section.Irregular-zeroes}

We come back to the situation of Section \ref{Section.Location-of-zeroes-regular-case}, but now drop the requirement of $k$ being regular. On the other hand, we keep the assumption \eqref{Eq.New-characterization-approximation-numbers} that the weight system is $\mathds{N}_{0} = (0,1,2,\dots)$. Accordingly, $\Lambda$
is the lattice $A$, $C_{k} = C_{k,A}$ and $G_{k} = G_{k,A}$.

\subsection{} Amending our notation hitherto, we define 
\begin{equation}
	\begin{split}
		\tilde{\gamma}_{r} 	= \tilde{\gamma}_{r}(k) 	&\defeq \# \{ \text{zeroes $z$ of $C_{k}$} \mid \log z = r \} \text{ and} \\
		\gamma_{r}				= \gamma_{r}(k)	&\defeq \tilde{\gamma}_{r}/q^{i+1} \quad \text{if } i \leq r < i+1, i \in \mathds{N}_{0},
	\end{split}
\end{equation}
so $\gamma_{r} = \# \{ \text{zeroes $x$ of $G_{k}$} \mid \log x = L(r) \}$.

Fix $i \in \mathds{N}_{0}$ and let 
\begin{equation} \label{Eq.Set-R-of-radii}
	\mathbf{R} = \mathbf{R}(k,i) = \{ \rho_{1}, \dots, \rho_{t} \}
\end{equation}
be the set $\{ r \in \mathds{Q} \mid i < r < i+1 \text{ and } \tilde{\gamma}_{r}(k) > 0 \}$, ordered such that $\rho_{1} < \rho_{2} < \dots < \rho_{t}$. Recall that by Theorem \ref{Theorem.Regularity-and-zeroes-of-Ck-Lambda}, 
$\mathbf{R}$ is empty if $k$ is $(i+1)$-regular. We call zeroes $z$ of $C_{k}$ and the corresponding zeroes $x=t_{A}(z)$ of $G_{k}$ \textbf{regular} if $\log z \in \mathds{N}_{0}$, and \textbf{irregular} otherwise; so $\mathbf{R}$ is about irregular zeroes. Put further $\rho_{0} \defeq i$, $\rho_{t+1} \defeq i+1$.

\subsection{} For each $s$ with $0 \leq s \leq t$ let $\mathbf{B}_{s}$ be a ball of radius $q^{r_{s}}$, where $\rho_{s} < r_{s} < \rho_{s+1}$. Then the equation analogous with \eqref{Eq.Determining-equation-for-gamma-j-k} holds:
\begin{equation} \label{Eq.vartheta-k-i-s}
	\Big( k - \sum_{0 \leq j < i} \gamma^{(j)}(k) - \sum_{0 \leq u \leq s} \gamma_{\rho_{u}}(k) \Big) q^{i+1} = k + \vartheta(k,i,s).
\end{equation}\stepcounter{subsubsection}%
Here $\vartheta(k,i,s) \in \mathds{N}$ is such that $-k-\vartheta(k,i,s)$ is the subscript of the dominating term in the Laurent expansion \eqref{Eq.Laurent-expansion-of-Ck-on-Sphere} of $C_{k}$ on $\partial \mathbf{B}_{s}$, 
and the left hand side of \eqref{Eq.vartheta-k-i-s} comes in exactly the same way as \eqref{Eq.Determining-equation-for-gamma-j-k} from counting zeroes and poles of $C_{k}$ in $\mathbf{B}_{s}$. Note that $\vartheta(k,i,s)$ 
is strictly decreasing in $s$, since the number of zeroes inside $\mathbf{B}_{s}$ grows with $s$.

\begin{Lemma} \label{Lemma.vartheta-k-i-s-and-mu-i+1-k}
	The number $\vartheta(k,i,t)$ equals $\mu_{i+1}(k)$.
\end{Lemma}

\begin{proof}
	Let $\mathbf{Sh}(\mu_{i+1}(k)) = (X_{1}, \dots, X_{i+1})$ and $n > \mu_{i+1}(k)$ be any other $(i+1)$-admissible natural number with $\mathbf{Sh}^{(i+1)}(n) = (Y_{1}, \dots, Y_{i+1})$. We must show that the Laurent coefficient
	$a_{-k-\mu_{i+1}}$ in the expansion \eqref{Eq.Laurent-expansion-of-Ck-on-Sphere} of $C_{k}$ on $\partial \mathbf{B}_{t}$ is strictly larger than $a_{-k-n}$ (here and in the following we briefly write $\mu_{i+1}$ for $\mu_{i+1}(k)$).
	As in the proof of Proposition \ref{Proposition.Regularity-and-Laurent}, this amounts to showing that 
	\[
		A_{\mu_{i+1}}^{(r_{t})} = \sum_{1 \leq j \leq i+1} (j-1)X_{j} - r_{t} \mu_{i+1} > A_{n}^{(r_{t})} = \sum_{1 \leq j \leq i+1} (j-1) Y_{j} - r_{t}n,
	\]
	which is equivalent with 
	\begin{equation}
		\sum (j-1) Y_{j} - \sum(j-1) X_{j} < r_{t}(n - \mu_{i+1}).
	\end{equation}
	As we can choose $r_{t}$ arbitrarly close to $\rho_{t+1} = i+1$ in $\mathds{Q} \cap (\rho_{t}, \rho_{t+1})$, it suffices to show 
	\begin{equation}
		\sum_{1 \leq j \leq i+1} (j-1)Y_{j} - \sum_{1 \leq j \leq i+1} (j-1) X_{j} < (i+1)(n - \mu_{i+1}).
	\end{equation}
	But the latter is another way of stating the inequality $\mathrm{ct}^{(i+1)}(n) > \mathrm{ct}(\mu_{i+1})$, which has been shown in \ref{Proposition.Admissibility-and-weights}.
\end{proof}

\subsection{} \label{Subsection.Equations-for-gamma-i} As in \ref{Theorem.Regularity-and-zeroes-of-Ck-Lambda}, we may solve the system \eqref{Eq.vartheta-k-i-s} for the $\gamma_{\rho_{s}}$ and obtain
\begin{align}
	\gamma_{i} 				&= \gamma_{\rho_{0}} = k - \gamma^{(0)}(k) - \dots - \gamma^{(i-1)}(k) - (k+ \vartheta(k,i,0))/q^{i+1}; \\
	\gamma_{\rho_{s}}	&= (\vartheta(k,i,s-1) - \vartheta(k,i,s))/q^{i+1} \qquad (1 \leq s \leq t); \label{Eq.Gamma-rho-s-and-varthetas}\\
	\gamma^{(i)}				&= \sum_{0 \leq s \leq t} \gamma_{\rho_{s}} = \frac{(q-1)k + q\mu_{i} - \mu_{i+1}}{q^{i+1}}. \label{Eq.Formula-for-gamma-(i)}
\end{align}
Note that \eqref{Eq.Formula-for-gamma-(i)} collapses to \eqref{Eq.Formula-for-the-gamma-j-k} if $k$ is $(i+1)$-regular.

\begin{Corollary}
	The approximation number $\mu_{i}(k)$ satisfies 
	\[
		\mu_{i}(k) + k \equiv 0 \pmod{q^{i}}, \qquad i \in \mathds{N}.
	\]
\end{Corollary}

\begin{proof}
	For $k$ $i$-regular, this already follows from the proof of Theorem \ref{Theorem.Regularity-and-zeroes-of-Ck-Lambda}, see \eqref{Eq.Determining-equation-for-gamma-j-k}. For the other $k$, it results from \eqref{Eq.vartheta-k-i-s} in conjunction with Lemma \ref{Lemma.vartheta-k-i-s-and-mu-i+1-k}.
\end{proof}

For the rest of the section, we simplify notation and write $\vartheta(s)$ for $\vartheta(k,i,s)$.

\subsection{}\label{Subsection.rhos-and-weights} The $\rho_{s}$ ($1 \leq s \leq t$) or rather the $q^{\rho_{s}}$ are critical radii of the Laurent expansion of $C_{k}$, where the coefficients $a_{-k-\vartheta(s-1)}$ and $a_{-k-\vartheta(s)}$ 
have the same size. That is, the quantities introduced in \eqref{Eq.Formula-for-A-n-r}
\begin{align*}
	A_{\vartheta(s-1)}^{(\rho_{s})}	&= \sum_{1 \leq j \leq i+1} (j-1) X_{j} - \vartheta(s-1) \rho_{s} \eqdef A^{(s)} \qquad \text{and} \\
	A_{\vartheta(s)}^{(\rho_{s})}		&= \sum_{1 \leq j \leq i+1} (j-1) Y_{j} - \vartheta(s) \rho_{s} \eqdef B^{(s)}
\end{align*}
agree, where $\mathbf{Sh}^{(i+1)}(\vartheta(s-1)) = (X_{1}, \dots, X_{i+1})$ and $\mathbf{Sh}^{(i+1)}(\vartheta(s)) = (Y_{1}, \dots, Y_{i+1})$. Now
\begin{align*}
	-A^{(s)}	&= \vartheta(s-1) (\rho_{s} - i) + \mathrm{ct}_{i+1}(\vartheta(s-1))
\intertext{and accordingly,}
	-B^{(s)}	&= \vartheta(s)(\rho_{s} - i) + \mathrm{ct}_{i+1}(\vartheta(s)),
\end{align*}
which, as $\vartheta(s-1) > \vartheta(s)$, implies
\begin{equation}
	\mathrm{ct}_{i+1}(\vartheta(s-1)) < \mathrm{ct}_{i+1}(\vartheta(s)).
\end{equation}
This is somewhat counter-intuitive, and thereby underlines the exceptional character of such a situation.

\subsection{}\label{Subsection.Properties-vartheta-k-i-s} The numbers $\vartheta = \vartheta(k,i,s)$, where $0 \leq s \leq t$, satisfy: They are $(i+1)$-admissible, i.e.,
\begin{equation} \label{Eq.Property-1-vartheta-k-i-s}
	\vartheta \equiv 0 \pmod{q-1}, \qquad \vartheta <_{p} \kappa, \qquad \mathrm{ht}(\vartheta) \geq i+1,
\end{equation}
and moreover
\begin{align}
	\vartheta		&\leq (k - \gamma^{(0)}(k) - \dots - \gamma^{(i-1)}(k)) q^{i+1} - k \eqdef \Gamma(k,i), \label{Eq.Property-2-vartheta-k-i-s} \\
	\vartheta+k	&\equiv 0 \pmod{q^{i+1}}, \label{Eq.Property-3-vartheta-k-i-s}
\end{align}
where the last assertion follows from \eqref{Eq.vartheta-k-i-s}.

We let $\boldsymbol{\Theta} \defeq \boldsymbol{\Theta}(k,i)$ be the set of all $\vartheta \in \mathds{N}$ subject to \eqref{Eq.Property-1-vartheta-k-i-s}-\eqref{Eq.Property-3-vartheta-k-i-s}. 

\subsection{} The sequence $(\vartheta(s)))_{0 \leq s \leq t}$ has its entries in $\boldsymbol{\Theta}$ and is strictly monotonically decreasing in $s$, while $(\mathrm{ct}_{i+1}(\vartheta(s)))$ is increasing. Consider the points
\begin{equation}
	P_{s} \defeq ( \vartheta(s), \mathrm{ct}_{i+1}(\vartheta(s)) )
\end{equation}
in the euclidean plane and the polygon $\mathbf{P}(k,i)$ connecting them. 
\begin{figure}[h!]
	\centering
	\begin{tikzpicture}[scale=0.75]
		\draw[->, thick] (0,0) -- (0,5.5) node[left] {$\mathrm{ct}_{i+1}(n)$}; 
		\draw[->, thick] (0,0) -- (14.5,0) node[below] {$n$};
		
		\draw[dotted] (3,0) -- (3,5) node[above right] {$P_{t}$};
		\draw[dotted] (7,0) -- (7,2.5) node[above right] {$P_{s}$};
		\draw[dotted] (9,0) -- (9,1.75) node[above right] {$P_{s-1}$};
		\draw[dotted] (11,0) -- (11,1.5) node[above right] {$P_{0}$};
		
		\draw (3,5) -- (7,2.5) -- (9,1.75) -- (11,1.5);
		
		\draw (3,0.125) -- (3,-0.125) node[below] {$\mu_{i+1}(k) = \vartheta(k,i,t)$};
		\draw (7,0.125) -- (7,-0.125) node[below] {$\vartheta(k,i,s) = \vartheta(s)$};
		\draw (11,0.125) -- (11,-0.125) node[below] {$\vartheta(k,i,0)$};
		\draw (13,0.125) -- (13,-0.125) node[below] {$\Gamma(k,i)$};
		
		\draw[fill=black] (0,0) circle (1pt);
		\draw[fill=black] (3,5) circle (1pt);
		\draw[fill=black] (7,2.5) circle (1pt);
		\draw[fill=black] (9,1.75) circle (1pt);
		\draw[fill=black] (11,1.5) circle (1pt);
		
		\node (P) at (9,3.5) {$\mathbf{P}(k,i)$};
	\end{tikzpicture}
	\caption{} \label{Figure.Graph-of-Pi}
\end{figure}
By the equality of $A^{(s)}$ and $B^{(s)}$ in \ref{Subsection.rhos-and-weights}, the slope 
\[
	( \mathrm{ct}_{t+1}(\vartheta(s-1)) - \mathrm{ct}_{i+1}(\vartheta(s)))/(\vartheta(s-1) - \vartheta(s))
\]
between $P_{s}$ and $P_{s-1}$ equals $-(\rho_{s}-i)$, and is therefore strictly decreasing in $s = 1, \dots, t$. Hence $\mathbf{P}(k,i)$ is convex with break points $P_{1}, \dots, P_{t-1}$ (see Figure \ref{Figure.Graph-of-Pi}).
We summarize our insight in

\begin{Theorem} \label{Theorem.Irregular-zeroes-Ck}
	If there exist irregular zeroes $z$ of $C_{k}$ with $i < \log z < i+1$, that is, if the set $\mathbf{R}(k,i) = \{ \rho_{1}, \dots, \rho_{t} \}$ of \eqref{Eq.Set-R-of-radii} is non-empty, then there exists a strictly monotonically decreasing sequence
	$\vartheta(0) > \vartheta(1) > \dots > \vartheta(t) = \mu_{i+1}(k)$ of length $t$ in the set $\boldsymbol{\Theta}(k,i)$ of \ref{Subsection.Properties-vartheta-k-i-s} such that 
	\begin{equation} \label{Eq.Irregular-zeroes-Ck-and-weights}
		(\mathrm{ct}_{i+1}(\vartheta(s-1)) - \mathrm{ct}_{i+1}(\vartheta(s))) /( \vartheta(s-1) - \vartheta(s)) = i - \rho_{s}
	\end{equation}
	for $s = 1,2,\dots, t$. Here $\vartheta(s)$ is brief for $\vartheta(k,i,s)$. Conversely, assume there exists $\vartheta \in \boldsymbol{\Theta}(k,i)$ with $\mathrm{ct}_{i+1}(\vartheta) < \mathrm{ct}_{i+1}(\mu_{i+1}(k))$. Let 
	$\vartheta(0)$ be the least element $\vartheta$ of $\boldsymbol{\Theta}(k,i)$ that minimizes $\mathrm{ct}_{i+1}(\vartheta)$, and let $\mathbf{P}(k,i)$ be the lower convex hull of the point set 
	$\{ (\vartheta, \mathrm{ct}_{i+1}(\vartheta) \mid \vartheta \in \boldsymbol{\Theta}(k,i) \}$ in the euclidean plane. Number its break points, starting with $P_{0} = (\vartheta(0), \mathrm{ct}_{i+1}(\vartheta(0)))$, \dots, 
	$P_{t} = (\vartheta(t), \mathrm{ct}_{i+1}(\vartheta(t)))$, where $\vartheta(t) = \mu_{i+1}(k)$. Then the numbers $\rho_{1}, \dots, \rho_{t}$ determined by \eqref{Eq.Irregular-zeroes-Ck-and-weights} satisfy
	\[
		i = \rho_{0} < \rho_{1} < \dots < \rho_{t} < \rho_{t+1} = i+1,
	\]	
	and the term $a_{-k-\vartheta(s)}$ is a dominating term in the Laurent expansion of $C_{k}$ for $\rho_{s} < \log z < \rho_{s+1}$. The irregular zeroes $z$ of $C_{k}$ with $i < \log z < i+1$ satisfy $\log z = \rho_{s}$ for
	some $s \in \{1,2,\dots,t\}$, and their numbers are given by \eqref{Eq.Gamma-rho-s-and-varthetas}.
\end{Theorem}

\begin{proof}
	We have shown the first direction. The converse direction is straightforward, as the arguments may be reversed, using the same formulas.
\end{proof}

The following is a handy coarsening of the theorem.

\begin{Corollary} \label{Corollary.On-existence-of-irregular-zeroes}
	There exist irregular zeroes $z$ of $C_{k}$ with $i < \log z < i+1$ if and only if there exists some $\vartheta \in \boldsymbol{\Theta}(k,i)$ with $\mathrm{ct}_{i+1}(\vartheta) < \mathrm{ct}_{i+1}(\mu_{i+1}(k))$. 
\end{Corollary}

Even more handy is the following:

\begin{Corollary} \label{Corollary.Existence-of-irregular-zeroes-and-approximation-numbers}
	The existence of irregular zeroes $z$ of $C_{k}$ with $i < \log z < i+1$ is equivalent with $\mu_{i}(k) \not\varprec \mu_{i+1}(k)$.
\end{Corollary}

\begin{proof}
	Suppose that $\mu_{i} \varprec \mu_{i+1}$. Then $\mu_{i} = \mu_{i+1} - \mathrm{Sh}(\mu_{i+1})$ is the next-to-last term of the Sheats sequence of $\mu_{i+1}$. By \eqref{Eq.Relation-between-different-weights}, 
	$\mathrm{ct}_{i+1}(\mu_{i+1}) = \mathrm{ct}(\mu_{i})$; more generally, if $n$ is $(i+1)$-admissible with $\mathbf{Sh}^{(i+1)}(n) = (Y_{1}, \dots, Y_{i+1})$, then $\mathrm{ct}_{i+1}(n) = \mathrm{ct}^{(i)}(Y_{1}+Y_{2} + \dots+ Y_{i})$. 
	Now by \ref{Proposition.Admissibility-and-weights}, $\mathrm{ct}(\mu_{i})$ is minimal among all $\mathrm{ct}^{(i)}(m)$, where $m$ is $i$-admissible. Hence there are no zeroes $z$ with $i < \log z < i+1$ by the last corollary.
	
	Conversely, assume $\mu_{i} \not\varprec \mu_{i+1}$. Let $\mathbf{Sh}(\mu_{i+1}) = (X_{1}, \dots, X_{i+1})$, $\mathbf{Sh}(\mu_{i}) = (Y_{1}, \dots, Y_{i})$, and $m \defeq X_{1} + \dots + X_{i}$. Then $\mu_{i} < m$. Choose $Y_{i+1}$ such that
	$n \defeq \mu_{i} \oplus Y_{i+1}$ is $(i+1)$-admissible and $\mathbf{Sh}(n) = (Y_{1}, \dots, Y_{i}, Y_{i+1})$. Now 
	\[
		\mathrm{ct}_{i+1}(n) = in - \mathrm{wt}(n) = iY_{1} + (i-1)Y_{2} + \dots + Y_{i} = \mathrm{ct}(\mu_{i}) = \mathrm{ct}^{(i)}(\mu_{i}).
	\]
	Again by \ref{Proposition.Admissibility-and-weights}, $\mathrm{ct}^{(i)}$ is minimal on $\mu_{i}$, so 
	\begin{equation} \label{Eq.Irregular-zeroes-and-weight-inequalities}
		\mathrm{ct}_{i+1}(\mu_{i+1}) = \mathrm{ct}^{(i)}(m) > \mathrm{ct}^{(i)}(\mu_{i}) = \mathrm{ct}_{i+1}(n).
	\end{equation}
	Now assume that $n$ is minimal with the conditions above ($n$ $(i+1)$-admissible, $\mu_{i} \precneq n$, $Y_{i+1} = \mathrm{Sh}(n)$). Then \eqref{Eq.Irregular-zeroes-and-weight-inequalities} implies that the coefficient 
	$a_{-k-n}$ of the expansion of $C_{k}$ on the sphere $\mathbf{S}(q^{r})$ dominates if $r > i$ is sufficiently close to $i$ (see the calculations in \ref{Subsection.Admissibility}, notably \eqref{Eq.Formula-for-A-n-r}). 
	The residue formula gives that, in analogy with \eqref{Eq.Determining-equation-for-gamma-j-k},
	\begin{equation}
		\Big( k - \sum_{0 \leq j < i} \gamma^{(j)}(k) \Big)q^{i+1} = k+n
	\end{equation}
	holds. This implies, first, that $n \leq \Gamma(k,i)$ and, second, $k+n \equiv 0 \pmod{q^{i+1}}$, together that $n \in \boldsymbol{\Theta}(k,i)$. Now we conclude by \ref{Corollary.On-existence-of-irregular-zeroes}.
\end{proof}

\section{Determination of the vanishing numbers $\gamma(k) = \gamma_{A}(k)$} \label{Section.Determination-of-vanishing-numbers}

We keep the assumptions of the last sections; so all results and notations refer to the lattice $\Lambda = A$.

\subsection{} For arbitrary $k$, regardless of regularity, the multiplicity $\gamma(k)$ of the zero $x=0$ of $G_{k}(X) = G_{k,A}(X)$ is given by
\begin{equation} \label{Eq.Formula-for-zero-multiplicity}
	\gamma(k) = \lim_{i \to \infty} \Big( k - \sum_{0 \leq j < i} \gamma^{(j)}(k) \Big) = \lim \frac{k + \mu_{i}(k)}{q^{i}},
\end{equation}
as follows from \eqref{Eq.vartheta-k-i-s} together with \ref{Lemma.vartheta-k-i-s-and-mu-i+1-k}. We are going to give a more explicit description.

\subsection{} Define
\begin{equation}
	\begin{split}
		\alpha		&= \alpha(k) \defeq \deg_{q}(k-1) \\
		M				&=M(k) \defeq \min\{ n \in \mathds{N} \text{ admissible and $n+k \equiv 0 \pmod{q^{\alpha+1}}$} \} \\
		H				&=H(k) \defeq \mathrm{ht}(M(k)).
	\end{split}
\end{equation}
Here as usual, $n$ admissible means $n \equiv 0 \pmod{q-1}$ and $n <_{p} \kappa$, which depends on $k$. These quantities along with the approximation numbers $\mu_{i}(k)$ will play a prominent role. Recall that the 
$q$-expansions are
\begin{equation}
	k-1 = \sum_{0 \leq j \leq \alpha} k_{j}q^{j}, \qquad -k = \kappa = \sum_{j \geq 0} \kappa_{j} q^{j} \qquad (k_{j}, \kappa_{j} \in Q)
\end{equation}
with $\kappa_{j} = q-1-k_{j}$. Now the condition $M+k \equiv 0 \pmod{q^{\alpha+1}}$ means that $M$ as a \enquote{power series} agrees up to the $q^{\alpha}$-term with $\kappa$. Quite generally, for any element
$\varphi = \sum_{j \geq 0} f_{j} q^{j}$ of $\mathds{Z}_{p}$ ($f_{j} \in Q$), let $\varphi^{*}$ be the part $\sum_{j < \alpha} f_{j}q^{j}$. Then 
\begin{equation}
	M = \kappa^{*} + \kappa_{\alpha}q^{\alpha} + Rq^{\alpha+1},
\end{equation}
where $(q-1) \neq R \in Q$, so $R \in Q' = \{0,1,2,\dots,q-2\}$ is determined by 
\begin{equation}
	\ell(\kappa^{*}) + \kappa_{\alpha} + R \equiv 0 \pmod{q-1}.
\end{equation}
As $\ell(\kappa^{*}) + \kappa_{\alpha} = (\alpha+1)(q-1) - \ell(k-1)$, $R$ is the representative $\pmod{q-1}$ of $\ell(k-1)$ in $Q'$. Note also that, as $k_{\alpha} \neq 0$, $\kappa_{\alpha} \in Q'$.

\subsection{} We need some more definitions. For admissible $n$ of height $h$ with $\mathbf{Sh}(n) = (X_{1}, \dots, X_{h})$ and Sheats sequence $(n_{0}, n_{1}, \dots, n_{h})$, where $n_{i} = \sum_{1 \leq j \leq i} X_{j}$, put
\begin{equation}
	^{i}n \defeq n_{h-i}.
\end{equation}
Further, for any admissible $n$
\begin{equation}
	\begin{split}
		n^{(1)} 	&\defeq \min\{ m \in \mathds{N} \mid m \text{ admissible and $n \precneq m$} \} \\
		n^{(i)}		&\defeq (n^{(i-1)})^{(1)} \quad \text{for $i \geq 2$, $n^{(0)} = n$.}
	\end{split}
\end{equation}
Then we may formally write $\mu_{1}(k) = 0^{(1)}$, and: $k$ is regular $\Leftrightarrow$ $\forall i \in \mathds{N}_{0} : \mu_{i}(k) = 0^{(i)}$ $\Leftrightarrow$ $\forall i \forall j \leq i : {}^{j}\mu_{i}(k) = \mu_{i-j}(k)$.

\subsection{}\label{Subsection.Ms-and-mus} As $^{1}M \leq \kappa^{*} + \kappa_{\alpha}q^{\alpha}$ (this follows from \eqref{Eq.deg-of-factors-in-monotonically-increasing}), we have $\mu_{H-1}(k) \leq {}^{1}M \leq \kappa^{*} - \kappa_{\alpha}q^{\alpha}$, which implies $\mu_{H-1}(k) \lneq_{p} M$. By 
definition of the Sheats sequence and the minimality property of $\mu_{H-1}$, therefore
\begin{equation}
	\mu_{H-1} = {}^{1}M.
\end{equation}
(Here and henceforth, we suppress \enquote{$k$} in the notation.) With similar arguments, we will describe $\mu_{H}$ and the higher $\mu_{H+i}$ ($i\geq 0$). For the moment we note 
\begin{equation} \label{Eq.Formula-for-M-(1)}
	M^{(1)} = \kappa^{*} + \kappa_{\alpha} q^{\alpha} + (q-1)q^{\alpha+1} + Rq^{\alpha+2},
\end{equation}
and more generally, for $i \in \mathds{N}_{0}$,
\begin{equation} \label{Eq.Formula-for-M-(i)}
	M^{(i)} = \kappa^{*} + \kappa_{\alpha}q^{\alpha} + (q-1)q^{\alpha+1} + \dots + (q-1)q^{\alpha+i} + Rq^{\alpha+i+1}.
\end{equation}
Here \eqref{Eq.Formula-for-M-(1)} results as the right and side is $M+(q-1-R)q^{\alpha+1} + Rq^{\alpha+2}$, i.e., $M$ plus the least $X$ with $\ell(X) = q-1$ and such that $M \oplus X <_{p} \kappa$. The argument for 
\eqref{Eq.Formula-for-M-(i)} is similar.

\subsection{} In what follows, we often simplify notation and describe $q$- or $p$-expansions through strings. Then, e.g., \eqref{Eq.Formula-for-M-(i)} reads 
\[
	M^{(i)} = (\kappa^{*}, \kappa_{\alpha}, \underbrace{q-1, \dots, q-1}_{i}, R),
\]
a $q$-string, or the $p$-string
\[
	(q-1) = (\underbrace{p-1, \dots, p-1}_{f}).
\]
As in \ref{Subsection.Ms-and-mus}, $\mu_{H+i-1}$, as the minimal admissible of height $H+i-1$, is less or equal to $^{1}(M^{(i)})$, which implies $\mu_{H+i-1} \lneq_{p} M^{(i)}$, and thus 
\begin{equation} \label{Eq.Formula2-for-mu-H+i-1}
	\mu_{H+i-1} ={} ^{1}(M^{(i)}), \qquad \text{for } i \geq 0.
\end{equation}
Hence in view of \eqref{Eq.Formula-for-zero-multiplicity} we should investigate the operator $n \mapsto {}^{1}n$ on $n = M^{(i)}$, or equivalently, $n \mapsto \mathrm{Sh}(n)$, where $\mathrm{Sh}(n) = n- {}^{1}n$ is the largest 
Sheats factor of $n$.

\subsection{} Recall that some $s \in F = \{0,1, \dots, f-1\}$ or its image $\sigma_{s} \in \mathbf{C}$ is critical for $M$ if $I(M) = \min_{\sigma \in \mathbf{C}} \ell^{\sigma}(M) = \ell^{\sigma_{s}}(M)$. Since 
$\ell^{\sigma}(M^{(i)}) = \ell^{\sigma}(M) + i(q-1)$, the sets of critical $s$ for $M$ and $M^{(i)}$ agree. We call that set $F(k)$ with image $\mathbf{C}(k) \defeq \{ \sigma \in \mathbf{C} \mid \sigma \text{ critical for } $M$\}$
in $\mathbf{C}$ and put 
\begin{equation}
	\bar{s} \defeq \min \{ s \in F(k) \} = \min \{ s \in F \mid s \text{ critical for } M \}.
\end{equation}
If $\bar{s} > 0$, we split each $x = \sum_{i \in F} x_{i}p^{i} \in Q$ into the parts 
\begin{equation}
	\underline{x} \defeq \sum_{0 \leq i < \bar{s}} x_{i}p^{i} \quad \text{and} \quad \bar{x} \defeq \sum_{\bar{s} \leq i < f} x_{i}p^{i}.
\end{equation}
We distinguish the two cases:
\begin{equation}
	\begin{split}
		\text{Case 1}:		&\text{$\bar{s} = 0$, or ($\bar{s} > 0$ and $\underline{R} = (p-1, \dots, p-1) = p^{\bar{s}}-1$) or ($\bar{s} > 0$ and $\overline{R} = 0$)}; \\
		\text{Case 2}:	&\text{$\bar{s} > 0$, $\underline{R} \neq p^{\bar{s}}-1$ and $\overline{R} \neq 0$},
	\end{split}
\end{equation}

\begin{Proposition} \label{Proposition.Sheats-and-Cases}
	\begin{enumerate}[label = $\mathrm{(\roman*)}$] 
		\item Assume we are in $\mathrm{Case~1}$. Then for $i \geq 1$,
		\begin{equation} \label{Eq.Case-1-formula-for-Sh-M-(i)}
			\mathrm{Sh}(M^{(i)}) = (q-1-R)q^{\alpha +i} + Rq^{\alpha + i +1}
		\end{equation}\stepcounter{subsubsection}%
		holds. Therefore $\mu_{H+i-1} = {}^{1}(M^{(i)}) = M^{(i-1)}$.
		\item Assume $\mathrm{Case~2}$ and $i \geq 2$. Then 
		\begin{equation} \label{Eq.Sheats-composition-of-M-(i)}
			\mathrm{Sh}(M^{(i)}) = (q-\overline{R}) q^{\alpha +i-1} + (q-1 - (\underline{R} +1))q^{\alpha +i} + Rq^{\alpha + i +1}.
		\end{equation}\stepcounter{subsubsection}%
		In this case,
		\begin{equation} \label{Eq.Formula-for-mu-H+i-1}
			\mu_{H+i-1} = {}^{1}(M^{(i)}) = (\kappa^{*}, \kappa_{\alpha}, \underbrace{q-1, \dots, q-1}_{i-2}, \overline{R}-1, \underline{R}+1).
		\end{equation}\stepcounter{subsubsection}%
	\end{enumerate}
\end{Proposition}

\begin{proof}
	\begin{enumerate}[label=(\alph*), wide]
		\item The assertion of (i) is immediate if $\bar{s} = 0$ or $\bar{s} > 0$ and $R=0$, as $X \defeq (q-1-R)q^{\alpha+i} + Rq^{\alpha+i+1}$ has height 1, satisfies $\mathrm{ht}(M^{(i)} - X) = \mathrm{ht}(M^{(i-1)}) = H+i-1$, and $X$ is
		obviously maximal with these properties.
		\item We now suppose that $\bar{s} > 0$ and $R > 0$ and analyze the effect of the algorithm $\mathbf{Alg}$ on $M^{(i)}$. We will see that $\mathrm{Sh}(M^{(i)}) = aq^{\alpha + i -1} + bq^{\alpha+i} + Rq^{\alpha +i+1}$ with
		numbers $a,b \in Q$ which do not depend on $i$. A closer look will then reveal that $a$ and $b$ are as stated.
		\item The numbers 
		\[
			a = \sum_{j \in F} a_{j}p^{j}, \qquad b = \sum_{j \in F} b_{j}p^{j}\qquad (a_{j}, b_{j} \in P)
		\]
		are the coefficients of $q^{\alpha+i-1}$ (resp. $q^{\alpha+i}$) in $\mathrm{Sh}(M^{(i)})$ as output from $\mathbf{Alg}$. We first show that they suffice to describe $\mathrm{Sh}(M^{(i)})$, i.e., that 
		$\mathrm{Sh}(M^{(i)}) \equiv 0 \pmod{q^{\alpha +i-2}}$. We also use the dynamic variables $c,d \in \mathds{N}_{0}$, whose values change during execution of $\mathbf{Alg}$. In the $\nu$-th step $(1 \leq \nu \leq f)$, a certain
		value is assigned to $b_{f- \nu}$, then, for $f+1 \leq \nu \leq 2f$, to $a_{2f-\nu}$.
		\item Define the starting value of $c$ as $c[0] \defeq R$, and the value after the $\nu$-th step $(\nu \geq 1)$ by 
		\[
			c[\nu] \defeq c[\nu-1] + \begin{cases} b_{f-\nu} p^{f-\nu},	&\text{if } 1 \leq \nu \leq f, \\ a_{2f-\nu} p^{2f-\nu},	&\text{if } f+1 \leq \nu \leq 2f. \end{cases}
		\]
		As long as $\ell^{\sigma}(c) \leq q-1$ for all $\sigma \in \mathbf{C}(k)$, we define $d \defeq (\ell^{\sigma}(c))^{\sigma^{-1}} \in Q \smallsetminus \{0\} = Q'+1$, which is independent of $\sigma \in \mathbf{C}(k)$, and
		equals the representative modulo $q-1$ of $c$ in $Q' + 1$. In particular, $d[0] = c[0] = R < q-1$, as are all the values of $d$ except for the last one, when $\mathbf{Alg}$ terminates. We will show that $\mathbf{Alg}$
		actually stops for some $\nu \leq 2f$ (in fact, for $\nu \leq 2f - \bar{s}$) with final value $d = q-1$.
		\item Suppose until further notice that $1 \leq \nu \leq f$, that $b_{f-1}, \dots, b_{f-\nu+1}$ are assigned as prescribed by $\mathbf{Alg}$, and $d < q-1$. The condition on $b_{f-\nu}$ from \ref{Algorithm.To-determine-Sheats-factor} (2) is
		\begin{equation} \label{Eq.Condition-on-b-f-nu}
			\ell^{\sigma}(( b_{f- \nu}p^{f-\nu} + b_{f-\nu+1}p^{f-\nu+1} + \dots + b_{f-1}p^{f-1})q^{\alpha+i} + Rq^{\alpha+i+1}) \leq q-1
		\end{equation} \stepcounter{subsubsection}%
		for all the $\sigma \in \mathbf{C}(k)$. That is, $b_{f-\nu}$ is the maximal number in $P$ that satisfies \eqref{Eq.Condition-on-b-f-nu}. Equivalently, $b_{f- \nu}$ is maximal such that
		\begin{equation} \tag{\ref{Eq.Condition-on-b-f-nu}'} \label{Eq.Condition-on-b-f-nu-prime}
			\forall \sigma \in \mathbf{C}(k) : \ell^{\sigma}(b_{f-\nu} p^{f-\nu} + c) \leq q-1,
		\end{equation}
		where $c = c[\nu -1]$ is the current value of $c$. (Note that this condition does not depend on $i$!) Now
		\[
			\ell^{\sigma}(b_{f-\nu}p^{f-\nu} + c ) = b_{f-\nu}(p^{f- \nu})^{\sigma} + \ell^{\sigma}(c) = b_{f- \nu}(p^{f-\nu})^{\sigma} + d^{\sigma}
		\]
		(since $d$ and thus $\ell^{\sigma}(c) < q-1$), so \eqref{Eq.Condition-on-b-f-nu-prime} reads
		\begin{equation} \label{Eq.Condition-on-b-f-nu-prime-equiv}
			\forall \sigma \in \mathbf{C}(k) : b_{f-\nu}(p^{f- \nu})^{\sigma} + d^{\sigma} \leq q-1.
		\end{equation}\stepcounter{subsubsection}%
		Note that:
		\subsubsection{}\stepcounter{equation}%
		Throughout the procedure, the $d^{\sigma}$ $(\sigma \in \mathbf{C})$ are all strictly less than $q-1$, or all equal to $q-1$, by \eqref{Eq.Congruence-to-zero-implies-congruence-of-m-sigma-to-0};
		\subsubsection{}\stepcounter{equation}%
		The conditions analogous with \eqref{Eq.Condition-on-b-f-nu} for non-critial $\sigma$ may be discarded, as for such $\sigma$ still
		\[
			\ell^{\sigma}(b_{f-\nu} p^{f- \nu} + c) < 2(q-1)
		\]
		holds;
		\subsubsection{}\label{subsubsection.Growth-of-variable-c}\stepcounter{equation}%
		The variable $c$ is $\leq 2(q-1)$ for $\nu \leq 2f$. As it grows with $\nu$, we have $d[\nu] = c[\nu]$ for $\nu \leq n$ and $d[\nu] = c[\nu]-(q-1)$ for $\nu > n$ for some $n$.
		
		If, after assigning $b_{f-\nu}$ its value, the current $d = d[\nu]$ equals $q-1$, then the remaining $b_{f-\nu-1}, \dots, b_{0}, a_{f-1}, \dots, a_{0}$ are put to 0 and $\mathbf{Alg}$ stops.
		\item Now we analyze the condition of \eqref{Eq.Condition-on-b-f-nu-prime-equiv}. Using the rule \eqref{Eq.True-action-on-p-parts} to evaluate $(p^{f-\nu})^{\sigma}$ and writing $d = \sum_{j \in F} d_{j}p^{j}$ $(d_{j} \in P)$ as usual,
		condition \eqref{Eq.Condition-on-b-f-nu-prime-equiv} for any $\sigma = \sigma_{s}$, critical or not, and $t \defeq f- \nu$ reads: 
		\subsubsection{}\stepcounter{equation}%
		If \fbox{$t \geq s$}, 
		\begin{align*}
			&b_{t}p^{t-s} + d_{f-1}p^{f-1-s} + d_{f-2}p^{f-2-s} + \dots + d_{t}p^{t-s} + \dots + d_{s}p^{0} \\ 
			&\qquad + d_{s-1}p^{f-1} + \dots + d_{0}p^{f-s} \leq q-1,
		\end{align*}
		and for \fbox{$t<s$}
		\begin{align*}
			&b_{t}p^{f+t-s} + d_{f-1}p^{f-1-s} + \dots + d_{s}p^{0} \\ 
			&\qquad + d_{s-1}p^{f-1} + \dots + d_{t}p^{f+t-s} + \dots + d_{0}p^{f-s} \leq q-1.
		\end{align*}
		Taking carryovers into account, these are equivalent with:
		\subsubsection{}\label{subsubsection.Condition-for-b-f-nu-and-all-critical-s}\stepcounter{equation}%
		\fbox{$t \geq s$} There exists $j \in F$ with $j < s$ or $j > t$ such that $d_{j} < p-1$, or $b_{t} \leq p-1-d_{t}$;
		
		\fbox{$t<s$} There exists $j \in F$ with $t < j < s$ and $d_{j} < p-1$, or $p_{t} \leq p-1-d_{t}$.
		
		As $b_{t} = b_{f-\nu}$ is the maximal number in $P$ subject to \ref{subsubsection.Condition-for-b-f-nu-and-all-critical-s} for all critical $s$, it may take only the values $p-1$ or $p-1-d_{f-\nu}$, where
		$d = d[\nu-1]$ is the current value of $d$. Also, once the condition from one $\sigma \in \mathbf{C}(k)$ implies $b_{f-\nu} \leq p-1-d_{f-\nu}$, we have equality regardless of possible further conditions
		coming from other critial $\sigma$. Now it is easily checked that the condition coming from $\bar{\sigma} = \sigma_{\bar{s}}$ implies the condition for each $\sigma \in \mathbf{C}(k)$. Therefore, we may
		replace \eqref{Eq.Condition-on-b-f-nu-prime-equiv} by the single condition \ref{subsubsection.Condition-for-b-f-nu-and-all-critical-s} with $s = \bar{s}$.
		\item Suppose that $b_{f-\nu-1} = p-1-d_{f-\nu} = p-1 - d[\nu]_{f-\nu-1}$ for some $\nu$, $0 \leq \nu < f$. If $c[\nu] = d[\nu] < q-1$, then still
		$c[\nu+1] = c[\nu] + b_{f-\nu-1} p^{f-\nu-1} \leq q-1$ by construction. Hence $\nu$ differs from the bound $n$ in \ref{subsubsection.Growth-of-variable-c}, and in both possible cases ($c[\nu] = d[\nu]$ or
		$c[\nu] = d[\nu] +q-1$), $d[\nu+1] = d[\nu] \oplus b_{f-\nu-1} p^{f-\nu-1}$, which gives
		\begin{align} \label{Eq.Characterization-of-d-nu+1-f-nu-1}
			&\text{$d[\nu+1]_{f-\nu-1} = p-1$, if $b_{f-\nu-1} = p-1-d[\nu]_{f-\nu-1}$} \\
			&\qquad \text{for some $\nu$ with $0 \leq \nu < f$}. \nonumber
		\end{align}\stepcounter{subsubsection}%
		(We will later see a similar statement for $f \leq \nu < 2f$.)
		\item Now \ref{subsubsection.Condition-for-b-f-nu-and-all-critical-s} with $s = \bar{s}$ implies $b_{\bar{s}-1} = p-1-d_{\bar{s}-1}$; more precisely, $b_{\bar{s}-1} = p-1-d[\nu]_{\bar{s}-1}$ with the current value 
		$d[\nu]$ of $d$, $\nu = f-\bar{s}$. By \eqref{Eq.Characterization-of-d-nu+1-f-nu-1} this gives
		\begin{equation}
			d[f-\bar{s}-1]_{\bar{s}-1} = p-1.
		\end{equation}\stepcounter{subsubsection}%
		By descending induction on $j = \bar{s}-1, \bar{s}-2, \dots, 0$, one shows first, that $b_{j} = p-1-d[f-j-1]_{j}$ and second, that 
		\[
			d[f-j] = (\underbrace{*, \dots, *}_{j}, \underbrace{(p-1), \dots, (p-1)}_{\bar{s}-j}, \underbrace{*,\dots,*}_{f-\bar{s}})
		\]
		and in particular
		\begin{equation} \label{Eq.Characterization-of-df}
			d[f] = (d_{0}, \dots, d_{f-1}) = ( \underbrace{(p-1), \dots, (p-1)}_{\bar{s}}, *, \dots, *),
		\end{equation}\stepcounter{subsubsection}%
		i.e., $d[f]_{j} = p-1$ for $0 \leq j < \bar{s}$. 
		\item Having taken the first $f$ steps without stopping of $\mathbf{Alg}$, we come to the next steps, which assign values to $a_{f-1}, \dots, a_{0}$. There are the same conditions like 
		\ref{subsubsection.Condition-for-b-f-nu-and-all-critical-s} with $s = \bar{s}$ on the $a_{t}$, namely:
		\subsubsection{}\label{Subsubsection.Casedistinction-for-Conditions-on-a-t}\stepcounter{equation}%
		\fbox{$t \geq \bar{s}$} No condition on $a_{t}$ if $d_{j} < p-1$ for some $j \in F$ with $j < \bar{s}$ or $j > t$; otherwise $a_{t} \leq p-1-d_{t}$;
		
		\fbox{$t < \bar{s}$} No condition on $a_{t}$ if $d_{j} < p-1$ for some $j$ with $t<j<\bar{s}$; otherwise $a_{t} \leq p-1-d_{t}$ (and possible further conditions of type \eqref{Eq.Condition-on-b-f-nu} from non-critical $\sigma$
		wouldn't enforce these). Furthermore, with the reasoning of (g) we find
		\begin{align}
			&\text{$d[\nu+1]_{2f-\nu-1} = p-1$, if $a_{2f-\nu-1} = p-1-d[\nu]_{2f-\nu-1}$} \label{Eq.Formula-for-d-nu+1-2f-nu-1}\\
			 &\qquad \text{for some $\nu$ with $f \leq \nu < 2f$.} \nonumber 
		\end{align}
		\item We conclude from \ref{Subsubsection.Casedistinction-for-Conditions-on-a-t} and \eqref{Eq.Characterization-of-df} that $a_{f-1} \leq p-1-d[f]_{f-1}$. As $a_{f-1}$ is maximal with that condition, equality holds,
		which by \eqref{Eq.Formula-for-d-nu+1-2f-nu-1} gives $d[f+1]_{f-1} = p-1$. Again by descending induction on $j = f-1, f-2, \dots$, one shows $a_{j} = p-1-d[2f-j-1]_{j}$ and 
		\[
			d[2f-j] = ( \underbrace{(p-1), \dots, (p-1)}_{\bar{s}}, *, \dots, *, \underbrace{(p-1), \dots, (p-1)}_{f-j}),
		\]
		until at latest $d[2f-\bar{s}] = ( (p-1), \dots, (p-1)) = q-1$, and $\mathbf{Alg}$ stops.
		\item So far we know that $\mathrm{Sh}(M^{(i)}) = aq^{\alpha+i-1} + bq^{\alpha+i} + Rq^{\alpha + i +1}$ with $a,b$ independent of $i$, as announced in (b). 
		
		Suppose that \fbox{$\overline{R} = 0$}. Then, as $d[0] = c[0] = R$,
		$d[0]_{j} = 0$ for $\bar{s} \leq j < f$, and we see by descending induction on $j$ that $b_{j} = p-1-R_{j}$, so $b = q-1-R$, $\mathbf{Alg}$ stops at latest at the $f$-th step, and $a=0$.
		
		Now suppose that \fbox{$\underline{R} = p^{\overline{s}}-1$} $= ( \underbrace{(p-1), \dots, (p-1)}_{\bar{s}}, 0, \dots, 0)$. Then $b_{f-1} = p-1-R_{f-1}$, and (by induction, as usual), $b_{j} = p-1-R_{j}$ for
		$j = f-1, \dots, 0$, and so $b= q-1-R$, $a=0$ as before.
		
		This completes the proof of assertion (i) of the proposition and we may now exclude these two subcases of Case 1.
		\item Let $D$ be the value $d[\nu]$ for $\nu = f- \bar{s}$, $D=(d_{0}, \dots, d_{f-i})$. As (h) shows, we get in the $(\nu+1)$-th step $b_{\bar{s}-1} = p-1-d_{\bar{s}-1}$, and
		$d[\nu+1]$ is obtained from $D = d[\nu]$ by putting the $(\bar{s}-1)$-entry to $p-1$. More generally, for $j = \bar{s}-1$ down to $j=0$, $b_{j} = p-1-d[f-j-1]_{j}$, and $d[f-j]$ is obtained from 
		$d[f-j-1]$ by putting the $j$-th entry to $p-1$. Hence
		\begin{equation} \label{Eq.Characterizing-equation-for-b-bar-and-D-bar}
			\underline{b} + \underline{D} = p^{\overline{s}} - 1.
		\end{equation}
		\item Let $\bar{j} \in F$ be maximal with $R_{j} \neq 0$; then $\bar{j} \geq \bar{s}$ by the assumption $\overline{R} \neq 0$. $\mathbf{Alg}$ produces $b_{f-1} = \dots = b_{\bar{j}} = p-1$,
		\begin{align*}
			c[f-\bar{j}-1]		&= (R_{0}, \dots, R_{\bar{j}-1}, R_{\bar{j}}, (p-1), \dots, (p-1))< q-1 \quad \text{and} \\
			c[f-\bar{j}]			&= c[f-\bar{j}-1] + (p-1)p^{\bar{j}} > q-1.
		\end{align*}
		Hence the bound $n$ of \ref{subsubsection.Growth-of-variable-c} equals $f-\bar{j}-1$, and $d[f- \bar{j}] = c[f-j] - (q-1) = (\underline{R} +1, \overline{R}-q^{\bar{j}})$, where
		$\underline{R}+1$ and $\overline{R} - q^{\bar{j}}$ are the $(\underline{~})$- and $(\bar{~})$-components of $d[f-\bar{j}]$, respectively. As $d[f-\bar{j}]_{\bar{j}} = R_{\bar{j}}-1 < p-1$, we find from
		$\mathbf{Alg}$ that $b_{\bar{j}-1} = \dots = b_{\bar{s}} = p-1$. Therefore,
		\begin{equation} \label{Eq.Characterizing-equation-for-b-bar}
			\bar{b} = (\underbrace{0,\dots, 0}_{\bar{s}}, \underbrace{(p-1), \dots, (p-1)}_{f-\bar{s}}).
		\end{equation}
		\item We already know  from (j) that $\mathbf{Alg}$ stops after step $2f-\bar{s}$ latest. This means $a_{j} = 0$ for $j < \bar{s}$, that is,
		\begin{equation}
			\underline{a} = 0.
		\end{equation}
		Again from (j), for $j = f-1$ downto $j = \bar{s}$, 
		\[
			p-1 = a_{j} + d[2f-j-1]_{j} = a_{j} + d_{j}
		\] 
		(since steps $f-\bar{s}+1, \dots, f, \dots, 2f-j-1$ that led to the assignments of $b_{\bar{s}-1}, \dots, b_{0}$, $a_{f-1}, \dots, a_{j+1}$
		didn't change the $j$-entry from $d[f-\bar{s}] = D$ to $d[2f-j-1]$), which means 
		\begin{equation}
			\bar{a} + \overline{D} = (\underbrace{0,\dots,0}_{\bar{s}}, \underbrace{(p-1), \dots, (p-1)}_{f-\bar{s}}) = q-p^{\bar{s}}.
		\end{equation}
		\item Writing $a$ and $b$ in their respective components $a = \underline{a} \oplus \bar{a} = (\underline{a}, \bar{a})$, $b = (\underline{b}, \bar{b})$, we have found: 
		$a = (0,q-p^{\bar{s}} - \overline{D})$, $b = (p^{\bar{s}}-1- \underline{D}, q-p^{\overline{s}})$. Now $d[0] = R = (\underline{R}, \overline{R})$, and $\underline{d}$ remains unchanged from $\nu = 0$ to 
		$\nu = n = f-\bar{j}-1$; in the step $\nu = n$ to $n+1 = f-\bar{j}$, 1 is added to $\underline{d}$ (see (m)), which subsequently remains unchanged up to $\nu = f-\bar{s}$. Hence $\underline{D} = d[f-\bar{s}] = \underline{R}+1$.
		Together with \eqref{Eq.Characterizing-equation-for-b-bar-and-D-bar} and \eqref{Eq.Characterizing-equation-for-b-bar}, this gives
		\begin{equation}
			\underline{b} = p^{\overline{s}} -1 - (\underline{R}+1) \quad \text{and} \quad b = q-1-(\underline{R}+1).
		\end{equation}
		Finally, the trivial relation $a+b+R = 2(q-1)$ yields
		\begin{equation}
			a = q-\overline{R} = q-1-(\overline{R}-1).
		\end{equation}
		We have shown \eqref{Eq.Sheats-composition-of-M-(i)}, and \eqref{Eq.Formula-for-mu-H+i-1} is a trivial consequence.
	\end{enumerate}
\end{proof}

\subsection{} We come to the principal result of this section. For ease of orientation, we recall the relevant definitions and requirements. Given $k \in \mathds{N}$, let $R$ be the representative (modulo $q-1$) of $k-1$
in $Q' = \{0,1,\dots, q-2\}$. Write the $q$-adic expansion of $k-1$:
\[
	k-1 = \sum_{0 \leq j \leq \alpha} k_{j}q^{j} \quad \text{with} \quad k_{j} \in Q = \{0,1,\dots,q-1\}, k_{\alpha} \neq 0.
\]
Let $\kappa$ be the $p$-adic number $-k = \kappa = \sum_{j \geq 0} \kappa_{j}q^{j}$ with $\kappa_{j} = q-1-\kappa_{j}$ (so $\kappa_{j}=q-1$ for $j > \alpha$).

Let $M$ be the least natural number such that $M \equiv 0 \pmod{q-1}$, $M <_{p} \kappa$, $M+k \equiv 0 \pmod{q^{\alpha+1}}$. Then $M = \sum_{j \leq \alpha} \kappa_{j}q^{j} + Rq^{\alpha+1}$. Let
$H = \min_{\sigma \in \mathbf{C}} \ell^{\sigma}(M)/(q-1)$ be the height $\mathrm{ht}(M)$ of $M$ and $\bar{s}$ the minimal $s \in F = \{0,1,\dots,f-1\}$ such that 
$\ell^{\sigma_{s}}(M)/(q-1) = H$. Split $R = \sum_{j \in F} R_{j}p^{j}$ ($R_{j} \in P = \{0,1,\dots,p-1\}$) into $R = \underline{R} + \overline{R}$, where $\underline{R} = \sum_{0 \leq j < \bar{s}} R_{j}p^{j}$, 
$\overline{R} = \sum_{\bar{s} \leq j < f} R_{j}p^{j}$. 

\begin{Theorem} \label{Theorem.Vanishing-orders-of-Goss-polynomials-in-Cases}
	\begin{enumerate}[label=$\mathrm{(\roman*)}$]
		\item The vanishing order $\gamma(k)$ of the Goss polynomial $G_{k}(X) = G_{k,A}(X)$ at $x=0$ is given as follows: 
		
		\fbox{$\mathrm{Case~1}$} If $\bar{s} = 0$ or $\underline{R} = p^{\overline{s}}-1$ or $\overline{R} = 0$, then 
		\begin{equation} \label{Eq.Case-1-vanishing-order}
			\gamma(k) = (R+1)q^{\alpha-H+1};
		\end{equation}
		\fbox{$\mathrm{Case~2}$} If $\bar{s} > 0$ and neither $\underline{R} = p^{\overline{s}}-1$ nor $\overline{R} = 0$, then 
		\begin{equation} \label{Eq.Case-2-vanishing-order}
			\gamma(k) = (\overline{R} + q(\underline{R}+1))q^{\alpha -H}.
		\end{equation}
		\item If $\bar{s} = 0$, then the exponent $\alpha-H+1$ equals $[\ell(k-1)/(q-1)]$ ($[\cdot] = \text{Gauß bracket}$). In any case, $(R+1)q^{[\ell(k-1)/q-1]}$ is a lower bound for $\gamma(k)$.
		\item If $k$ is regular (see \ref{Lemma.Characterization-Regularity}), the formulas \eqref{Eq.Case-1-vanishing-order} resp. \eqref{Eq.Case-2-vanishing-order} also hold for $\gamma_{\Lambda}(k)$, where $\Lambda$ is an arbitrary infinite separable 
		lattice (instead of $\Lambda = A$).
	\end{enumerate}
\end{Theorem}

\begin{proof}
	\begin{enumerate}[label=(\roman*), wide]
		\item We use \eqref{Eq.Formula-for-zero-multiplicity}, $\gamma(k) = \lim_{i \to \infty} (k+\mu_{i}(k))/q^{i}$. In Case 1, by \eqref{Eq.Case-1-formula-for-Sh-M-(i)} and \eqref{Eq.Formula-for-M-(i)}, we have for $i \geq 1$:
		\begin{equation} \label{Eq.k+mu-H+i-1}
			k + \mu_{H+i-1}(k) = k + M^{(i-1)} = (R+1)q^{\alpha +i},
		\end{equation}
		whence \eqref{Eq.Case-1-vanishing-order}. In Case 2, similarly for $i \geq 2$:
		\begin{equation} \label{Eq.k+mu-H-i-1-Case2}
			k + \mu_{H+i-1}(k) = k + {}^{1}(M^{(i)}) = (\overline{R} + q(\underline{R}+1))q^{\alpha+i-1},
		\end{equation}
		which gives \eqref{Eq.Case-2-vanishing-order}.
		\item If $\bar{s} = 0$ then $H = \ell(M)/(q-1)$. A brief calculation shows $[\ell(k-1)/(q-1)] = \alpha-H+1$. The trailer ($\gamma(k) \geq (R+1)q^{\alpha-H+1}$) follows by a trivial computation, as $H < \ell(M)/(q-1)$ if $\bar{s}> 0$.
		\item This follows from Theorem \ref{Theorem.Regularity-and-zeroes-of-Ck-Lambda}(iii). \qedhere
	\end{enumerate}
\end{proof}

As to regularity of $k$, we have to following consequence of Proposition \ref{Proposition.Sheats-and-Cases}.

\begin{Corollary} \label{Corollary.Approximation-numbers-in-cases}
	The approximation numbers $\mu_{i}(k)$ of $k$ satisfy $\mu_{i}(k) \varprec \mu_{i+1}(k)$ if $i \geq H-1$ $\mathrm{(Case~1)}$ or $i \geq H$ $\mathrm{(Case~2)}$. Therefore, $k$ is regular if it is $(H-1)$-regular 
	$\mathrm{(Case~1)}$ or $H$-regular $\mathrm{(Case~2)}$.
\end{Corollary}

\begin{proof}
	By \eqref{Eq.Formula2-for-mu-H+i-1}, $\mu_{H+i-1} = {}^{1}(M^{(i)})$ for $i \geq 0$. Suppose we are in Case 1. Then by \eqref{Eq.Case-1-formula-for-Sh-M-(i)}, the sequence $(\mu_{H-1}, \mu_{H}, \dots, \mu_{H+i}, \dots)$ equals 
	$(^{1}M, M, \dots, M^{(i)}, \dots)$, which gives the assertion in this case. Now suppose Case 2. First note that the numbers $\underline{R}, \overline{R}$ then satisfy
	\begin{equation} \label{Eq.p-inequality-over-and-under-parts-of-R}
		\underline{R} + 1 <_{p} \overline{R} - 1,
	\end{equation}
	since $(\overline{R}-1)_{j} = p-1$ for $0 \leq j < \bar{s}$ and $(\underline{R} +1)_{j} = 0$ for $\bar{s} \leq j < f$. The description of $\mu_{H+i-1}$ given in \eqref{Eq.Formula-for-mu-H+i-1} shows that 
	$\mu_{H+i-1} <_{p} \mu_{H+i}$ (and thus $\mu_{H+i-1} \varprec \mu_{H+i}$) for $i \geq 2$. Further, the proof of \ref{Proposition.Sheats-and-Cases} shows that (as an extension of \eqref{Eq.Sheats-composition-of-M-(i)} to $i=1$)
	\begin{equation}
		\mathrm{Sh}(M^{(1)}) = \lambda + (q-1 - (\underline{R}+1))q^{\alpha+1} + Rq^{\alpha+2}
	\end{equation}
	with some $\lambda \in \mathds{N}_{0}$, $\lambda <_{p} \kappa^{*} + \kappa_{\alpha}q^{\alpha}$. This gives 
	\[
		\mu_{H} = {}^{1}(M^{(1)}) = M - \mathrm{Sh}(M^{(1)}) = (\kappa^{*} + \kappa_{\alpha}q^{\alpha} - \lambda) + (\underline{R}+1)q^{\alpha+1},
	\]
	where the first term is $<_{p} \kappa^{*} + \kappa_{\alpha}q^{\alpha}$. This, together with \eqref{Eq.p-inequality-over-and-under-parts-of-R}, assures that $\mu_{H} <_{p} \mu_{H+1}$ (thus $\mu_{H} \varprec \mu_{H+1}$), too. \qedhere
\end{proof}

\begin{Corollary}
	\begin{enumerate}[label=$\mathrm{(\roman*)}$]
		\item Assume we are in $\mathrm{Case~1}$. Then there are no zeroes $z$ of $C_{k, A}$ with $\log z > H-1$. Accordingly, there are no zeroes $x$ of $G_{k,A}$ with $\log x < L(H-1)$.
		\item In $\mathrm{Case~2}$, there are no zeroes $z$ of $C_{k,A}$ with $\log z > H$, and no zeroes $x$ of $G_{k,A}$ with $\log x < L(H)$.
		\item If $k$ is regular, the assertions from $\mathrm{(i)}$ and $\mathrm{(ii)}$ hold for $A$ replaced by an arbitrary infinite separable $\mathds{F}$-lattice $\Lambda$, where the bound $H-1$ (resp. $H$) must be 
		replaced with $r_{H-1}$ (resp. $r_{H}$) as in Section \ref{Section.Location-of-zeroes-regular-case}.
	\end{enumerate}
\end{Corollary}

\begin{proof}
	\begin{enumerate}[label=(\roman*), wide]
		\item By \eqref{Eq.k+mu-H+i-1} and \eqref{Eq.Formula-for-gamma-(i)}, there are no zeroes $z$ with $\log z \geq H$. Since $\mu_{H-1} \varprec \mu_{H}$, also $\tilde{\gamma}_{H-1,A} = \tilde{\gamma}_{A}^{(H-1)}$ by 
		\ref{Corollary.Existence-of-irregular-zeroes-and-approximation-numbers}, which excludes $z$ with $\log z > H-1$.
		\item The argument is similar: no $z$ with $\log z \geq H+1$ by \eqref{Eq.k+mu-H-i-1-Case2}, and no $z$ with $H < \log z < H+1$ by $\mu_{H} \varprec \mu_{H+1}$.
		\item This is obvious, as for regular $k$ the assertions depend only on the $\mu_{i}(k)$, but not on $\Lambda$. \qedhere
	\end{enumerate}
\end{proof}

\begin{Remarks}
	\begin{enumerate}[label=(\roman*), wide]
		\item Although we have identified $\mu_{H-1}(k) ={} ^{1}M$ and $\mu_{H}(k) ={} ^{1}(M^{(1)})$, it is not clear whether $\mu_{H-1}(k) \varprec \mu_{H}(k)$ also holds in Case 2. Therefore, the results about 
		zero-free regions of $C_{k}$ and $G_{k}$ are weaker in Case 2.
		\item While in Case 1 there are plenty of examples of zeroes $z$ of $C_{k,A}$ with $\log z = H-1$, we lack so far of examples in Case 2 that the bound $H$ may be attained by $\log z$.
	\end{enumerate}
\end{Remarks}

\section{Complements and examples} \label{Section.Complements-and-examples}

\subsection{} So far, we were mainly concerned with infinite separable lattices $\Lambda$. Suppose now that $\Lambda'$ is finite, separable, of finite $\mathds{F}$-dimension $t+1$ ($t \in \mathds{N}_{0}$), and embedded into an
infinite separable $\Lambda$ such that $\Lambda' = \Lambda_{t+1}$. So $\{\lambda_{0} = 1, \lambda_{1}, \dots \}$ is a separating $\mathds{F}$-basis of $\Lambda$ the first $t+1$ elements of which form an $\mathds{F}$-basis
of $\Lambda'$.

Now, as \ref{Corollary.Coefficient-condition-Laurent-expansion-of-Ck} also holds for the finite lattice $\Lambda'$, and the distribution of zeroes of $C_{k, \Lambda}$ and $G_{k, \Lambda}$ with $\log z < r_{t+1}$ is governed by the power sums
\[
	S_{i+1, \Lambda}(n) = \sum_{\lambda \in \Lambda_{i+1}} \lambda^{n} \qquad (i \leq t)
\]
which only involve $\Lambda_{t+1} = \Lambda'$, the results of Sections \ref{Section.Location-of-zeroes-regular-case} and \ref{Section.Irregular-zeroes} also hold for $\Lambda'$, as long as only zeroes $z$ of $C_{k, \Lambda'}$ with $\log z < r_{t+1}$ are concerned. We formulate the outcome for the two 
cases with $\Lambda' = A_{t+1} = \{ a \in A \mid \deg a \leq t \}$ and $\Lambda' = \Lambda_{t+1} = \sum_{0 \leq i \leq t} \mathds{F}\lambda_{i}$ arbitrary with $\lambda_{0} = 1 < \lvert \lambda_{1} \rvert = q^{r_{1}} < \dots$, where in the second 
case we restrict to $k$ being $(t+1)$-regular.

\begin{Corollary}[to Theorems \ref{Theorem.Regularity-and-zeroes-of-Ck-Lambda} and \ref{Theorem.Irregular-zeroes-Ck}] \label{Corollary.Zeroes-of-Ck-on-finite-lattices}
	Assume $\Lambda' = A_{t+1}$.
	\begin{enumerate}[label=$\mathrm{(\roman*)}$]
		\item All the zeroes $z$ of $C_{k, A_{t+1}}$ satisfy $\log z < t+1$.
		\item Their numbers $\tilde{\gamma}_{i}(k)$, $\tilde{\gamma}^{(i)}(k)$, $\tilde{\gamma}_{r}$ ($i \leq t, r \in \mathds{Q}_{<t+1}$; we omit reference to $A_{t+1}$ in the symbols) and the corresponding numbers
		$\gamma_{i}(k)$, $\gamma^{(i)}(k), \gamma_{r}(k)$ for the polynomials $G_{k, A_{t+1}}$ agree with those for the lattice $A$ given in \eqref{Eq.Formula-for-the-gamma-j-k} and \ref{Subsection.Equations-for-gamma-i}.
		\item We have $\gamma_{A_{t+1}}(k) = k - \sum_{0 \leq i \leq t} \gamma^{(i)}(k) \geq \gamma_{A}(k)$.
	\end{enumerate}
\end{Corollary}

\begin{proof}
	\begin{enumerate}[label=(\roman*)]
		\item As the proof of Lemma \ref{Lemma.vartheta-k-i-s-and-mu-i+1-k} shows, the term $a_{-k-\mu_{t+1}(k)}$ dominates in the Laurent expansion of $C_{k, A_{t+1}}$ on the sphere $\mathbf{S}(q^{r})$, if $r \geq t+1$ (actually, if $r > t+1-\varepsilon$ for some small
		$\varepsilon$); whence there are no zeroes $z$ with $\log z \geq t+1$.
	\end{enumerate}
	Items (ii) and (iii) follow from the foresaid. \qedhere
\end{proof}

\begin{Corollary}
	Assume that $k$ is $(t+1)$-regular and $\Lambda' = \sum_{0 \leq i \leq t} \mathds{F}\lambda_{i}$ is an arbitrary $(t+1)$-dimensional separable $\mathds{F}$-lattice with $\lvert \lambda_{i} \rvert = q^{r_{i}}$.
	\begin{enumerate}[label=$\mathrm{(\roman*)}$]
		\item All the zeroes $z$ of $C_{k, \Lambda'}$ satisfy $\log z \leq r_{t}$, and lie on critical spheres $\mathbf{S}(q^{r_{i}})$ with $0 \leq i \leq t$.
		\item The numbers $\tilde{\gamma}_{i}$ and the corresponding $\gamma_{i}$ agree with those of $A$, given in Theorem \ref{Theorem.Regularity-and-zeroes-of-Ck-Lambda}.
		\item We have $\gamma_{\Lambda'}(k) = k - \sum_{0 \leq i \leq t} \gamma_{i}(k) \geq \gamma_{A}(k)$.
	\end{enumerate}
\end{Corollary}

\begin{proof}
	As in \ref{Corollary.Zeroes-of-Ck-on-finite-lattices}, the term $a_{-k-\mu_{t+1}(k)}$ dominates on $\mathbf{S}(q^{r})$ for $r > r_{t}$ (where the larger domain of domination compared to \ref{Corollary.Zeroes-of-Ck-on-finite-lattices} 
	comes from the additional $(t+1)$-regularity assumption). The rest is now obvious.
\end{proof}

\subsection{} Now we present a handful of examples that exhibit phenomena like irregular zeroes and exceptional zero numbers $\gamma(k)$ due to Case 2 in Theorem \ref{Theorem.Vanishing-orders-of-Goss-polynomials-in-Cases}. 
In what follows, we assume first that \fbox{$q=4$}. The relevant numbers $k-1$, $\kappa = -k$, $\mu_{i}(k)$, $M(k)$ are displayed as the $0$-$1$-strings (infinite for $\kappa$) of their dyadic expansions,
\[
	n = \sum_{j \geq 0} n_{j}q^{j}, \qquad n_{j} = n_{j,0} + 2n_{j,1}, \qquad n_{i,j} \in P = \{0,1\};
\]
empty spaces correspond to zeroes. 

\begin{Example}[$q=4$, $k=21$] \label{Example.Regular-k-Case-2} ~
	\begin{center}
		\begin{tabular}{l|c|c|c|c|c} \toprule
			$n$\textbackslash$j$	& 0 		& 1 			& $\alpha=2$	& 3 		& 4 			\\ \midrule
			$k-1 = 20$							& 0\,0	& 1\ 0		& 1\,0					&			&				\\
			$\kappa$							& 1\,1	& 0\,1		& 0\,1					& 1\,1	& 1\,1 		\\
			$M = 171$							& 1\,1	& 0\,1		& 0\,1					& 0\,1	&				\\
			$\mu_{1} = 3$					& 1\,1	&				&							&			&\\
			$\mu_{2} = 75$					& 1\,1	& 0\,1		& 0\,0					& 1\,0	& \\
			$\mu_{3} = 363$				& 1\,1	& 0\,1		& 0\,1					& 1\,0	& 1\,0 \\				
			\bottomrule
		\end{tabular}
	\end{center}
	Here $\alpha = 2$, $\kappa_{\alpha} = 2$, $R=2$, $H=2$, $\bar{s} = 1$, $\overline{R} = 2$, $\underline{R} = 0$. $k$ is regular as $\mu_{1} \varprec \mu_{2}$, and Case 2 of Theorem 
	\ref{Theorem.Vanishing-orders-of-Goss-polynomials-in-Cases} holds. We have 
	\begin{equation}
		\gamma_{0}(k) = 15, \qquad \gamma_{1}(k) = 0, \qquad \gamma(k) = 6.
	\end{equation}
	Hence $G_{21, \Lambda}(X) = P(X) X^{6}$, where $P$ is a monic polynomial of degree 15, $\lvert \text{coefficients} \rvert \leq 1$ and $\lvert P(0) \rvert = 1$. In fact, one may calculate $P(X)$ by hand (and with some pain),
	which gives 
	\begin{equation}
		P(X) = X^{15} + \alpha_{1} X^{12} + \alpha_{1}^{2} X^{9} + \alpha_{1}^{3} X^{6} + (\alpha_{1}^{5} + \alpha_{2}).
	\end{equation}
	Here the $\alpha_{i} = \alpha_{i}(\Lambda)$ are the coefficients of $e_{\Lambda}$ (see \ref{Eq.Series-expansion-of-exponential-function}), $\lvert \alpha_{2} \rvert < \lvert \alpha_{1} \rvert = 1$. As $k=21$ is regular, the lattice
	$\Lambda$ may be arbitrary separable with $\lambda_{0} = 1$.
\end{Example}

\begin{Example}[$q=4$, $k = 69$] \label{Example.Regular-k-Case-2-more-complicated} ~
	\begin{center}
		\begin{tabular}{l|c|c|c|c|c|c} \toprule
				$n$\textbackslash$j$		& 0 		& 1 		& 2 		& $\alpha = 3$ & 4 			& 5  \\ \midrule
				$k-1 = 68$							& 0\,0	& 1\,0	& 0\,0	& 1\,0					& 				& 		\\
				$\kappa$							& 1\,1	& 0\,1	& 1\,1	& 0\,1					& 1\,1		& 1\,1 \\
				$M = 699$							& 1\,1	& 0\,1	& 1\,1	& 0\,1					& 0\,1		& \\
				$\mu_{1} = 3$					& 1\,1	&			&			&							&				& \\
				$\mu_{2} = 27$					& 1\,1	& 0\,1	& 1\,0	&							&				& \\
				$\mu_{3} = 315$				& 1\,1	& 0\,1	& 1\,1	& 0\,0					& 1\,0		& \\
				$\mu_{4} = 1467$				& 1\,1	& 0\,1	& 1\,1	& 0\,1					& 1\,0		& 1\,0 \\ \bottomrule
		\end{tabular}
	\end{center}
	Here $\alpha=3$, $k_{\alpha} = 2$, $R=2$, $H=3$, $\bar{s} = 1$, $\overline{R} = 2$, $\underline{R} = 0$. $k$ is regular as $\mu_{1} \varprec \mu_{2} \varprec \mu_{3}$, and Case 2 holds. We have
	\begin{equation}
		\gamma_{0}(k)= 51,\qquad \gamma_{1}(k) = 12, \qquad \gamma_{2}(k) = 0, \qquad \gamma(k) = 6.
	\end{equation}
	The Newton polygon of $G_{69, A}$ has the break points $(6,48)$, $(18,0)$, $(69,0)$. (Here $48 = -12 \times L(1)$ with the function $L$ of \ref{Proposition.Map-L-IQ-geq-to-IQ-leq}.) For arbitrary $\Lambda \neq A$ with $\lambda_{0}=1$, the Newton polygon
	has the same shape, but the ordinate of the first break point will differ.
\end{Example}

\pagebreak

\begin{Example}[$q=4$, $k=75$] \label{Example.Irregular-k-Case-2} ~
	\begin{center}
		\begin{tabular}{l|c|c|c|c|c} \toprule
			$n$\textbackslash$j$		& 0			&1 		& 2 		& $\alpha=3$ 	& 4		 	\\ \midrule
			$k-1 = 74$							& 0\,1		& 0\,1	& 0\,0	& 1\,0					& 				\\ 
			$\kappa$							& 1\,0		& 1\,0	& 1\,1	& 0\,1					& 1\,1 		\\
			$M = \mu_{3} = 693$		& 1\,0		& 1\,0	& 1\,1	& 0\,1					& 0\,1		\\
			$\mu_{1} = 21$					& 1\,0		& 1\,0	& 1\,0	&							&				\\
			$\mu_{2} = 165$				& 1\,0		& 1\,0	& 0\,1	& 0\,1					&				\\
			$\vartheta = 309$			& 1\,0		& 1\,0	& 1\,1	& 0\,0					& 1\,0 \\ \bottomrule
		\end{tabular}
	\end{center}
	Here $\alpha=3$, $\kappa_{\alpha} = 2$, $R = 2$, $\bar{s} = 0$, Case 1 of Theorem \ref{Theorem.Vanishing-orders-of-Goss-polynomials-in-Cases}. As $\mu_{1} \not\varprec \mu_{2}$, $k$ is \textbf{irregular}. The set 
	$\boldsymbol{\Theta}(k,1)$ of \ref{Subsection.Properties-vartheta-k-i-s} is $\boldsymbol{\Theta}(k,1) = \{ \mu_{2}, \vartheta\}$ with $\vartheta = 309$, $\mathbf{Sh}(\vartheta) = (21, 288)$, $\mathbf{Sh}(\mu_{2}) = (33,132)$. Then 
	$\mathrm{ct}_{2}(\vartheta) = 21 < 33 = \mathrm{ct}_{2}(\mu_{2})$, and so
	\begin{equation}
		\rho = 1 + \frac{\mathrm{ct}(\mu_{2}) - \mathrm{ct}_{2}(\vartheta)}{\vartheta - \mu_{2}} = \frac{13}{12}
	\end{equation}
	equals $\log z$ for some irregular zeroes $z$ of $C_{75, A}$.
	
	Our formulas yield $\gamma_{0}(k) = 51$, $\gamma_{1}(k) = 0$, $\gamma^{(1)}(k) = \gamma_{\rho}(k) = 9$, $\gamma_{2}(k) = 3$, $\gamma(k) = 12$. To the values $\log z = 0, 13/12, 2$ correspond
	$L(\log z) = 0, -16/3, -20$ for the $\log x$, where $x = t_{A}(z)$ is the corresponding zero of $G_{75, A}$. There result the break points $(12, 108)$, $(15,48)$, $(24,0)$ and $(75,0)$ of
	$\mathrm{NP}(G_{75,A}(X))$.
\end{Example}

Next we generalize the preceding example and produce irregular $k$ for arbitrar non-prime $q = p^{f}$. For ease of presentation, we restrict to $p=2$, but one easily sees how the construction generalizes to arbitrary $p$.

\begin{Example}
	Let $q = 2^{f}$ with $f > 2$. We first write down an $M = M(k)$ with a very specific $q$-expansion, from which we derive the ($q$-expansions of the) relevant numbers $k$, $\mu_{1}$, $\mu_{2}$, 
	$\mathrm{Sh}(\mu_{2}) = \mathrm{Sh}_{2}(\mu_{2})$, $^{1}\mu_{2} = \mathrm{Sh}_{1}(\mu_{2})$, $\mu_{1}^{(1)}$. This will show that $\mu_{1} \not\varprec \mu_{2}$.
	
	In the following table, the $(n,j)$-entry is $n_{j}$, where $n$ has $q$-expansion $\sum n_{j}q^{j}$. As usual, empty spaces denote zeroes. We assume that $f>2$ since for $f=2$ several columns overlap, which changes some formulas
	and gives Example \ref{Example.Irregular-k-Case-2}.
	
	\begin{center}
		\begin{tabular}{l|c|c|c|c|c|c|c|c|c|c|c} \toprule
			$n$\textbackslash$j$			& 0			& \dots	& $q-4$ & $q-3$ & $q-2$ & $q-1$ 	& $q$ 		& \dots	& $\alpha = 2q-5$	& $2q-4$ 	& \dots \\ \midrule
			$k-1$										& $q-2$	& \dots	& \dots	& $q-2$	& 0			& 1			& \dots	& \dots	& 1								&					& 				\\
			$\kappa$								& 1			& \dots	& \dots	& 1			& $q-1$	& $q-2$	& \dots	& \dots	& $q-2$						& $q-1$		& \dots	\\
			$M$										& 1			& \dots	& \dots	& 1			& $q-1$	& $q-2$	& \dots	& \dots	& $q-2$						& $q-2$		& 				\\
			$\mu_{1}$								& 1			& \dots	& \dots	& 1			& 1			&				&				&				&									&					&				\\
			$\mu_{2}$								& 1			& \dots	& \dots	& 1			& $q-2$	& 2			&				&				&									&					&				\\
			$\mathrm{Sh}(\mu_{2})$	&				&				&				& 1			& $q-4$	& 2			&				&				&									&					&				\\
			${}^{1}\mu_{2}$						& 1			& \dots	& 1			& 0			& 2			& 				&				&				&									&					&				\\
			$\mu_{1}^{(1)}$						& 1			& \dots	& \dots	& 1			& $q-1$	& $q-2$	& 2			&				&									&					&				\\		
		\end{tabular}
	\end{center}
	
	Start with $M = \sum_{0 \leq j \leq q-3} q^{j} + (q-1)q^{q-2} + \sum_{q-1 \leq j \leq 2q-4} (q-2)q^{j}$, which gives the $M$-th row in the table, as well as $k-1$ and $\kappa = -k$. 
	(As $M_{2q-4} < q-1$, $\alpha = \deg_{q}(k-1) = 2q-5$, and the coefficients of $k-1$ are complementary to those of $\kappa$.) Summing up the geometric series gives
	\begin{equation}
		\begin{split}
			k-1	&= (q-2)(q^{q-2}-1)/(q-1) + q^{q-1}(q^{q-3}-1)/(q-1) \quad \text{and} \\
			k		&= (q^{q-2}-1)^{2}/(q-1).
		\end{split}
	\end{equation}
	In particular, the quantity $R = \text{representative of $k-1$ in $Q'$}$ equals $q-2$. The row for $\mu_{1}$ is then obvious. Since $\ell( \sum_{0 \leq j \leq q-2} \kappa_{j}q^{j}) = 2q-3 < 2(q-1)$, $\mu_{2}$ must have some entry
	$0 \neq (\mu_{2})_{j}$ with $j \geq q-1$, which must be $\geq 2$, since $1 \lneq_{p} q-2$. Therefore, $\mu_{2}$ as given is the least possible value. Again, the Sheats factors $\mathrm{Sh}_{2}(\mu_{2}) = \mathrm{Sh}(\mu_{2})$
	and $\mathrm{Sh}_{1}(\mu_{2}) = {}^{1}\mu_{2}$ are then obvious. Also, it is not difficult to show that $\mu_{1}^{(1)}$ is as indicated. We find that $\mu_{1} \not\varprec \mu_{2}$ as intended, so $k$ is irregular.
	
	A view to the table shows that $\ell_{i}(M) = q-1$ for $0 \leq i < f$, so $\ell^{\sigma}(M) = (q-1) \sum_{0 \leq i < f} 2^{i} = (q-1)^{2}$ for each $\sigma \in \mathbf{C}$, and therefore
	\begin{equation}
		H = \mathrm{ht}(M) = q-1.
	\end{equation}
	In particular, Case 1 of Theorem \ref{Theorem.Vanishing-orders-of-Goss-polynomials-in-Cases} holds, which gives 
	\begin{equation}
		\gamma(k) = (R+1) q^{[\ell(k-1)/(q-1)]} = (q-1)q^{q-3}.
	\end{equation}
	Some admissible $n$ with $\mathrm{ht}(n) = 1$ and $n < {}^{1}\mu_{2}$ must satisfy $n_{q-2} < 2$, and must agree with $\mu_{1}$. Therefore, with the considerations of Section \ref{Section.Irregular-zeroes}, we see that the set $\mathbf{R} = \mathbf{R}(k,1)$
	contains exactly one element $\rho$ which appears as $\rho = \log z$ for an irregular zero $z$ of $C_{k,A}$. It is given by
	\begin{equation}
		\rho = 1 + \frac{ {}^{1}\mu_{2} - \mu_{1}}{\mu_{1}^{(1)}-\mu_{2}} = 1 + \frac{1}{q(3q-1)},
	\end{equation}
	where the last equality comes from calculation.
\end{Example}

\begin{Concluding_Remarks}
	As we just saw, irregular numbers $k$ exists for arbitrary prime powers $q = p^{f}$ with $f \geq 2$, but by construction, these numbers are large. Therefore, we expect the share of irregular $k$ to be small; more explicitly, that for 
	given $q$,
	\[
		\limsup_{N \to \infty} \frac{ \#\{ k \leq N \mid k \text{ is irregular}\} }{N} \leq c(q)
	\]
	with a constant $c(q) < 1$ that tends to 0 for $q \to \infty$. Finding good estimates for $c(q)$ is an interesting challenge. Similarly, the question of relative frequencies of the Cases 1 and 2 in Theorem 
	\ref{Theorem.Vanishing-orders-of-Goss-polynomials-in-Cases} needs further investigation.
	Our tiny database of the cases $(q,k)$ with $q=4$ and $k < 100$ odd (we may restrict to $p \nmid k$ in view of the relation $G_{kp} = (G_{k})^{p}$) gives that
	\begin{itemize}
		\item $k$ is irregular only if $k = 75$, see Example \ref{Example.Irregular-k-Case-2};
		\item Case 2 holds if and only if $k \in \{21, 69, 81, 87, 93\}$, see Examples \ref{Example.Regular-k-Case-2} and \ref{Example.Regular-k-Case-2-more-complicated}.
	\end{itemize}
	But this is certainly not impressive enough to draw any conclusion \dots
\end{Concluding_Remarks}


\begin{bibdiv}
	\begin{biblist}
		\bib{BosserPellarin09}{article}{
			  author={Bosser, Vincent},
			  author={Pellarin, Federico},
			  title={On certain families of Drinfeld quasi-modular forms},
			  journal={J. Number Theory},
			  volume={129},
			  date={2009},
			  number={12},
			  pages={2952--2990},
			  issn={0022-314X},
			  review={\MR{2560846}},
			  doi={10.1016/j.jnt.2009.04.014},
		}
		\bib{Carlitz48}{article}{
			  author={Carlitz, L.},
			  title={Finite sums and interpolation formulas over $GF[p^n,x]$},
			  journal={Duke Math. J.},
			  volume={15},
			  date={1948},
			  pages={1001--1012},
			  issn={0012-7094},
			  review={\MR{27289}},
		}
		\bib{Gekeler88}{article}{
			  author={Gekeler, Ernst-Ulrich},
			  title={On the coefficients of Drinfel\cprime d modular forms},
			  journal={Invent. Math.},
			  volume={93},
			  date={1988},
			  number={3},
			  pages={667--700},
			  issn={0020-9910},
			  review={\MR{952287}},
			  doi={10.1007/BF01410204},
		}
		\bib{Gekeler12}{article}{
			  author={Gekeler, Ernst-Ulrich},
			  title={Zeroes of Eisenstein series for principal congruence subgroups
			  over rational function fields},
			  journal={J. Number Theory},
			  volume={132},
			  date={2012},
			  number={1},
			  pages={127--143},
			  issn={0022-314X},
			  review={\MR{2843303}},
			  doi={10.1016/j.jnt.2011.07.008},
		}
		\bib{Gekeler13}{article}{
			  author={Gekeler, Ernst-Ulrich},
			  author={Stopp, Philipp},
			  title={On the zeroes of certain periodic functions over valued fields of
			  positive characteristic},
			  journal={J. Number Theory},
			  volume={133},
			  date={2013},
			  number={3},
			  pages={940--954},
			  issn={0022-314X},
			  review={\MR{2997777}},
			  doi={10.1016/j.jnt.2012.02.006},
		}
		\bib{Gekeler13-2}{article}{
			  author={Gekeler, Ernst-Ulrich},
			  title={On the zeroes of Goss polynomials},
			  journal={Trans. Amer. Math. Soc.},
			  volume={365},
			  date={2013},
			  number={3},
			  pages={1669--1685},
			  issn={0002-9947},
			  review={\MR{3003278}},
			  doi={10.1090/S0002-9947-2012-05699-6},
		}
		\bib{GerritzenvdPut80}{book}{
			  author={Gerritzen, Lothar},
			  author={van der Put, Marius},
			  title={Schottky groups and Mumford curves},
			  series={Lecture Notes in Mathematics},
			  volume={817},
			  publisher={Springer, Berlin},
			  date={1980},
			  pages={viii+317},
			  isbn={3-540-10229-9},
			  review={\MR{590243}},
		}
		\bib{Goss80}{article}{
			  author={Goss, David},
			  title={The algebraist's upper half-plane},
			  journal={Bull. Amer. Math. Soc. (N.S.)},
			  volume={2},
			  date={1980},
			  number={3},
			  pages={391--415},
			  issn={0273-0979},
			  review={\MR{561525}},
			  doi={10.1090/S0273-0979-1980-14751-5},
		}
		\bib{Goss80-2}{article}{
			  author={Goss, David},
			  title={$\pi $-adic Eisenstein series for function fields},
			  journal={Compositio Math.},
			  volume={41},
			  date={1980},
			  number={1},
			  pages={3--38},
			  issn={0010-437X},
			  review={\MR{578049}},
		}
		\bib{Goss83}{article}{
			  author={Goss, David},
			  title={On a new type of $L$-function for algebraic curves over finite
			  fields},
			  journal={Pacific J. Math.},
			  volume={105},
			  date={1983},
			  number={1},
			  pages={143--181},
			  issn={0030-8730},
			  review={\MR{688411}},
		}
		\bib{Goss96}{book}{
			  author={Goss, David},
			  title={Basic structures of function field arithmetic},
			  series={Ergebnisse der Mathematik und ihrer Grenzgebiete (3) [Results in
			  Mathematics and Related Areas (3)]},
			  volume={35},
			  publisher={Springer-Verlag, Berlin},
			  date={1996},
			  pages={xiv+422},
			  isbn={3-540-61087-1},
			  review={\MR{1423131}},
			  doi={10.1007/978-3-642-61480-4},
		}
		\bib{Goss14}{article}{
			  author={Goss, David},
			  title={A construction of $\germ{v}$-adic modular forms},
			  journal={J. Number Theory},
			  volume={136},
			  date={2014},
			  pages={330--338},
			  issn={0022-314X},
			  review={\MR{3145337}},
			  doi={10.1016/j.jnt.2013.10.010},
		}
		\bib{Neukirch99}{book}{
			  author={Neukirch, J\"{u}rgen},
			  title={Algebraic number theory},
			  series={Grundlehren der Mathematischen Wissenschaften [Fundamental
			  Principles of Mathematical Sciences]},
			  volume={322},
			  note={Translated from the 1992 German original and with a note by Norbert
			  Schappacher;
			  With a foreword by G. Harder},
			  publisher={Springer-Verlag, Berlin},
			  date={1999},
			  pages={xviii+571},
			  isbn={3-540-65399-6},
			  review={\MR{1697859}},
			  doi={10.1007/978-3-662-03983-0},
		}
		\bib{PapanikolasZeng17}{article}{
			  author={Papanikolas, Matthew A.},
			  author={Zeng, Guchao},
			  title={Theta operators, Goss polynomials, and $v$-adic modular forms},
			  language={English, with English and French summaries},
			  journal={J. Th\'{e}or. Nombres Bordeaux},
			  volume={29},
			  date={2017},
			  number={3},
			  pages={729--753},
			  issn={1246-7405},
			  review={\MR{3745247}},
		}
		\bib{Pellarin14}{article}{
			  author={Pellarin, Federico},
			  title={Estimating the order of vanishing at infinity of Drinfeld
			  quasi-modular forms},
			  journal={J. Reine Angew. Math.},
			  volume={687},
			  date={2014},
			  pages={1--42},
			  issn={0075-4102},
			  review={\MR{3176606}},
			  doi={10.1515/crelle-2012-0048},
		}
		\bib{Sheats98}{article}{
			  author={Sheats, Jeffrey T.},
			  title={The Riemann hypothesis for the Goss zeta function for $\mathbf{F}_q[T]$},
			  journal={J. Number Theory},
			  volume={71},
			  date={1998},
			  number={1},
			  pages={121--157},
			  issn={0022-314X},
			  review={\MR{1630979}},
			  doi={10.1006/jnth.1998.2232},
		}
		\bib{Thakur04}{book}{
			  author={Thakur, Dinesh S.},
			  title={Function field arithmetic},
			  publisher={World Scientific Publishing Co., Inc., River Edge, NJ},
			  date={2004},
			  pages={xvi+388},
			  isbn={981-238-839-7},
			  review={\MR{2091265}},
			  doi={10.1142/9789812562388},
		}
		\bib{Thakur15}{article}{
			  author={Thakur, Dinesh S.},
			  title={Power sums of polynomials over finite fields and applications: a
			  survey},
			  journal={Finite Fields Appl.},
			  volume={32},
			  date={2015},
			  pages={171--191},
			  issn={1071-5797},
			  review={\MR{3293409}},
			  doi={10.1016/j.ffa.2014.08.004},
		}
	\end{biblist}
\end{bibdiv}
\end{document}